\documentclass[a4paper, 11pt]{article}
 
\usepackage[top=2cm, bottom=2cm, left=2cm, right=2cm]{geometry}
\usepackage{graphicx, array, amsmath, amsthm, mathrsfs, framed, stmaryrd, marvosym, bbm, amssymb, mathtools, verbatim, tikz}

\usepackage[all]{xy}
\tikzset{node distance=2cm, auto}
\usepackage{enumerate}
\usepackage{color}
\usepackage{youngtab}

\allowdisplaybreaks

\input xy
\xyoption{matrix}\xyoption{arrow}
\def\edge{\ar@{-}}
\def\dedge{\ar@{.}}

\newtheorem{theorem}{Theorem}[section]

\newtheorem{proposition}[theorem]{Proposition}

\newtheorem{lemma}[theorem]{Lemma}
\newtheorem{corollary}[theorem]{Corollary}

\theoremstyle{definition}

\newtheorem{remark}[theorem]{Remark}

\newtheorem{notation}[theorem]{Notation}

\def\k{{\mathbb K}}

\def\spec{{\rm Spec}}
\def\hspec{\ch\!-\!\spec} 

\def\ch{{\mathcal H}}

\def\cR{{\mathcal R}}

\newcommand{\N}{\mathbb{N}}

\newcommand{\zbar}{Z}

\newcommand{\height}{{\rm ht}}
\newcommand{\gkdim}{{\rm GKdim}}
\newcommand{\der}{{\rm Der}}

\def\uq{U_q^+(G_2)}

\usepackage{version}

\includeversion{full} 

\title{Derivations of a family of quantum second Weyl algebras} 
\author{S Launois and~ I Oppong}

\begin{document} 
\maketitle
\begin{abstract} In view of a well-known theorem of Dixmier, its is natural to consider primitive quotients of $U_q^+(\mathfrak{g})$ as quantum analogues of Weyl algebras. In this work, we study primitive quotients of $\uq$ and compute their Lie algebra of derivations. 
\end{abstract}



\section{Introduction} 

Weyl algebras have been extensively studied in the last 60 years due to their link to Lie theory, differential operators, quantum mechanics, etc. One of the main questions remaining is the famous Dixmier Conjecture that asserts that every endomorphism of a complex Weyl algebra is an automorphism. 

Let $\k$ be a field and  $q$ be a non-zero element of $\k$ that is not a root of unity. The aim of this article is to produce quantum analogues of the second Weyl algebra and to compare their properties to those of the second Weyl algebra. There exist in the literature various families of ``quantum Weyl algebras'', e.g. the so-called quantum Weyl algebras and generalised Weyl algebras (GWA for short). Most of the time, they are obtained by generators and relations through a deformation of the classical defining relation of the first Weyl algebra: $xy-yx=1$. 

To produce potential quantisations, we take a different approach in this article. Our inspiration comes from a Theorem of Dixmier (see, for instance, \cite[Th\'eor\`eme 4.7.9]{dix}) that asserts that primitive quotients of enveloping algebras of complex nilpotent Lie algebras are isomorphic to Weyl algebras.  

We have at hand a quantum analogue of at least some enveloping algebras of complex nilpotent Lie algebras, namely the positive part $U_q^+(\mathfrak{g})$ of a quantised enveloping algebra $U_q^+(\mathfrak{g})$ of a complex simple Lie algebra $\mathfrak{g}$. As a consequence, it is natural to consider primitive quotients of $U_q^+(\mathfrak{g})$ as quantum analogues of Weyl algebras. In the $A_2$ and $B_2$ cases, primitive ideals of $U_q^+(\mathfrak{g})$ have been classified and it turns out that in the $B_2$ case, some of the resulting primitive quotients provide `nice' quantum analogues of the first Weyl algebra. For instance, they are simple---this is not the case of quantum Weyl algebras---and do not possess non-trivial units---this is not the case of a quantum GWA over a Laurent polynomial ring. (See \cite{sl} for details.) 

The present article is concerned with the $G_2$ case. More precisely, we identify a family of primitive ideals of $\uq$ and then proceed in proving that the corresponding primitive quotients have (at least for some choice of the parameters) properties similar to those of the second Weyl algebra. More precisely, the center of $\uq$ is a polynomial algebra in two variables $\k[\Omega_1,\Omega_2]$ and we prove that the quotient algebra 
$$A_{\alpha,\beta}:= \uq / \langle \Omega_1 - \alpha, \Omega_2 - \beta\rangle$$
is simple for all $(\alpha , \beta)\neq (0,0)$. We then proceed and study these quotient algebras. In particular, we show that $A_{\alpha,\beta}$ has the same (Gelfand-Kirillov) dimension as the second Weyl algebra $A_2(\k)$. We also establish that for certain choice of the parameters $\alpha$ and $\beta$, the algebra $A_{\alpha,\beta}$ is a deformation of a quadratic extension of $A_2(\k)$ at $q=1$.  

In the final section, we compute the derivations of $A_{\alpha,\beta}$. Our results show that when $\alpha$ and $\beta$ are both non-zero, all derivations of $A_{\alpha,\beta}$ are inner, a property that is well known to hold in $A_2(\k)$. 

In view of the celebrated Dixmier Conjecture, it would be interesting to describe automorphisms and endomorphisms of $A_{\alpha,\beta}$ when $\alpha$ and $\beta$ are both non-zero. We intend to come back to these questions in the future. 

This article is organized as follows. In Section 2, we recall the presentation of $\uq$ as a so-called quantum nilpotent algebra (QNA for short). This allows the use of two different tools to study the prime and primitive spectra of $\uq$: the $\ch$-stratification theory of Goodearl and Letzter, and the Deleting Derivation Theory of Cauchon. We recall both theories in the context of $\uq$ in Section 2. In Section 3, we use these two theories to establish that $\langle \Omega_1 - \alpha, \Omega_2 - \beta\rangle$ is a maximal ideal of $\uq$ when $(\alpha, \beta) \neq (0,0)$. 

In Section 4, we focus on comparing $A_{\alpha,\beta}$ with the second Weyl algebra $A_2(\k)$. In particular, we show that both have Gelfand-Kirillov dimension equal to 4. Through a direct computation, we also establish that $A_{1,\frac{1}{9(q^6-1)}}$ is a quadratic extension of $A_2(\k)$ at $q=1$. In this section we also compute a linear basis for $A_{\alpha,\beta}$. 

In the final section, we compute the derivations of $A_{\alpha,\beta}$. Our strategy here is to make use of the following tower of algebras arising from the Deleting Derivation Algorithm:
$$A_{\alpha,\beta}\subset \cR_6=\cR_7\Sigma_6^{-1}\subset \cR_5=\cR_6\Sigma_5^{-1}\subset \cR_4=\cR_5\Sigma_4^{-1} \subset \cR_3.$$
The later algebra $\cR_3$ is a simple quantum torus whose derivations have been described by Osborn and Passman in \cite{op}. We pull back their description to obtain a description of the derivations of $A_{\alpha,\beta}$ through a step-by-step process consisting in ``reverting'' the Deleting Derivation Algorithm. Our results show that when $\alpha$ or $\beta$ is equal to zero, then the first Hochschild cohomology group of $A_{\alpha,\beta}$ is a 1-dimension vector space, whereas when both $\alpha$ and $\beta$ are non-zero, all derivations are inner.


\section{The quantum nilpotent algebra $\uq$ and its primitive ideals}


\subsection{The quantum nilpotent algebra $\uq$}

Let $\k$ be a field and  $q$ be a non-zero element of $\k$ that is not a root of unity. 
The algebra of $\uq$ is the so-called positive part of the quantum enveloping algebra $U_q(\mathfrak{g})$ of a Lie algebra $\mathfrak{g}$ of type $G_2$. It is well known, see for instance \cite{bg}, that this algebra is generated over $\k$ by two indeterminates $E_{\alpha}$ and $E_{\beta}$ subject to the following quantum Serre relations: 
\begin{align*}
(S1) \ \ E_{\alpha}^4E_{\beta}&- \left[\begin{smallmatrix} 4 \\ \\ 1 \end{smallmatrix}\right]_q E_{\alpha}^3E_{\beta}E_{\alpha}+\left[\begin{smallmatrix} 4 \\ \\ 2 \end{smallmatrix}\right]_q E_{\alpha}^2E_{\beta}E_{\alpha}^2-\left[\begin{smallmatrix} 4 \\ \\ 1 \end{smallmatrix}\right]_q  E_{\alpha}E_{\beta}E_{\alpha}^3+E_{\beta}E_{\alpha}^4=0,\\
(S2) \ \ E_{\beta}^2E_{\alpha}&-\left[\begin{smallmatrix} 2 \\ \\ 1 \end{smallmatrix}\right]_{q^3} E_{\beta}E_{\alpha}E_{\beta}+E_{\alpha}E_{\beta}^2=0,\\
\end{align*}
where $\left[\begin{smallmatrix} n \\ \\ i \end{smallmatrix}\right]_{z}$ denotes the quantum binomial coefficients (see \cite[I.6.1]{bg}).

One can construct a PBW-basis of $\uq$ using the so-called Lusztig automorphisms of $U_q(\mathfrak{g})$, see for instance \cite[I.6.8]{bg}. In the present case, such a basis was computed by De Graaf in \cite{degraaf1}. We will use the convention of that paper, but with $E_1:=E_{\alpha}, \ E_2:=E_{3\alpha+\beta}, \ E_3:=E_{2\alpha+\beta}, \ E_4:=E_{3\alpha+2\beta}, \ E_5:=E_{\alpha+\beta}$ and $E_6:=E_{\beta}$. 

With these notations, the defining relations of $U_q^+(G_2)$ are as follows:
\begin{align*}
E_2E_1&=q^{-3}E_1E_2& E_3E_1&=q^{-1}E_1E_3-(q+q^{-1}+q^{-3})E_2\\
E_3E_2&=q^{-3}E_2E_3&E_4E_1&=E_1E_4+(1-q^2)E_3^2\\
E_4E_2&=q^{-3}E_2E_4-\frac{q^4-2q^2+1}{q^4+q^2+1}E_3^3&E_4E_3&=q^{-3}E_3E_4
\\
E_5E_1&=qE_1E_5-(1+q^2)E_3&E_5E_2&=E_2E_5+(1-q^2)E_3^2\\
E_5E_3&=q^{-1}E_3E_5-(q+q^{-1}+q^{-3})E_4&E_5E_4&=q^{-3}E_4E_5\\
E_6E_1&=q^3E_1E_6-q^3E_5&E_6E_2&=q^3E_2E_6+(q^4+q^2-1)E_4+(q^2-q^4)E_3E_5\\
 E_6E_3&=E_3E_6+(1-q^2)E_5^2  &E_6E_5&=q^{-3}E_5E_6\\
E_6E_4&=q^{-3}E_4E_6-\frac{q^4-2q^2+1}{q^4+q^2+1}E_5^3, 
\end{align*}
and the monomials $E_1^{k_1} \dots E_6^{k_6}$ ($k_1, \dots ,k_6 \in \N $) form a basis of $\uq$ over $\k$. 

Even better, one may write $\uq$ as a Quantum Nilpotent Algebra (QNA for short) or Cauchon-Goodearl-Letzter extension in the sense of \cite[Definition 3.1]{llr-cgl}, by adjoining the generators $E_i$ in lexicographic order. This means in particular that $\uq$ can be presented as an iterated Ore extension:
\[
\uq=\k[E_{1}][E_{2};\sigma_{2},\delta_{2}]\cdots [E_{6};\sigma_{6},\delta_{6}],
\]
where the $\sigma_{i}$ are automorphisms and the $\delta_{i}$ are left $\sigma_{i}$-derivations of the appropriate subalgebras. We would not need the precise definition of a QNA for what follows, but it is worth reminding the reader of the algebraic torus action involved in writing $\uq$ as a QNA.

The algebraic torus $\ch=(\k^\times)^{2}$ acts by automorphisms on 
$\uq$ as follows:
\[
h\cdot E_i=h_i E_{i} \mbox{ for all } i \in \{1,6\} \mbox{ and } h=(h_1,h_6) \in \ch .
\]
Note that the action of the automorphism $h$ on the generators $E_2, \dots , E_5$ follows from the above defining relations. 

By \cite[Theorem II.2.7]{bg}, the action of $\ch$ on $\uq$ is rational in the sense of \cite[Definition II.2.6]{bg}. 

A consequence of the QNA condition is that important tools such as Cauchon's deleting derivations procedure and the Goodearl-Letzter stratification theory 
(this is the origin of the CGL terminology, see \cite{llr-cgl}) are available to study prime and primitive ideals. These ideas will be introduced in following sections. At the moment, we merely note that it is immediate that $\uq$ is a noetherian domain and that all prime ideals are completely prime (in the case of $\uq$, it was proved in \cite[Section 5]{picp}). We denote by $F_q$ its skew-field of fractions, i.e. $F_q:= \mathrm{Frac}(\uq)$.


\subsection{Prime ideals in $\uq$ and $\ch$-stratification}

A two-sided ideal $I$ of $\uq$ is said to be {\em $\ch$-invariant} if $h\cdot I=I$ for all
$h \in \ch$.  An {\em $\ch$-prime ideal} of $\uq$ is a proper $\ch$-invariant ideal
$J$ of $\uq$ such that if  $J$ contains the product of two
$\ch$-invariant ideals of $\uq$ then $J$ contains at least one of them. We denote
by $\ch$-$\spec(\uq)$ the set of all $\ch$-prime ideals of $\uq$. Observe
that if $P$ is a prime ideal of $\uq$ then
\begin{equation}
(P:\ch)\ := \ \bigcap_{h\in \ch} h\cdot P
\end{equation}
 is an $\ch$-prime ideal of $\uq$. Indeed, let $J$ be an $\ch$-prime ideal of $\uq$. We denote
by $\spec_J (\uq)$ the {\em $\ch$-stratum} associated  to $J$; that is,  
\begin{equation}
\spec_J (\uq)=\{ P \in \spec(\uq) \mbox{ $\mid$ } (P:\ch)=J \}.
\end{equation}
Then the $\ch$-strata of $\spec(\uq)$ form a partition of $\spec(\uq)$
\cite[Chapter II.2]{bg}; that
is,
\begin{equation}
\label{eq:Hstratification}
 \spec(\uq)= \bigsqcup_{J \in \ch\mbox{-}\spec(\uq)}\spec_J(\uq).
 \end{equation}
This partition is the so-called {\em $\ch$-stratification} of $\spec(\uq)$.

It follows from the work of Goodearl and Letzter \cite{gz2} that every $\ch$-prime ideal of $\uq$ is completely prime, so $\ch$-$\spec(\uq)$ coincides with the set of $\ch$-invariant completely prime ideals of $\uq$. Moreover there are precisely $|W|$ $\ch$-prime ideals in $\uq$, where $W$ denotes the Weyl group of type $G_2$ (see \cite[Remark 6.2.2]{meriaux}). As a consequence,  the $\ch$-stratification of $\spec ( \uq ) $ is finite and so the full strength of the $\ch$-stratification theory of Goodearl and Letzter is available to study $\spec (\uq)$. 

For each $\ch$-prime ideal $J$ of $\uq$, the space $\spec_J(\uq)$ is homeomorphic to the prime spectrum $\spec (\k[z_1^{\pm 1},\ldots ,z_d^{\pm 1}])$ of a commutative Laurent polynomial ring whose dimension depends on $J$, \cite[Theorems II.2.13 and II.6.4]{bg}. These dimensions were computed in \cite{bcl,yak}. Finally, let us mention that the primitive ideals of $A$ are precisely the prime ideals that are maximal in their $\ch$-strata \cite[Theorem II.8.4]{bg}.

In this article, we will mainly focus on one specific $\ch$-stratum. Since $\uq$ is a domain, $0$ (technically, $\langle0\rangle$) is clearly an $\ch$-invariant completely prime ideal of $\uq$, and so an $\ch$-prime, and we will focus on computing its stratum, the so-called $0$-stratum. The motivation here is twofold: first, in the $B_2$ case, we obtain ``new'' quantum deformation of the first Weyl algebra as $U_q^+(B_2) / P$, where $P$ is a primitive ideal from the $0$-stratum of $\spec  (U_q^+(B_2))$ \cite{sl}. Next, in the present case, we would like to construct algebras of GK dimension 4 as explained in the introduction. Since Tauvel's height formula holds in $\uq$ \cite{gll}, we need to quotient $\uq$ by a primitive ideals of height 2. Given that the $\ch$-spectrum of $\uq$ is homeomorphic to the Weyl group of type $G_2$, such primitive ideals can only be found in the $0$-stratum and the strata associated to one of the two height 1 $\ch$-primes. In this article, we mainly present results for the $0$-stratum, but we will also indicate results obtained for the primitive quotients coming from the height 1 $\ch$-prime strata.


\subsection{Deleting derivations algorithms in $\uq$}

As $\uq$ is a QNA, we can apply Cauchon's Deleting Derivation Algorithm to study its prime spectrum. 

Recall first that $\uq$ is an itereated Ore extension of the form:
$$U_q^+(G_2)=\k[E_1][E_2;\sigma_2][E_3;\sigma_3,\delta_3][E_4;\sigma_4,\delta_4][E_5;\sigma_5,\delta_5][E_6;\sigma_6,\delta_6];$$ where, $\sigma_2$ denotes the automorphism of $\k[E_1]$ defined by:
$$\sigma_2(E_1)=q^{-3}E_1,$$
$\sigma_3$ denotes the automorphism of $\k[E_1][E_2;\sigma_2]$ defined by:
$$\sigma_3(E_1)=q^{-1}E_1 \ \ \ \ \ \sigma_3(E_2)=q^{-3}E_2,$$
$\delta_3$ denotes the $\sigma_3$-derivation of $\k[E_1][E_2;\sigma_2]$ defined by:
$$\delta_3(E_1)=-(q+q^{-1}+q^{-3})E_2 \ \ \ \ \ \delta_3(E_2)=0,$$
$\sigma_4$ denotes the automorphism of $\k[E_1]\cdots[E_3;\sigma_3,\delta_3]$ defined by:
$$\sigma_4(E_1)=E_1 \ \ \ \ \ \sigma_4(E_2)=q^{-3}E_2 \ \ \ \ \sigma_4(E_3)=q^{-3}E_3,$$
$\delta_4$ denotes the $\sigma_4$-derivation of $\k[E_1]\cdots[E_3;\sigma_3,\delta_3]$ defined by:
$$\delta_4(E_1)=(1-q^2)E_3^2 \ \ \ \ \ \delta_4(E_2)=\frac{-q^4+2q^2-1}{q^4+q^2+1}E_3^3 \ \ \ \ \ \delta_4(E_3)=0,$$
$\sigma_5$ denotes the automorphism of $\k[E_1]\cdots[E_4;\sigma_4,\delta_4]$ defined by:
$$\sigma_5(E_1)=qE_1 \ \ \ \ \ \sigma_5(E_2)=E_2 \ \ \ \ \sigma_5(E_3)=q^{-1}E_3 \ \ \ \ \ \sigma_5(E_4)=q^{-3}E_4,$$
$\delta_5$ denotes the $\sigma_5$-derivation of $\k[E_1]\cdots[E_4;\sigma_4,\delta_4]$ defined by:
$$\delta_5(E_1)=-(1+q^2)E_3 \ \ \ \ \ \delta_5(E_2)=(1-q^2)E_3^2 \ \ \ \ \ \delta_5(E_3)=-(q+q^{-1}+q^{-3})E_4 \ \ \ \ \delta_5(E_4)=0,$$
$\sigma_6$ denotes the automorphism of $\k[E_1]\cdots[E_5;\sigma_5,\delta_5]$ defined by:
$$\sigma_6(E_1)=q^3E_1 \ \ \ \ \ \sigma_6(E_2)=q^3E_2 \ \ \ \ \sigma_6(E_3)=E_3 \ \ \ \ \ \sigma_6(E_4)=q^{-3}E_4 \ \ \ \ \sigma_6(E_5)=q^{-3}E_5,$$
and $\delta_6$ denotes the $\sigma_6$-derivation of $\k[E_1]\cdots[E_5;\sigma_5,\delta_5]$ defined by:
$$\delta_6(E_1)=-q^3E_5 \ \ \ \ \ \delta_6(E_2)=(q^2-q^4)E_3E_5+(q^4+q^2-1)E_4 \ \ \ \ \ \delta_6(E_3)=(1-q^2)E_5^2 $$
$$\delta_6(E_4)=\dfrac{-q^4+2q^2-1}{q^4+q^2+1}E_5^3 \ \ \ \ \ \delta_6(E_5)=0.$$

The deleting derivations algorithm (DDA for short) constructs by a decreasing induction a family 
$\{E_{1,j},\dots,E_{6,j}\}$ of elements of the division ring of fractions $F_q={\rm Fract}(\uq)$ of $\uq$ for each $2\leq j \leq 7$. The precise definition of these elements in the general context of QNAs can be found in \cite{ca}. 

In the present case, direct computation leads to: 

\begin{align*}
E_{1,6}&=E_1+rE_5E_6^{-1}\\
E_{2,6}&=E_2+tE_3E_5E_6^{-1}+uE_4E_6^{-1}+nE_5^3E_6^{-2}\\
E_{3,6}&=E_3+sE_5^2E_6^{-1}\\
E_{4,6}&=E_4+bE_5^3E_6^{-1}\\
E_{1,5}&=E_{1,6}+hE_{3,6}E_{5,6}^{-1}+gE_{4,6}E_{5,6}^{-2}\\ 
E_{2,5}&=E_{2,6}+fE_{3,6}^2E_{5,6}^{-1}+pE_{3,6}E_{4,6}E_{5,6}^{-2}+eE_{4,6}^2E_{5,6}^{-3
}\\
E_{3,5}&=E_{3,6}+aE_{4,6}E_{5,6}^{-1}\\
E_{1,4}&=E_{1,5}+sE_{3,5}^2E_{4,5}^{-1}\\
E_{2,4}&=E_{2,5}+bE_{3,5}^3E_{4,5}^{-1}\\
E_{1,3}&=E_{1,4}+aE_{2,4}E_{3,4}^{-1}\\
T_1&:=E_{1,2}=E_{1,3}\\
T_2&:= E_{2,2}=E_{2,3}=E_{2,4}\\
T_3&:=E_{3,2}=E_{3,3}=E_{3,4}=E_{3,5}\\
T_4&:=E_{4,2}=E_{4,3}=E_{4,4}=E_{4,5}=E_{4,6}\\
T_5&:=E_{5,2}=E_{5,3}=E_{5,4}=E_{5,5}=E_{5,6}=E_5\\
T_6&:=E_{6,2}=E_{6,3}=E_{6,4}=E_{6,5}=E_{6,6}=E_6,
\end{align*}
where the parameters $a,b,e,f,g,h,n,p,r,s,t,u$ are all defined in Appendix \ref{appenc}.

In the following, we set $A:=\uq$ and we denote by $A^{(j)}$ the subalgebra of $F_q$ generated by $E_{1,j}$, ..., $E_{6,j}$. The following results were proved by Cauchon \cite[Th\'eor\`eme 3.2.1 and Lemme 4.2.1]{ca}. 
For $2\leq j \leq 7$, we have:
\begin{enumerate}
\item when $j=7$, $(E_{1,7},\dots,E_{6,7}) =(E_1, \dots , E_6)$, so that $A^{(7)}=A=\uq$;
\item $A^{(j)}$ is isomorphic to an iterated Ore extension of the form $$\k[y_1]\dots[y_{j-1};\sigma_{j-1},\delta_{j-1}][y_j;\tau_j]\cdots[y_6;\tau_6]$$ 
by an isomorphism that sends $E_{i,j}$ to $y_i$ ($1 \leq i \leq 6$), where 
$\tau_j,\dots,\tau_6$ denote the $\k $-linear automorphisms  such that 
$\tau_{\ell}(y_i)=\lambda_{\ell,i} y_i$ ($1 \leq i \leq \ell$).
\item Assume that $j \neq 7$ and set $\Sigma_j:=\{E_{j,j+1}^n ~|~  n\in \mathbb{N} \}=
\{E_{j,j}^n  ~|~  n\in \mathbb{N} \}$.
\\This is a multiplicative system of regular elements of $A^{(j)}$ and $A^{(j+1)}$, that satisfies the Ore condition
in $A^{(j)}$ and $A^{(j+1)}$. Moreover we have 
$$A^{(j)}\Sigma_j^{-1}=A^{(j+1)}\Sigma_j^{-1}.$$
\end{enumerate}
It follows from these results that $A^{(j)}$ is a noetherian domain, for all $2\leq j \leq 7$.

As in \cite{ca}, we use the following notation.

\begin{notation}
We set $\overline{A}:=A^{(2)}$ and $T_i:=E_{i,2}$ for all $1\leq i \leq 6$.
\end{notation}

It follows from \cite[Proposition 3.2.1]{ca} that  $\overline{A}$ is a quantum affine space in the indeterminates $T_1,\dots,T_6$ and so can be presented as an iterated Ore extension in the $T_i$s with no skew-derivations: it is for this reason that Cauchon used the expression ``effacement des d\'erivations". More precisely, let $M=\left( \mu_{i,j} \right) \in M_6(\k ^*)$ be the
multiplicatively antisymmetric matrix defined as follows: 
$$M = \begin{bmatrix}
0&3&1&0&-1&-3\\
-3&0&3&3&0&-3\\
-1&-3&0&3&1&0\\
0&-3&-3&0&3&3\\
1&0&-1&-3&0&3\\
3&3&0&-3&-3&0
\end{bmatrix}.$$
Then we have \begin{equation}\overline{A}=\k_{q^M} [T_1,\dots,T_6]=\mathcal{O}_{q^M}(\k ^6),  \end{equation}
where $\k_{q^M} [T_1,\dots,T_6]=\mathcal{O}_{q^M}(\k ^6)$ denotes the $\k$-algebra generated by $T_1, \dots , T_6$ with relations $T_iT_j = q^{\mu_{i,j}}T_jT_i$ for all $i,j$. 
 
 
\subsection{Canonical embedding}
\label{c1s1}

Since $A=\uq$ is a QNA, one can use Cauchon's DDA in order to relate the prime spectrum of $A$ to the prime spectrum of the associated quantum affine space $\overline{A}$. More precisely, the DDA allows the construction of embeddings 
\begin{equation}
\psi_j:{\rm Spec}(A^{(j+1)}) \longrightarrow {\rm Spec}(A^{(j)}) \qquad (2\leq j \leq 6).
\end{equation} 
Recall from \cite[Section 4.3]{ca} that these embeddings are defined as follows.

Let $P \in  {\rm Spec}(A^{(j+1)})$. Then 
$$ 
\psi_j (P) = \left\{\begin{array}{ll} 
P\Sigma_j^{-1} \cap A^{(j)} & \mbox{ if } E_{j,j+1}=T_j \notin P \\
g_j^{-1} \left( P/\langle E_{j,j+1}\rangle \right) & \mbox{ if } E_{j,j+1} \in P \\
\end{array}\right.
$$
where $g_j$ denotes the surjective homomorphism 
$$
g_j:A^{(j)}\rightarrow A^{(j+1)}/\langle E_{j,j+1} \rangle
$$ 
defined by 
$$
g_j(E_{i,j}):=E_{i,j+1} + \langle E_{j,j+1} \rangle
$$
 (for more details, see \cite[Lemme 4.3.2]{ca}).  It was proved by Cauchon \cite[Proposition 4.3.1]{ca} that $\psi_j$ induces an increasing homeomorphism from the topological space $$\{P \in  {\rm Spec}(A^{(j+1)}) \mid E_{j,j+1} \notin P \}$$ onto $$\{Q \in  {\rm Spec}(A^{(j)}) \mid E_{j,j} \notin Q \}$$ whose inverse is also an increasing homeomorphism. Also, $\psi_j$ induces an increasing homeomorphism from $$\{P \in  {\rm Spec}(A^{(j+1)}) \mid E_{j,j+1} \in P \}$$ onto its image by $\psi_j$ whose inverse similarly is an increasing homeomorphism. Note however that, in general, $\psi_j$ is not an homeomorphism from ${\rm Spec}(A^{(j+1)})$ onto its image.

Composing these embeddings, we get an embedding \begin{equation} 
\psi:=\psi_2 \circ \dots \circ \psi_6 : 
{\rm Spec}(A) \longrightarrow {\rm Spec}(\overline{A}), \end{equation} which is called the  \emph{canonical embedding} from ${\rm Spec}(A)$ into 
${\rm Spec}(\overline{A})$. 

The canonical embedding $\psi$ is $\ch$-equivariant so that $\varphi (\hspec(A)) \subseteq \hspec(\overline{A})$. 
Interestingly, the set $\hspec(\overline{A})$ has been described by Cauchon as follows. For any subset $C$ of $\{1,\dots,6\}$, let $K_C$ denote the $\ch$-prime ideal of 
 $\overline{A}$  generated by the $T_i$ with $i\in C$, that is
 \[
K_C= \langle T_i\mid i\in C\rangle .
\]
It follows from \cite[Proposition 5.5.1]{ca} that  
$$\hspec(\overline{A}) = \{K_C ~|~ C \subseteq \{1,\dots,6\} \},$$
so that 
$$\psi (\hspec(A)) \subseteq \{K_C ~|~ C \subseteq \{1,\dots,6\} \}.$$


\section{Primitive ideals of $\uq$ in the $0$-stratum}

The aim of this section is to give explicit generating sets for the primitive ideals of $\uq$ that belong to the $0$-stratum. They are intimately related to the centre of $\uq$ and so we start this section by making explicit the centre of $\uq$ and related algebras. 

\subsection{Centre of $\uq$} 
\label{c2s1}

Recall that $\overline{A}:=A^{(2)}=\k_{\Lambda}[T_1,\cdots, T_6]$ is a quantum affine space.  Set $\Omega_1:=T_1T_3T_5$ and $\Omega_2:=T_2T_4T_6.$ One can easily verify that $\Omega_1$ and $\Omega_2$ are central elements of $\overline{A}$ by checking they commute with all $T_i$s. 

We now want to successively pull $\Omega_1$ and $\Omega_2$ from the quantum affine space $\overline{A}$ into the algebra $A$ using the data of DDA of $A$ discussed above. Direct computation shows that 
\begin{align*}
\Omega_1&:=T_1T_3T_5\\&=E_{1,4}E_{3,4}E_{5,4}+aE_{2,4}E_{5,4}\\
&=E_{1,5}E_{3,5}E_{5,5}
+aE_{2,5}E_{5,5}\\
&=E_{1,6}E_{3,6}E_{5,6}
+aE_{1,6}E_{4,6}
+aE_{2,6}
E_{5,6}+a'E_{3,6}^2\\
&=E_1E_3E_5+aE_1E_4+aE_2E_5+a'E_3^2, 
\end{align*}
and 
\begin{align*}
\Omega_2&:=T_2T_4T_6\\
&=E_{2,5}E_{4,5}E_{6,5}+bE_{3,5}^3E_{6,5}\\
&=E_{2,6}E_{4,6}E_{6,6}+bE_{3,6}^3E_{6,6}\\
&=E_2E_4E_6+bE_2E_5^3+bE_3^3E_6
+b'E_3^2E_5^2+c'E_3E_4E_5+d'E_4^2,
\end{align*}
where the parameters $a,b,a',b',c',d'$ can be found in Appendix \ref{appenc}. Note, $\Omega_1$ and $\Omega_2$ are central elements of $A^{(j)}$ for each $2\leq j\leq 7$ since Fract$(A^{(j)})=$ Fract$(\overline{A}).$  

We now want to show that the centre of $A$ and other related algebras is a polynomial ring generated by $\Omega_1$ and $\Omega_2$ over $\k.$ The following discussions will lead us to the proof. 

Set $S_j:=\{\lambda T_j^{i_j}T_{j+1}^{i_{j+1}}\cdots T_6^{i_6}\mid i_j,\cdots,i_6\in \mathbb{N} \ \text{and} \ \lambda\in \k^* \}$ for each $2\leq j\leq 6.$ One can observe that $S_j$ is a multiplicative system of non-zero divisors of $A^{(j)}=\k\langle E_{i,j}\mid \ \text{for all} \ i=1,\cdots, 6\rangle.$ Furthermore, the elements $T_j, \cdots, T_6$ are all normal in $A^{(j)}.$ Hence, $S_j$ is an Ore set in $A^{(j)}.$ We can therefore localize $A^{(j)}$ at $S_j$ as follows:
$$R_j:=A^{(j)}S_j^{-1}.$$
Recall that $\Sigma_j:=\{T_j^n\mid n\in \mathbb{N}\}$ is an Ore set in both $A^{(j)}$ and $A^{(j+1)}$ for each $2\leq j\leq 6,$ and that
$$A^{(j)}\Sigma_j^{-1}=A^{(j+1)}\Sigma_j^{-1}.$$
For all $2\leq j\leq 6,$ we have that:
\begin{align}
\label{emb0}
R_j= A^{(j)}S_j^{-1}=(A^{(j)}\Sigma_j^{-1})S_{
1+j}^{-1}=(A^{(j+1)}\Sigma_j^{-1})S_{
1+j}^{-1}=(A^{(j+1)}S_{j+1}^{-1})\Sigma_{
j}^{-1}=R_{j+1}\Sigma_j^{-1}.
\end{align}
Note, $R_7:=A.$

Again, one can also observe that $T_1$ is normal in $R_2.$ As a result, we can form the localization $R_1:=R_2[T_1^{-1}]$. The algebra $R_1$ is the quantum torus associated to the quantum affine space $\overline{A}.$ As a result, $R_1=\mathbb{\k}_{q^{M}}[T_1^{\pm 1},\cdots, T_6^{\pm 1}],$ where  $T_iT_j=q^{\mu_{ij}}T_jT_i$ for all $1\leq i,j\leq 6$ and $\mu_{ij}\in M.$  
Similar to \cite[\S 31]{ss}, we construct the following tower of algebras:
\begin{align}
\label{emb}
A=R_7&\subset R_6=R_7\Sigma_6^{-1}\subset R_5=R_6\Sigma_5^{-1}\subset R_4=R_5\Sigma_4^{-1}\nonumber\\
&\subset R_3=R_4\Sigma_3^{-1}\subset R_2=R_3\Sigma_2^{-1}\subset R_1.
\end{align}
Note, the family $(E_{1,j}^{k_1} \cdots E_{6,j}^{k_6}),$ where $k_i\in \mathbb{N}$ if $i<j$ and $k_i\in \mathbb{Z}$ otherwise is a PBW-basis of 
$R_j$ for all $1\leq i,j\leq 7.$ In particular, the family $( T_1^{k_1}T_2^{k_2}T_3^{k_3}T_4^{k_4}T_5^{k_5}T_6^{k_6})_{k_1,\cdots, k_6\in \mathbb{Z}}$ is a basis of $R_1.$ 

\begin{lemma}
\label{rlr}
\begin{enumerate}
\item $Z(R_1)=\k[\Omega_1^{\pm 1},\Omega_2^{\pm 1}].$
\item $Z(R_3)=\k[\Omega_1,\Omega_2].$
\item $Z(\overline{A})=\k[\Omega_1,\Omega_2].$
\end{enumerate}
\end{lemma}
\begin{proof}
\begin{enumerate}
\item It follows from \cite[1.3]{gz2} that $Z(R_1)$ is a commutative Laurent polynomial ring generated by certain monomials in the $T_i$s. A direct computation proves the result. 
\item  Clearly, $\k[\Omega_1,\Omega_2]\subseteq Z(R_3).$ For the reverse inclusion, let $y\in Z(R_3).$ Then, $y$ can be written in terms of the basis of $R_3$ (recall, $T_i=E_{i,3}$) as: $$y=\displaystyle \sum_{(i,\cdots,n)\in \mathbb{N}^2\times \mathbb{Z}^4}a_{(i,\cdots,n)} T_1^iT_2^{j}T_3^{k}T_4^{l}T_5^{m}T_6^{n}.$$  
Using the fact that $T_1, \cdots, T_6$ are all normal elements in $R_3$ and $yT_i=T_i y$ for all $i$, one easily conclude that $i=k=m$ and $j=l=n$ for all monomials appearing in $y$. Since $i,j\geq 0$,   we have that $y=\sum_{(i,j)\in \mathbb{N}^2}q^{\bullet}a_{(i,j)} T_1^iT_3^{i}T_5^{i}T_2^{j}T_4^{j}T_6^{j}=\sum_{(i,j)\in \mathbb{N}^2}q^{\bullet}a_{(i,j)}\Omega_1^i\Omega_2^j.$ This implies that $y\in \k[\Omega_1,\Omega_2]$ as expected. 
\item Similar to the previous case, and so left to the reader. 
\end{enumerate}
\end{proof}

\begin{lemma}
\label{c3l1}
$Z(A)=\k[\Omega_1,\Omega_2].$
\end{lemma}
\begin{proof}
 Since $R_i$ is a localization of $R_{i+1},$ it follows that $Z(R_{i+1})\subseteq Z(R_i).$ From \eqref{emb}, we have that  $Z(A)\subseteq Z(R_3).$ Observe that
 $\k[\Omega_1,\Omega_2]\subseteq Z(A)\subseteq Z(R_3)=\k[\Omega_1,\Omega_2].$ Hence, $Z(A)=\k[\Omega_1,\Omega_2].$
\end{proof}
\begin{remark}
\label{remark-centre-Ri}
Since $Z(A)=Z(R_3)=\k[\Omega_1,\Omega_2]$ and $Z(R_{i+1})\subseteq Z(R_i),$  it follows from  \eqref{emb} that $Z(A)=Z(R_6)=Z(R_5)=Z(R_4)=Z(R_3)=\k[\Omega_1,\Omega_2].$ One can also deduce from Lemma \ref{rlr} that $Z(R_2)=\k[\Omega_1, \Omega_2^{\pm 1}].$   
\end{remark}

\begin{remark}
The centre of  the positive part of a quantised enveloping algebra of a simple Lie algebra has been described by Caldero in \cite{caldero} but we will need Remark \ref{remark-centre-Ri} later on.
\end{remark}

\subsection{$\Omega_1$ and $\Omega_2$ generate completely prime ideals of $\uq$}

The aim of this paragraph is to show that $\langle\Omega_1\rangle$ and $\langle\Omega_2\rangle$ are (completely) prime. We will make use of  DDA to establish these facts. Note that we could also have used the results of \cite{goodearl2016quantum} to obtain these results. However, we will need some of the intermediate steps obtained here to compute the derivations of certain primitive quotients of $\uq$ in the final section.

From Section \ref{c1s1} we know that there is a bijection between  $\{P\in \text{Spec}(A^{(j+1)})\mid P\cap \Sigma_j=\emptyset\}$ and $\{Q\in \text{Spec}(A^{(j)})\mid Q\cap \Sigma_j=\emptyset\}$ via $P=Q\Sigma_j^{-1}\cap A^{(j+1)}.$ Note, $\langle T_1\rangle$ and 
 $\langle T_2\rangle$ are prime ideals of the quantum affine space $\overline{A},$ since each of the factor algebras
 $\overline{A}/\langle T_1\rangle$ and $\overline{A}/\langle T_2\rangle$ is  isomorphic to a quantum affine space of rank 5 which is well known to be a domain.

 The following result and its proof show that $\langle T_1\rangle$ belongs to the image $\mathrm{Im}(\psi)$ of the canonical embedding $\psi$ and that $\langle\Omega_1\rangle$ is the completely prime ideal of $A$ such that $\psi (\langle\Omega_1\rangle)=\langle T_1\rangle.$

\begin{lemma} 
\label{lem1}
$\langle \Omega_1\rangle\in  \spec(A)$.
\end{lemma}
\begin{proof}
We will prove this result in several steps by showing that:
\begin{itemize}
\item[1.]  $\langle T_1\rangle_{A^{(3)}} \in$ Spec$(A^{(3)})$.
\item[2.] $\langle E_{1,4}T_3+a T_2\rangle =\langle T_1\rangle_{A^{(3)}}[T_3^{-1}]\cap A^{(4)},$  hence 
$Q_1:=\langle E_{1,4}T_3+a T_2\rangle \in $ Spec$(A^{(4)})$.
\item[3.] $\langle E_{1,5}T_3+a E_{2,5}\rangle =Q_1[T_4^{-1}]\cap A^{(5)},$ hence $Q_2:=\langle E_{1,5}T_3+aE_{2,5} \rangle \in $ Spec$(A^{(5)})$.
\item[4.] $\langle \Omega_{1}\rangle_{A^{(6)}} = Q_2[T_5^{-1}]\cap A^{(6)},$  hence $\langle \Omega_{1}\rangle_{A^{(6)}} \in $ Spec$(A^{(6)})$.
\item[5.] $\langle \Omega_1\rangle_{A}=\langle \Omega_{1}\rangle_{A^{(6)}}[T_6^{-1}]\cap A,$ hence 
$\langle \Omega_1\rangle_{A}\in$ Spec$(A)$.
\end{itemize}
We now proceed to prove the above claims. 

1. One can easily verify that $A^{(3)}/\langle T_1\rangle$ is isomorphic to a quantum affine space of rank 5, which is a domain, hence $\langle T_1\rangle$ is a prime ideal in $A^{(3)}.$

2. Note, $T_1=E_{1,4}+a T_2T_3^{-1}.$ We want to show that $\langle E_{1,4}T_3+a T_2\rangle =\langle T_1\rangle_{A^{(3)}}[T_3^{-1}]\cap A^{(4)}.$  Observe that $\langle E_{1,4}T_3+a T_2\rangle\subseteq \langle T_1\rangle_{A^{(3)}}[T_3^{-1}]\cap A^{(4)}.$ We established the reverse inclusion. Let $y\in\langle T_1\rangle_{A^{(3)}}[T_3^{-1}]\cap A^{(4)}.$ Then,  $y\in \langle T_1\rangle_{A^{(3)}}[T_3^{-1}].$ Therefore, there exists $i\in\mathbb{N}$ such that 
$y T_3^i\in \langle T_1\rangle_{A^{(3)}}.$ This implies that
$y T_3^i=T_1v ,$ for some $v\in A^{(3)}.$ Since $A^{(3)}[T_3^{-1}]=A^{(4)}[T_3^{-1}],$ there exists $j\in\mathbb{N}$ such that $v T_3^j=v',$ for some 
$v'\in A^{(4)}.$ It follows that
$y T_3^{i+j}=T_1v T_3^j=T_1v'=(E_{1,4}+aT_2T_3^{-1})v'=(E_{1,4}T_3+aT_2)T_3^{-1}v'.$ 
The multiplicative system generated by $T_3$ satisfies the Ore condition in $A^{(4)}$, hence, there exists
$k\in \mathbb{N}$ and $v''\in A^{(4)}$ such that $T_3^{-1}v'=v''T_3^{-k}.$
One can therefore  write $y T_3^{i+j}=(E_{1,4}T_3+aT_2)v''T_3^{-k}.$
This implies that
$yT_3^{\delta}=\Omega_{1}'v'',$ where $\Omega_{1}':=E_{1,4}T_3+aT_2$ and  $\delta=i+j+k.$ 
Set $S:=\{s\in\mathbb{N}\mid \exists v''\in A^{(4)}: y T_3^{s}=\Omega_{1}'v''\}.$ Note, $S\neq \emptyset,$  since $\delta\in S.$   Let $s=s_0$ be the minimum element of $S$ such that $y T_3^{s_0}=\Omega_{1}'v''.$ We want to show that $s_0=0.$   Remember, $\Omega_1'T_5=\Omega_1$  in $A^{(4)}.$  
Since $\Omega_1$ is central in $A^{(4)},$ and $T_5$ is normal in $A^{(4)},$ we must have $\Omega_1'$ to be a normal element in $A^{(4)},$ otherwise, there will be a contradiction. Therefore, there exists $w\in A^{(4)}$ such that $y T_3^{s_0}=\Omega_{1}'v''=w\Omega_1'.$
Now, $A^{(4)}$ can be viewed as a free left
$\k\langle E_{1,4},T_2,T_4,T_5,T_6\rangle-$module with basis $\left( T_3^{\xi}\right) _{\xi\in\mathbb{N}}.$ One can therefore write $y=\sum_{\xi=0}^n\alpha_\xi T_3^{\xi}$
and $w=\sum_{\xi=0}^n\beta_\xi T_3^{\xi},$ where $\alpha_\xi,\beta_\xi\in \k\langle E_{1,4},T_2,T_4,T_5,T_6\rangle.$
This implies that
$\sum_{\xi=0}^n\alpha_\xi T_3^{\xi+s_0}=\sum_{\xi=0}^n\beta_\xi T_3^{\xi}\Omega_1'=\sum_{\xi=0}^n q^{\bullet} \beta_\xi \Omega_1'T_3^{\xi}$ (note, $T_3\Omega_1'=q^{-1}\Omega_1'T_3$). Given that $\Omega_{1}'=E_{1,4}T_3+aT_2,$ we have that
$\sum_{\xi=0}^n\alpha_\xi T_3^{\xi+s_0}=
\sum_{\xi=0}^nq^{\bullet}  \beta_\xi E_{1,4}T_3^{1+\xi}+\sum_{\xi=0}^nq^{\bullet}  a\beta_\xi T_2T_3^{\xi}.$ Suppose 
that $s_0>0.$ Then, identifying the constant coefficients, we have $q^{\bullet} a\beta_0T_2=0.$ As a result,  
$\beta_0=0,$ since $q^{\bullet} aT_2\neq 0.$ Hence, $w$ can be written as $w=\sum_{\xi=1}^n\beta_\xi T_3^{\xi}.$
Returning to $yT_3^{s_0}=w\Omega_{1}',$ we have that $y T_3^{s_0}=\sum_{\xi=1}^n\beta_\xi T_3^{\xi}\Omega_{1}'=\sum_{\xi=1}^nq^{\bullet}\beta_\xi \Omega_{1}'T_3^{\xi}=\Omega_{1}'\sum_{\xi=1}^nq^{\bullet}\beta_\xi' T_3^{\xi}.$ This implies that $y T_3^{s_0-1}=\Omega_{1}'w',$
where $w'=\sum_{\xi=1}^nq^{\bullet}\beta_\xi' T_3^{\xi-1}\in A^{(4)},$ with $\beta_{\xi}'\in \k\langle E_{1,4},T_2,T_4,T_5,T_6\rangle.$ Consequently,
$s_0-1\in S,$ a contradiction! Therefore, $s_0=0$ and $y=\Omega_{1}'v''\in\langle\Omega_{1}'\rangle=\langle E_{1,4}T_3+aT_2\rangle.$ Hence,
$\langle T_1\rangle_{A^{(3)}}[T_3^{-1}]\cap A^{(4)}\subseteq\langle E_{1,4}T_3+aT_2\rangle$ as desired.

The following steps are proved in a similar manner to Step 2. They are left to the reader who might want to check details in \cite{io-thesis}.
\comment{3. 
We want to show that $\langle E_{1,5}T_3+a E_{2,5}\rangle =\langle \Omega_1'\rangle_{A^{(4)}}[T_4^{-1}]\cap A^{(5)}.$ Note, $\Omega_1'=E_{1,4}T_3+a T_2=E_{1,5}T_3+a E_{2,5}.$ Observe that $\langle E_{1,5}T_3+a E_{2,5}\rangle\subseteq \langle \Omega_1'\rangle_{A^{(4)}}[T_4^{-1}]\cap A^{(5)}.$ We establish the reverse inclusion. Let $y\in\langle \Omega_1'\rangle_{A^{(4)}}[T_4^{-1}]\cap A^{(5)}.$ Then, $y\in \langle \Omega_1'\rangle_{A^{(4)}}[T_4^{-1}].$ Therefore, there exists $i\in\mathbb{N}$ such that 
$y T_4^i\in \langle \Omega_1'\rangle_{A^{(4)}}.$ This implies that
$y T_4^i=\Omega_1'v,$ for some $v\in A^{(4)}.$ Since $A^{(4)}[T_4^{-1}]=A^{(5)}[T_4^{-1}],$ there exists $j\in\mathbb{N}$ such that $v T_4^j=v',$ for some 
$v'\in A^{(5)}.$ It follows that
$y T_4^{i+j}=\Omega_1'v T_4^j=\Omega_1'v'.$ This implies that
$yT_4^{\delta}=\Omega_{1}'v',$ where  $\delta=i+j.$ 
Set $S:=\{s\in\mathbb{N}\mid \exists v'\in A^{(5)} : y T_4^{s}=\Omega_{1}'v'\}.$  Since $\delta\in S,$ we have that $S\neq \emptyset.$  Let $s=s_0$ be the minimum element of $S$ such that $y T_4^{s_0}=\Omega_{1}'v'.$ We want to show that $s_0=0.$  Note, $\Omega_1'T_5=\Omega_1$ in $A^{(5)}.$ Since 
$\Omega_1$ is central in $A^{(5)},$ and $T_5$ is normal in $A^{(5)},$ we must have $\Omega_1'$ as a normal element in $A^{(5)}.$ Therefore, there exists some $v''\in A^{(5)}$ such that $y T_4^{s_0}=\Omega_{1}'v'=v''\Omega_1'.$
Now, $A^{(5)}$ can be viewed as a free left
$\k\langle E_{1,5},E_{2,5},T_3,T_5,T_6\rangle-$module with basis $\left( T_4^{\xi}\right) _{\xi\in\mathbb{N}}.$ One can write $y=\sum_{\xi=0}^n\alpha_\xi T_4^{\xi}$
and $v''=\sum_{\xi=0}^n\beta_\xi T_4^{\xi},$ where $\alpha_\xi,\beta_\xi\in \k\langle E_{1,5},E_{2,5},T_3,T_5,T_6\rangle.$
This implies that
$\sum_{\xi=0}^n\alpha_\xi T_4^{\xi+s_0}=\sum_{\xi=0}^n \beta_\xi T_4^{\xi}\Omega_1'=\sum_{\xi=0}^nq^\bullet \beta_\xi \Omega_1'T_4^{\xi}$ (note, $T_4\Omega_1'=q^{-3}\Omega_1'T_4$).  Suppose 
that $s_0>0.$ Then, identifying the constant coefficients, we have that $q^\bullet \beta_0 \Omega_1'=0.$ Hence, $\beta_0=0,$ since $q^{\bullet}\Omega_1'\neq 0.$ One can therefore write $v''$ as $v''=\sum_{\xi=1}^n\beta_\xi T_4^{\xi}.$
Returning to $yT_4^{s_0}=v''\Omega_{1}',$ we have that $y T_4^{s_0}=\sum_{\xi=1}^n\beta_\xi T_4^{\xi}\Omega_{1}'=\sum_{\xi=1}^nq^{\bullet}\beta_\xi \Omega_{1}' T_4^{\xi}=\Omega_{1}'\sum_{\xi=1}^nq^{\bullet}\beta_\xi' T_4^{\xi},$ where $\beta_\xi'\in \k\langle E_{1,5},E_{2,5},T_3,T_5,T_6\rangle.$ This implies that $y T_4^{s_0-1}=\Omega_{1}'w,$
where $w=\sum_{\xi=1}^nq^{\bullet}\beta_\xi' T_4^{\xi-1}\in A^{(5)}.$ Consequently,
$s_0-1\in S,$ a contradiction! Therefore, $s_0=0$ and $y=\Omega_{1}'v'\in\langle\Omega_{1}'\rangle=\langle E_{1,5}T_3+aE_{2,5}\rangle.$ As a result,
$\langle \Omega_1'\rangle_{A^{(4)}}[T_4^{-1}]\cap A^{(5)}\subseteq\langle E_{1,5}T_3+aE_{2,5}\rangle$ as desired.

4. Observe that $\Omega_1'=E_{1,5}T_3+aE_{2,5}=\Omega_1T_5^{-1}$ in $A^{(6)}[T_5^{-1}].$ 
We want to show that $\langle\Omega_1'\rangle_{A^{(5)}}[T_5^{-1}]\cap A^{(6)}=\langle\Omega_1\rangle_{A^{(6)}}.$ Obviously, $\langle\Omega_1\rangle_{A^{(6)}}\subseteq \langle\Omega_1'\rangle_{A^{(5)}}[T_5^{-1}]\cap A^{(6)}.$ We establish the reverse inclusion.
Let $y\in \langle\Omega_1'\rangle_{A^{(5)}}[T_5^{-1}]\cap A^{(6)}.$ This implies that
$y\in \langle\Omega_1'\rangle_{A^{(5)}}[T_5^{-1}].$ There exists $i\in\mathbb{N}$ such that $yT_5^i\in \langle\Omega_1\rangle_{A^{(5)}}.$ Hence, $yT_5^i=\Omega_1'v,$ for some $v\in A^{(5)}.$ Furthermore, since $A^{(5)}[T_5^{-1}]=A^{(6)}[T_5^{-1}],$ there exists $j\in\mathbb{N}$ such that 
$vT_5^j=v',$ for some $v'\in A^{(6)}.$ It follows from $yT_5^i=\Omega_1'v$  that $yT_5^{i+j}=\Omega_1'vT_5^j=\Omega_1'v'=\Omega_1T_5^{-1}v'$ (note, $\Omega_1'T_5=\Omega_1$ in $A^{(6)}$).
The multiplicative system generated by $T_5$ satisfies the Ore condition in $A^{(6)}$, hence, there exists 
$k\in \mathbb{N}$ and $v''\in A^{(6)}$ such that $T_5^{-1}v'=v''T_5^{-k}.$
One can therefore write $yT_5^{i+j}
=\Omega_1v''T_5^{k}.$   Hence,
$yT_5^\delta=\Omega_1v'',$ where
$\delta=i+j+k.$  Set $S:=\{s\in\mathbb{N}\mid \exists v''\in A^{(6)} : y T_5^{s}=\Omega_{1}v''\}.$ Since $\delta\in S,$ we have that $S\neq \emptyset.$  Let $s=s_0$ be the minimum element of $S$ such that $ y T_5^{s_0}=\Omega_{1}v''.$ We want to show that $s_0=0.$
Now, $A^{(6)}$ can be viewed as a free
$\k\langle E_{1,6},E_{2,6},E_{3,6},T_4,T_6\rangle-$module with basis $\left( T_5^{\xi}\right) _{\xi\in\mathbb{N}}.$ One can write $y=\sum_{\xi=0}^n\alpha_\xi T_5^{\xi}$
and $v''=\sum_{\xi=0}^n\beta_\xi T_5^{\xi},$ where $\alpha_\xi,\beta_\xi\in \k\langle E_{1,6},E_{2,6},E_{3,6},T_4,T_6\rangle.$ With this, $y T_5^{s_0}=\Omega_{1}v''$ implies that $\sum_{\xi=0}^n\alpha_\xi T_5^{\xi+s_0}=
\sum_{\xi=0}^n\beta_\xi\Omega_1T_5^{\xi}.$ Write $\Omega_1=\gamma_1T_5+\gamma_2,$ where
$\gamma_1=E_{1,6}E_{3,6}+aE_{2,6}$ and $\gamma_2=a'E_{3,6}^2+aE_{1,6}E_{4,6}.$ It follows that $\sum_{\xi=0}^n\alpha_\xi T_5^{\xi+s_0}=
\sum_{\xi=0}^n \beta_\xi \gamma_1T_5^{\xi+1}+\sum_{\xi=0}^n \beta_\xi\gamma_2 T_5^{\xi}.$
Suppose 
that $s_0>0.$ Then, identifying the constant coefficients, we have that $\beta_0 \gamma_2=0.$ Hence, 
$\beta_0=0,$ since $\gamma_2\neq 0.$ One can therefore write $v''=\sum_{\xi=1}^n\beta_\xi T_5^{\xi}.$
Returning to $yT_5^{s_0}=\Omega_{1}v'',$ we have that $y T_5^{s_0}=\Omega_{1}\sum_{\xi=1}^n\beta_\xi T_5^{\xi}.$ This implies that $y T_5^{s_0-1}=\Omega_{1}w,$
where $w=\sum_{\xi=1}^n\beta_\xi T_5^{\xi-1}\in A^{(6)}.$ As a result,
$s_0-1\in S,$ a contradiction! Therefore, $s_0=0$ and $y=\Omega_{1}v''\in\langle\Omega_{1}\rangle_{A^{(6)}}.$ Consequently, 
$\langle\Omega_1'\rangle_{A^{(5)}}[T_5^{-1}]\cap A^{(6)}\subseteq \langle\Omega_1\rangle_{A^{(6)}}$ as desired. 

5.
We want to show that $\langle\Omega_1\rangle_{A^{(6)}}[T_6^{-1}]\cap A=\langle\Omega_1\rangle_{A}.$\newline Obviously, $\langle\Omega_1\rangle_{A}\subseteq \langle\Omega_1\rangle_{A^{(6)}}[T_6^{-1}]\cap A.$ We establish the reverse inclusion.
Let $y\in \langle\Omega_1\rangle_{A^{(6)}}[T_6^{-1}]\cap A.$ This implies that
$y\in \langle\Omega_1\rangle_{A^{(6)}}[T_6^{-1}].$ There exists $i\in\mathbb{N}$ such that $yT_6^i\in \langle\Omega_1\rangle_{A^{(6)}}.$ Hence, $yT_6^i=\Omega_1v,$ for some $v\in A^{(6)}.$ Furthermore, since $A^{(6)}[T_6^{-1}]=A[T_6^{-1}],$ there exists $j\in\mathbb{N}$ such that 
$vT_6^j=v',$ for some $v'\in A.$ It follows from $yT_6^i=\Omega_1v$  that $yT_6^{i+j}=\Omega_1vT_6^j=\Omega_1v'.$ Hence, $yT_6^\delta=\Omega_1v',$ where
$\delta=i+j.$  Set $S:=\{s\in\mathbb{N}\mid \exists v'\in A : y T_6^{s}=\Omega_{1}v'\}.$ Note, 
$\delta\in S,$ hence $S$ is non-empty. Let $s=s_0$ be the minimum element of $S$ such that $ y T_6^{s_0}=\Omega_{1}v'.$ We want to show that $s_0=0.$
Now, $A$ can be viewed as a free
$\k\langle E_{1},E_{2},E_{3},E_{4},T_5\rangle-$module with basis $\left(T_6^\xi\right) _{\xi\in\mathbb{N}}.$ One can write $y=\sum_{\xi=0}^n\alpha_\xi T_6^{\xi}$
and $v'=\sum_{\xi=0}^n\beta_\xi T_6^{\xi},$ where $\alpha_\xi,\beta_\xi\in \k\langle E_{1},E_{2},E_{3},E_4,T_5\rangle.$ With this, $y T_6^{s_0}=\Omega_{1}v'$ implies that $\sum_{\xi=0}^n\alpha_\xi T_6^{\xi+s_0}=
\sum_{\xi=0}^n\beta_\xi\Omega_1T_6^{\xi}.$ 
Suppose 
that $s_0>0.$ Then, identifying constant coefficients, we have that $\beta_0 \Omega_1=0.$ As a result, 
$\beta_0=0,$ since $\Omega_1\neq 0.$ One can therefore write $v'=\sum_{\xi=1}^n\beta_\xi T_6^{\xi}.$
Returning to $yT_6^{s_0}=\Omega_{1}v',$ we have that $y T_6^{s_0}=\Omega_{1}\sum_{\xi=1}^n\beta_\xi T_6^{\xi}.$ This implies that $y T_6^{s_0-1}=\Omega_{1}v^{\prime\prime},$
where $v^{\prime\prime}=\sum_{\xi=1}^n\beta_\xi T_6^{\xi-1}\in A.$ Hence,
$s_0-1\in S,$ a contradiction! Therefore, $s_0=0$ and $y=\Omega_{1}v'\in\langle\Omega_{1}\rangle_{A}.$ Consequently, 
$\langle\Omega_1\rangle_{A^{(6)}}[T_6^{-1}]\cap A\subseteq \langle\Omega_1\rangle_{A}$ as desired.}
\end{proof}

Using similar techniques, one can prove that $\langle T_2\rangle\in$ Im$(\psi)$ and that  $\langle\Omega_2\rangle$ is the completely prime ideal of $A$ such that $\psi (\langle\Omega_2\rangle)=\langle T_2\rangle$. Again, we refer the interested reader to \cite{io-thesis} for details. We record these facts in the following lemma.
\begin{lemma} 
\label{lem2}
 $\langle T_2\rangle\in \mathrm{Im}(\psi)$ and  $\langle\Omega_2\rangle$ is the completely prime ideal of $A$ such that $\psi (\langle\Omega_2\rangle)=\langle T_2\rangle.$
\end{lemma}

Since $\ch$-$\spec (\uq)$ is homeomorphic to the Weyl group $W$ of type $G_2$ by \cite{my2}, there are only 2 $\ch$-primes in $\uq$ of height 1. Since $\Omega_1$ and $\Omega_2$ are central, the prime ideals that they generate have height less than or equal to 1, and so equal to 1. As an immediate consequence, we get the following result.

 \begin{lemma}
 \label{rem1}
 \begin{enumerate}
\item $\langle\Omega_1\rangle$ and $\langle\Omega_2\rangle$ are the only height one $\ch$-invariant prime ideals of $A$.
 \item Every non-zero $\ch$-invariant prime ideal of $A$ contains either $\langle\Omega_1\rangle$ or $\langle\Omega_2\rangle$.
\end{enumerate}
 \end{lemma}


\subsection{Description of the $0$-stratum and beyond} 
In this section, we will often assume that our base field $\k$ is algebraically closed. This assumption is actually not necessary for the main result of this section, Theorem \ref{c3p29}, but makes the description of the $0$-stratum easier to present. 

This section focuses on finding the height two maximal ideals of $A=\uq$. Note first that such ideals can only belong to the $\ch$-stratum of an $\ch$-prime of height less than or equal to 1 (since $\ch$-$\spec(A)$ is isomorphic to $W$).  It follows from the previous sections that we need to compute the $\ch$-strata of 3 $\ch$-primes: $0$, $\langle\Omega_1\rangle$ and  $\langle\Omega_2\rangle$. We start with the $0$-stratum.

The strategy is similar to \cite[Propositions 2.3 and 2.4]{sl}. Note, in this subsection, all ideals in $A$ will simply be written as $\langle \Theta\rangle,$ where $\Theta\in A.$ However, if we want to refer to an ideal in any other algebra, say $R$, then that ideal will be written as $\langle \Theta\rangle_R,$ where in this case, $\Theta\in R.$ 
  
\begin{proposition}
\label{p0}
Assume $\k$ is algebraically closed. Let $\mathcal{P}$ be the set of those unitary irreducible polynomials 
$P(\Omega_1,\Omega_2)\in \k[\Omega_1,\Omega_2]$
 with $P(\Omega_1,\Omega_2)\neq \Omega_1$ and $P(\Omega_1,\Omega_2)\neq \Omega_2$. Then,  $\spec_{\langle 0\rangle}(A)=
\{\langle 0\rangle\}\cup \{\langle P(\Omega_1,\Omega_2)\rangle\mid P(\Omega_1,\Omega_2)\in \mathcal{P}\}\cup \{\langle \Omega_1-\alpha,\Omega_2-\beta\rangle\mid\alpha,\beta\in \k^*\}.$
 \end{proposition}

\begin{proof}
We claim that Spec$_{\langle 0\rangle}(A)=\{Q \in \text{Spec}(A)\mid\Omega_1,\Omega_2\not\in Q\}.$ To establish this claim, let us assume that this is not the case. That is, suppose there exists $ Q\in$ Spec$_{\langle 0\rangle}(A)$ such that $\Omega_1,\Omega_2\in Q;$ then the product $\Omega_1\Omega_2$ which is an $\mathcal{H}-$eigenvector belongs to $Q.$ Consequently, $\Omega_1\Omega_2\in \bigcap_{h\in \mathcal{H}}h\cdot Q=\langle 0\rangle,$ a contradiction. Hence, we have shown that  Spec$_{\langle 0\rangle}(A)\subseteq \{Q \in \text{Spec}(A)\mid\Omega_1,\Omega_2\not\in Q\}.$ Conversely, suppose that $Q\in$ Spec$(A)$  such that $\Omega_1,\Omega_2\not\in Q,$ then $\bigcap_{h\in \mathcal{H}}h\cdot Q$ is an $\mathcal{H}-$invariant prime ideal of $A,$ which contains neither $\Omega_1$ nor $\Omega_2.$ Obviously, the only possibility for $\bigcap_{h\in \mathcal{H}}h\cdot Q$ is $\langle0\rangle$ since every non-zero $\mathcal{H}-$invariant prime ideal contains at least $\Omega_1$ or $\Omega_2$. Thus, $\bigcap_{h\in \mathcal{H}}h\cdot Q=\langle 0\rangle.$ Hence, $Q\in$ Spec$_{\langle 0\rangle}(A).$ Therefore, $\{Q \in \text{Spec}(A)\mid\Omega_1,\Omega_2\not\in Q\}\subseteq$ Spec$_{\langle 0\rangle}(A).$ This confirms our claim.

Since $\Omega_1,\Omega_2\in Z(A),$ we have that the set $\{\Omega_1^i\Omega_2^j\mid i,j\in\mathbb{N}\}$ is a right denominator set in the noetherian domain $A.$ One can now localize $A$ as $R:=A[\Omega_1^{-1},\Omega_2^{-1}].$ Let
$Q\in$ Spec$_{\langle 0\rangle}(A),$ the map
 $\phi:Q\longrightarrow 
Q[\Omega_1^{-1},\Omega_2^{-1}]$ is an increasing bijection from 
Spec$_{\langle 0\rangle}(A)$ onto Spec$(R).$

Since $\Omega_1$ and $\Omega_2$ are $\mathcal{H}-$eigenvectors, and  $\mathcal{H}$ acts on
$A$, we have that $\mathcal{H}$ also acts on $R.$
Since every $\ch$-prime ideal of $A$ contains $\Omega_1$ or $\Omega_2$, one can easily check that $R$ is $\mathcal{H}-$simple (in the sense that the only $\ch$-invariant ideal of $R$ is the $0$ ideal).

We proceed to describe Spec$(R)$ and  Spec$_{\langle 0\rangle}(A).$  We deduce from \cite[Exercise II.3.A]{bg} that the action of $\mathcal{H}$ on $R$ is rational. This rational action coupled with $R$ being $\mathcal{H}-$simple implies that the extension and contraction maps provide a mutually inverse bijection between Spec$(R)$ and Spec$(Z(R))$ \cite[Corollary II.3.9]{bg}. From Lemma \ref{c3l1}, $Z(A)=\k[\Omega_1,\Omega_2],$ and so $Z(R)=\k[\Omega_1^{\pm 1},\Omega_2^{\pm 1}].$ Since $\k$ is algebraically closed, we have that 
Spec$(Z(R))=\{\langle 0\rangle_{Z(R)}\}\cup \{\langle P(\Omega_1,\Omega_2)\rangle_{Z(R)}\mid P(\Omega_1,\Omega_2)\in \mathcal{P}\}\cup \{\langle \Omega_1-\alpha,\Omega_2-\beta\rangle_{Z(R)}\mid\alpha,\beta\in \k^*\}.$ Since there is an inverse bijection between Spec$(R)$ and Spec$(Z(R)),$ and also 
$R$ is $\mathcal{H}-$simple, one can recover Spec$(R)$ from Spec$(Z(R))$ as follows: 
Spec$(R)=\{\langle 0\rangle_{R}\}\cup \{\langle P(\Omega_1,\Omega_2)\rangle_{R}\mid P(\Omega_1,\Omega_2)\in \mathcal{P}\}\cup \{\langle \Omega_1-\alpha,\Omega_2-\beta\rangle_R\mid\alpha,\beta\in \k^*\}.$ It follows that
Spec$_{\langle 0\rangle}(A)=\{\langle 0\rangle_{R}\cap A\}\cup \{\langle P(\Omega_1,\Omega_2)\rangle_{R}\cap A\mid P(\Omega_1,\Omega_2)\in \mathcal{P}\}\cup \{\langle \Omega_1-\alpha,\Omega_2-\beta\rangle_R\cap A\mid \alpha,\beta\in \k^*\}.$

Undoubtedly, $\langle 0\rangle_{R}\cap A=\langle 0 \rangle.$ We now have to show that $\langle P(\Omega_1,\Omega_2)\rangle_{R}\cap A=\langle P(\Omega_1,\Omega_2)\rangle,$ $\forall P(\Omega_1,\Omega_2)\in \mathcal{P},$ and $\langle \Omega_1-\alpha,\Omega_2-\beta\rangle_R\cap A=\langle \Omega_1-\alpha,\Omega_2-\beta\rangle,$ $\forall \alpha,\beta\in \k^*$ to complete the proof.

Fix $P(\Omega_1,\Omega_2)\in \mathcal{P}.$ Observe that $\langle P(\Omega_1,\Omega_2) \rangle \subseteq \langle P(\Omega_1,\Omega_2)\rangle_{R}\cap A.$ To show the reverse inclusion,
let $y\in \langle P(\Omega_1,\Omega_2)\rangle_{R}\cap A$. This implies that $y=dP(\Omega_1,\Omega_2), $ where $d\in R,$ since $y\in \langle P(\Omega_1,\Omega_2) \rangle_R.$ Also, 
$d\in R$ implies that there exist $ i,j\in \mathbb{N}$ such that $d=a\Omega_1^{-i}\Omega_2^{-j},$ where $a\in A.$ Therefore, 
$y=a\Omega_1^{-i}\Omega_2^{-j}P(\Omega_1,\Omega_2), $ which implies that $y\Omega_1^i\Omega_2^j=aP(\Omega_1,\Omega_2).$ Choose  $(i,j)\in \mathbb{N}^2$ minimal (in the lexicographic order on $\mathbb{N}^2$) such that the equality holds. Without loss of generality,  suppose that $i>0,$ then 
$aP(\Omega_1,\Omega_2)\in \langle\Omega_1\rangle.$ Given that $\langle\Omega_1\rangle$ is a completely prime ideal, it implies that $a\in \langle\Omega_1\rangle$ or $P(\Omega_1,\Omega_2)\in \langle\Omega_1\rangle.$ Since $P(\Omega_1,\Omega_2)\in \mathcal{P},$ it implies that  $P(\Omega_1,\Omega_2)\not\in \langle\Omega_1\rangle,$ hence $a\in \langle \Omega_1\rangle.$ This further implies that $a=t\Omega_1,$ where $ t\in A.$ Returning to $y\Omega_1^i\Omega_2^j=aP(\Omega_1,\Omega_2),$ we have that $y\Omega_1^i\Omega_2^j=t\Omega_1 P(\Omega_1,\Omega_2).$ Therefore, $y\Omega_1^{i-1}\Omega_2^j=tP(\Omega_1,\Omega_2).$ This clearly contradicts the minimality of $(i,j),$ hence $(i,j)=(0,0),$ and 
$y=aP(\Omega_1,\Omega_2)\in\langle P(\Omega_1,\Omega_2)\rangle.$ Consequently,
$\langle P(\Omega_1,\Omega_2)\rangle_{R}\cap A=\langle P(\Omega_1,\Omega_2)\rangle$ for all $P(\Omega_1,\Omega_2)\in \mathcal{P}$ as desired.

Similarly, we show that  $\langle \Omega_1-\alpha,\Omega_2-\beta\rangle_R\cap A=\langle \Omega_1-\alpha,\Omega_2-\beta\rangle;$ $\forall \alpha,\beta\in \k^*.$ Fix $ \alpha,\beta\in \k^*.$ Observe that 
$\langle \Omega_1-\alpha,\Omega_2-\beta\rangle\subseteq\langle \Omega_1-\alpha,\Omega_2-\beta\rangle_R\cap A.$ We establish the reverse inclusion. Let $y\in \langle \Omega_1-\alpha,\Omega_2-\beta\rangle_R \cap A.$ Since $y\in \langle \Omega_1-\alpha,\Omega_2-\beta\rangle_R,$ 
there exist 
 $i,j\in\mathbb{N}$ such that   
$ y\Omega_1^{i}\Omega_2^{j}=m(\Omega_1-\alpha)+n(\Omega_2-\beta),$ where 
$m,n\in A.$  Choose  $(i,j)\in \mathbb{N}^2$ minimal (in the lexicographic order on $\mathbb{N}^2$) such that the equality holds. Without loss of generality, suppose that 
$i>0$ and let $f:A\longrightarrow {A}/{\langle \Omega_2-\beta\rangle}$ be a canonical surjection. 
We have that $f(y)f(\Omega_1)^if(\Omega_2)^j=f(m)f(\Omega_1-\alpha).$ It follows that
$f(m)f(\Omega_1-\alpha)\in \langle f(\Omega_1)\rangle$ and so $\alpha f(m) \in \langle f(\Omega_1)\rangle$. 
Observe that $f(\Omega_1-\alpha)\not\in \langle f(\Omega_1)\rangle,$  hence $f(m)\in \langle f(\Omega_1)\rangle.$  As $\alpha \neq 0$, we obtain the existence of $\lambda\in A$ such that $f(m)=f(\lambda)f(\Omega_1).$
Consequently, $f(y)f(\Omega_1)^if(\Omega_2)^j=f(\lambda)f(\Omega_1)f(\Omega_1-\alpha).$  Since
$f(\Omega_1)\neq 0,$ it implies that $f(y)f(\Omega_1)^{i-1}f(\Omega_2)^j=f(\lambda)f(\Omega_1-\alpha).$  Therefore, $y\Omega_1^{i-1}\Omega_2^j=\lambda(\Omega_1-\alpha)+\lambda^\prime(\Omega_2-\beta)$ for some $\lambda^\prime\in A.$  This contradicts the minimality of $(i,j).$ Hence, $(i,j)=(0,0)$ and so $y=m(\Omega_1-\alpha)+n(\Omega_2-\beta)\in \langle \Omega_1-\alpha,\Omega_2-\beta\rangle.$ In conclusion, $\langle \Omega_1-\alpha,\Omega_2-\beta\rangle_R\cap A=\langle \Omega_1-\alpha,\Omega_2-\beta\rangle,$ $\forall \alpha,\beta\in \k^*.$  
\end{proof}

Using similar techniques, we obtain the following description for the $\ch$-strata of $\langle\Omega_1\rangle$ and $\langle\Omega_2\rangle$.

\begin{proposition}
\label{p1}
Assume $\k$ is algebraically closed.
\begin{enumerate} 
\item $\spec_{\langle\Omega_1\rangle}(A)=
\{\langle\Omega_1\rangle\}\cup \{\langle\Omega_1,\Omega_2-\beta\rangle\mid \beta\in \k^*\}.$ 
\item $\spec_{\langle\Omega_2\rangle}(A)=
\{\langle\Omega_2\rangle\}\cup \{\langle\Omega_1-\alpha,\Omega_2\rangle\mid\alpha\in \k^*\}.$ 
\end{enumerate}
\end{proposition}

Since maximal ideals in their stratum are primitive for a QNA, we obtain the following result. 

\begin{corollary}
Assume $\k$ is algebraically closed and let  $(\alpha, \beta)\in \k^2\setminus \{(0,0)\}.$  The ideal $\langle \Omega_1-\alpha,\Omega_2-\beta\rangle$ of $A$ is  primitive.
\end{corollary}

\begin{remark}
The statement of the above corollary is still valid without the assumption that $\k$ is algebraically closed. The proof is actually similar as we only use this assumption to get a full description of the strata we were interested in.  
\end{remark}

We can actually prove a stronger result. 

\begin{theorem}
\label{c3p29}
Let $(\alpha, \beta)\in \k^2\setminus \{(0,0)\}.$ The prime ideal $\langle \Omega_1-\alpha,\Omega_2-\beta\rangle$ of $A$ is maximal. 
\end{theorem}
\begin{proof}
Let $(\alpha, \beta)\in \k^2\setminus \{(0,0)\}.$
Suppose that there exists a maximal ideal $I$ of $A$ such that $\langle \Omega_1-\alpha,\Omega_2-\beta\rangle\varsubsetneq I\varsubsetneq A.$ Let $J$ be the $\mathcal{H}-$invariant prime ideal in $A$ such that $I\in \spec_J(A).$ 

We claim that $J$ cannot be $\langle 0\rangle,$  $\langle\Omega_1\rangle$ or $\langle\Omega_2\rangle$. For instance, if $\alpha, \beta \neq 0$, then $J$ cannot be equal to $\langle 0\rangle$ since in this case $\langle \Omega_1-\alpha,\Omega_2-\beta\rangle $ is maximal in the $0$-stratum. Moreover, $J \neq  \langle\Omega_1\rangle$ as otherwise $I$ would contain $\alpha=\Omega_1-(\Omega_1-\alpha)$, a contradiction. The other cases are similar and left to the reader. 

This means that $J$ is an $\ch$-prime of height at least equal to 2. As the poset of $\ch$-primes is isomorphic to $W$, this forces  $J$ to contain both $\Omega_1$ and $\Omega_2.$ Moreover, since $ J\subseteq I,$ it implies that $\Omega_1, \Omega_2\in I.$  Given that $\langle \Omega_1-\alpha,\Omega_2-\beta\rangle\subset I$, we have that $\Omega_1-\alpha,\Omega_2-\beta\in I.$ It follows that $\alpha, \beta\in I,$ hence   $I=A,$ a contradiction!  This confirms that $\langle \Omega_1-\alpha,\Omega_2-\beta\rangle$  is a maximal ideal in $A.$
\end{proof}


\section{Simple quotients of $U_q^+(G_2)$ and their relation to the second Weyl algebra}
\label{c3s1}
Now that we have found primitive ideals of $A=U_q^+(G_2),$ we are going to study the corresponding simple quotient algebras. In view of Dixmier's theorem, we consider these simple quotients as deformations of a Weyl algebra (of appropriate dimension), and so we compare their properties with some known properties of the Weyl algebras. In this section, we prove that the Gelfand-Kirillov dimension of $A_{\alpha,\beta}$ is 4 and consequently prove that the height of the maximal ideal $\langle \Omega_1-\alpha,\Omega_2-\beta\rangle$ is $2$ as expected.   Then we focus on describing a linear basis of $A_{\alpha,\beta}$; we use this basis in the following section to study the derivations of $A_{\alpha,\beta}$.  Finally, we show  that with appropriate choices of $\alpha$ and $\beta,$ the algebra $A_{\alpha,\beta}$ is a quadratic extension of the second Weyl algebra $A_2(\k)$ at $q=1.$ 

Recall from Theorem \ref{c3p29} that  $\Omega_1-\alpha$ and 
$\Omega_2-\beta,$ where $(\alpha,\beta)\in \k^2\setminus \{(0,0)\},$ generate a maximal ideal of $A.$ As a result,  the corresponding quotient 
$$A_{\alpha,\beta}:=\frac{A}{\langle \Omega_1-\alpha,\Omega_2-\beta\rangle}$$ 
is a simple noetherian domain. 
Denote the canonical images of $E_i$ in $A_{\alpha,\beta}$ by $e_i:=E_i+\langle\Omega_1-\alpha, \Omega_2-\beta\rangle$  for all 
$1\leq i\leq 6$. The algebra $A_{\alpha,\beta}$ satisfies the following relations:

\begin{align*}
e_2e_1&=q^{-3}e_1e_2& e_3e_1&=q^{-1}e_1e_3-(q+q^{-1}+q^{-3})e_2\\
e_3e_2&=q^{-3}e_2e_3&e_4e_1&=e_1e_4+(1-q^2)e_3^2\\
e_4e_2&=q^{-3}e_2e_4-\frac{q^4-2q^2+1}{q^4+q^2+1}e_3^3&e_4e_3&=q^{-3}e_3e_4\\
e_5e_1&=qe_1e_5-(1+q^2)e_3&e_5e_2&=e_2e_5+(1-q^2)e_3^2\\
e_5e_3&=q^{-1}e_3e_5-(q+q^{-1}+q^{-3})e_4&e_5e_4&=q^{-3}e_4e_5\\
e_6e_1&=q^3e_1e_6-q^3e_5&e_6e_2&=q^3e_2e_6+(q^4+q^2-1)e_4+(q^2-q^4)e_3e_5\\
e_6e_3&=e_3e_6+(1-q^2)e_5^2&e_6e_4&=q^{-3}e_4e_6-\frac{q^4-2q^2+1}{q^4+q^2+1}e_5^3\\
e_6e_5&=q^{-3}e_5e_6,
\end{align*} and
\begin{align}
e_1e_3e_5+ae_1e_4+ae_2e_5+a'e_3^2&=\alpha,\label{e1e}\\
e_2e_4e_6+be_2e_5^3+be_3^3e_6+b'e_3^2e_5^2+c'e_3e_4e_5+d'e_4^2&=\beta.\label{e2e}
\end{align}

The following additional relations of $A_{\alpha,\beta}$ in the lemma below will be helpful when computing linear basis for $A_{\alpha,\beta}$.  Note, we put constant coefficients of monomials in a square bracket $[ \ ]$ in order to distinguish them from monomials easily. These constants are defined in Appendix \ref{appenc}. 
\begin{lemma}
\label{c3l2}
\begin{align*}
 (1)\ e_3^2=&c_1\alpha+[c_2]e_2e_5+[c_2]e_1e_4+
[c_3]e_1e_3e_5.\\
\\
(2)\ e_4^2=&b_1\beta+[b_2]e_2e_4e_6
+[b_3]e_2e_5^3+[b_4\alpha]e_3e_6+[b_5]e_2e_3
e_5e_6+[b_6]e_1e_3e_4e_6
\\& +[b_7\alpha]e_1e_5e_6+[b_8]e_1e_2e_5^2
e_6+[b_9]e_1^2e_4e_5e_6+[b_{10}]
e_1^2e_3e_5^2e_6+[b_{11}\alpha]e_5^2\\&+
[b_{12}]e_1e_3e_5^3+
[b_{13}]e_3e_4e_5+[b_{14}]e_1e_4e_5
^2.\\
\\
(3)\ e_3^2e_4=& [c_1\alpha] e_4+[q^{-3}c_2]e_2e_4e_5
+[c_2b_4\alpha]e_1e_3e_6
+[b_{15}]e_1e_3e_4e_5
+[\beta b_1c_2]e_1\\&+[c_2b_3]e_1e_2e_5^3+[c_2b_5]e_1e_2e_3
e_5e_6+[c_2b_6]e_1^2e_3e_4e_6+
[c_2b_7\alpha]e_1^2e_5e_6\\&+[c_2b_{11}\alpha]e_1e_5^2+[c_2b_{12}]e_1^2e_3
e_5^3+[c_2b_8]e_1^2e_2e_5
^2e_6+[c_2b_9]e_1^3e_4e_5e_6\\
&+[c_2b_{10}]e_1^3
e_3e_5^2e_6+[c_2b_{14}]e_1^2
e_4e_5^2
+[c_2b_2]e_1e_2e_4e_6.\\
\\ 
(4) \ e_3e_4^2=&
[\beta b_1]e_3+[k_1]e_2e_3
e_4e_6+
[k_2]e_2e_3e_5^3
+[k_3\alpha^2] e_6+[k_4\alpha]
e_2e_5e_6+[k_5\alpha]e_1e_4
e_6
\\&+[k_6\alpha]e_1e_3e_5
e_6+[k_7]e_2
^2e_5^2e_6+[k_8 \beta] e_1e_5
+[k_9]
e_1e_2e_4e_5e_6+
[k_{10}]e_1^3e_2e_5^2e_6^2
\\&
+
[k_{11}\alpha] e_4
e_5+[k_{12}\alpha]e_1^3e_5e_6^2+[k_{13}]
e_1^2e_3
e_4e_5e_6+[k_{14}\beta]e_1^2e_6
+[b_{11}\alpha]e_3e_5^2
\\&+[k_{15}]e_1^2
e_2e_4e_6^2+[k_{16}]e_1^2e_2
e_5^3e_6+[k_{17}\alpha]e_1^2e_3e_6^2+
[k_{18}]e_1^2e_2e_3e_5e_6^2
+[k_{19}]e_1^3e_3e_4e_6^2
\\&+[k_{20}]e_1^4e_4e_5e_6^2+
[k_{21}\alpha]e_1^2
e_5^2e_6+[k_{22}]e_1^4e_3
e_5^2e_6^2+[k_{23}]
e_1^3e_3e_5^3e_6
+[k_{24}]e_1^3
e_4e_5^2e_6
\\&
+[k_{25}]e_1e_5^3+[k_{26}]e_1^2e_4
e_5^3+[k_{27}]e_1^2e_3e_5^4+[k_{28}]e_2e_4e_5^2+[k_{29}]e_1e_3e_4
e_5^2
\\&+[k_{30}]e_1e_2
e_5^4+[k_{31}]
e_1e_2e_3e_5^2e_6.
\end{align*}
\end{lemma}
\begin{proof}
This is proved by brute-force computation, left to the reader.
\end{proof}

\subsection{Gelfand-Kirillov dimension of $A_{\alpha,\beta}$}
\label{c3p3}

We refer the reader to \cite{gt} for background on the Gelfand-Kirillov dimension (GKdim for short).

Assume first that $\alpha,\beta \neq 0.$ Recall from Section \ref{c2s1} that $R_1=\k_{q^M}[T_1^{\pm1},\cdots,T_6^{\pm1}]$ is the quantum torus associated to the quantum affine space $\overline{A}=A^{(2)}$. Also, $\Omega_1=T_1T_3T_5$ and 
$\Omega_2=T_2T_4T_6$ in $\overline{A}.$
It follows from \cite[Theorem 5.4.1]{ca} that there exists an Ore set $S_{\alpha,\beta}$ in $A_{\alpha,\beta}$ such that
$A_{\alpha,\beta}S^{-1}_{\alpha,\beta}\cong {R_1}/{\langle T_1T_3T_5-\alpha,T_2T_4T_6-\beta\rangle}.$

Now, set $$\mathscr{A}_{\alpha,\beta}:=\frac{R_1}{\langle T_1T_3T_5-\alpha,T_2T_4T_6-\beta\rangle}.$$
Let  $t_i:=T_i+\langle T_1T_3T_5-\alpha, T_2T_4T_6-\beta\rangle$ denote the canonical images of the generators $T_i$ of  $R_1$ in $\mathscr{A}_{\alpha,\beta}.$ The algebra $\mathscr{A}_{\alpha,\beta}$ is generated by $t_1^{\pm 1}, \cdots, t_6^{\pm 1} $ subject to the following  relations:
\begin{align*}
t_it_j&=q^{\mu_{ij}}t_jt_i&t_1&=\alpha t_5^{-1}t_3^{-1}& t_2&=\beta t_6^{-1}t_4^{-1},
\end{align*}
for all $1\leq i,j\leq 6$ and $ M=(\mu_{ij})$ (the skew-symmetric matrix in Section \ref{c2s1}).
 Observe that 
$\mathscr{A}_{\alpha,\beta}\cong \k_{q^N}[t_3^{\pm1},t_4^{\pm1},t_5^{\pm1},t_6^{\pm 1}],$ where 
the skew-symmetric matrix $N$ can easily be deduced from $M$ (by deleting the first two rows and columns) as follows: 
\begin{align*}
N:=\begin{bmatrix}
0&3&1&0\\
-3&0&3&3\\
-1&-3&0&3\\
0&-3&-3&0
\end{bmatrix}.
\end{align*}
   
 Secondly, suppose that $\alpha=0$ and $\beta\neq 0.$\newline
  Then, $A_{0,\beta}S^{-1}_{0,\beta}\cong \mathscr{A}_{0,\beta}= {R_1}/{\langle T_1,T_2T_4T_6-\beta\rangle}.$ 
The algebra  $\mathscr{A}_{0,\beta}$ is generated by $t_2^{\pm 1}, \cdots, t_6^{\pm 1}$ subject to the relations
\begin{align*}
t_it_j&=q^{\mu_{ij}}t_jt_i&t_1&=0& t_2&=\beta t_6^{-1}t_4^{-1},
\end{align*}
for all $1\leq i,j\leq 6$ and $ \mu_{ij}\in M.$ We also have that 
$\mathscr{A}_{0,\beta}\cong \k_{q^N}[t_3^{\pm1},t_4^{\pm1},t_5^{\pm1},t_6^{\pm 1}].$

Finally, when $\alpha\neq 0$ and $\beta=0,$ then one can also verify that
$\mathscr{A}_{\alpha,0}\cong \k_{q^N}[t_3^{\pm1},t_4^{\pm1},t_5^{\pm1},t_6^{\pm 1}].$

 As a result of the above discussion, in all cases, we have that  $A_{\alpha,\beta}S^{-1}_{\alpha,\beta}\cong 
\mathscr{A}_{\alpha,\beta}\cong \k_{q^N}[t_3^{\pm1},t_4^{\pm1},t_5^{\pm1},t_6^{\pm 1}].$   With a slight abuse of notation, we write 
$A_{\alpha,\beta}S^{-1}_{\alpha,\beta}=
\mathscr{A}_{\alpha,\beta}= \k_{q^N}[t_3^{\pm1},t_4^{\pm1},t_5^{\pm1},t_6^{\pm 1}]$ for all 
$(\alpha,\beta)\in \k^2\setminus \{(0,0)\}.$ 
It follows from  \cite[Theorem 6.3]{gll} that $\gkdim (A_{\alpha,\beta})=\gkdim (A_{\alpha,\beta}S^{-1}_{\alpha,\beta})=\gkdim (\mathscr{A}_{\alpha,\beta})=4.$
Since Tauvel's height formula holds in $A=U_q^+(G_2)$ \cite{gll}, we have that 
$\gkdim (A)= \height (\langle \Omega_1-\alpha,\Omega_2-\beta \rangle) \ + \gkdim (A_{\alpha,\beta}).$ 
Since  $\gkdim (A)=6$, we conclude that $ \height (\langle\Omega_1-\alpha,\Omega_2-\beta\rangle) =2$ for all $(\alpha,\beta)\in \k^2\setminus \{(0,0)\}.$

\begin{proposition}
$\gkdim (A_{\alpha,\beta})=4$ for all $(\alpha, \beta) \neq (0,0)$.
\end{proposition} 


\subsection{Linear basis for  $A_{\alpha,\beta}$} 
\label{c3s2} 
Set $A_{\beta}:={A}/{\langle \Omega_2-\beta\rangle},$ where $\beta\in\k.$   Now, denote the canonical images of 
$E_i$ by  $\widehat{e}_i:=E_i+\langle \Omega_2-\beta\rangle$ in $A_{\beta}.$  Clearly,  $A_{\alpha,\beta}
\cong {A_{\beta}}/{\langle \widehat{\Omega}_1-\alpha\rangle}.$ As a result, one can identify $A_{\alpha,\beta}$ with ${A_{\beta}}/{\langle \widehat{\Omega}_1-\alpha\rangle}.$ 
Moreover, the algebra $A_{\beta}$ satisfies the relations of $A=U_q^+(G_2)$ and
\begin{align}
\widehat{e_2}\widehat{e_4}\widehat{e_6}+b\widehat{e_2}\widehat{e_5}^3+
b\widehat{e_3}^3\widehat{e_6}+b'\widehat{e_3}^2\widehat{e_5}^2+c'
\widehat{e_3}\widehat{e_4}\widehat{e_5}+d'\widehat{e_4}^2=\beta.\label{c3eo}
\end{align}
From Propositions \ref{p0} and \ref{p1}, one can conclude that 
$\langle \Omega_2-\beta\rangle$ is  a completely prime ideal (since it is a prime ideal) of $A$ for all $\beta\in \k.$  Hence, the algebra $A_{\beta}$ is a noetherian domain.

We are now going to find a linear basis for $A_{\alpha,\beta},$ where $(\alpha,\beta)\in \k^2\setminus \{(0,0)\}.$  Since $A_{\alpha,\beta}$ is identified with ${A_{\beta}}/{\langle \widehat{\Omega}_1-\alpha\rangle},$ we will first and foremost find a basis for ${A_{\beta}},$ and then proceed to find a basis for $A_{\alpha,\beta}.$ Note, the relations in  Lemma \ref{c3l3} are also valid in $A_{\beta}$ and $A_{\alpha,\beta},$ and are going to be very useful in this section.

\begin{proposition}
The set $\mathfrak{S}=\{\widehat{e_1}^i\widehat{e_2}^j\widehat{e_3}^k\widehat{e_4}^{\xi}\widehat{e_5}^l\widehat{e_6}^m\mid
i,j,k,l,m\in \mathbb{N} \ \text{and} \ \xi=0,1\}$ is a $\k-$basis of $A_\beta.$  \label{c3p4}
\end{proposition}

\begin{proof}
Since the family $(\Pi_{s=1}^6E_s^{i_s})_{i_s\in\mathbb{N}}$ is a PBW-basis of $A$ over $\k,$ it follows that the family $(\Pi_{s=1}^6\widehat{e_s}^{i_s})_{i_s\in\mathbb{N}}$ is a spanning set of $A_{\beta}$ over $\k.$ We want to show that  $\mathfrak{S}$ spans $A_{\beta}.$  We do this by showing that $\Pi_{s=1}^6\widehat{e_s}^{i_s}$ can be written as a finite linear combination of 
the elements of $\mathfrak{S}$  for all $i_1,\cdots,i_6\in \mathbb{N}$ 
 by an induction on $i_4.$  The result is obvious when $i_4=0$ or $1.$ 
For $i_4\geq 1,$  assume that 
 $$\Pi_{s=1}^6\widehat{e_s}^{i_s}=
\sum_{(\xi,\underline{v})\in I}a_{(\xi,\underline{v})}\widehat{e_1}^i\widehat{e_2}^j\widehat{e_3}^k\widehat{e_4}^\xi \widehat{e_5}^l
\widehat{e_6}
^m,$$ where
$\underline{v}:=(i,j,k,l,m)\in \mathbb{N}^5$ and  $a_{(\xi,\underline{v})}$ are all scalars.  Note, $I$ is a finite subset of $\{0,1\}\times \mathbb{N}^5.$  It follows from   the commutation relations of $A_{\beta}$ (see Lemma \ref{c3l3}) that
$$\widehat{e_1}^{i_1}\widehat{e_2}^{i_2}\widehat{e_3}^{i_3}\widehat{e_4}^{i_4+1}
\widehat{e_5}^{i_5}\widehat{e_6}^{i_6}=q^{\bullet}\Pi_{s=1}^6\widehat{e_s}^{i_s}\widehat{e_4}-q^{\bullet}d_1[i_6]\widehat{e_1}^{i_1}\widehat{e_2}^{i_2}\widehat{e_3}^{i_3}
\widehat{e_4}^{i_4}\widehat{e_5}^{i_5+3}\widehat{e_6}^{i_6-1}.$$
From the inductive hypothesis, $\widehat{e_1}^{i_1}\widehat{e_2}^{i_2}\widehat{e_3}^{i_3}
\widehat{e_4}^{i_4}\widehat{e_5}^{i_5+3}\widehat{e_6}^{i_6-1}\in$ Span($\mathfrak{S}$). Hence, we proceed to show that $\Pi_{s=1}^6\widehat{e_s}^{i_s}\widehat{e_4}$ is also  in the span of $\mathfrak{S}.$ From the inductive hypothesis, we have
\begin{align*}
\Pi_{s=1}^6\widehat{e_s}^{i_s}\widehat{e_4}=\sum_{(\xi,\underline{v})\in I}a_{(\xi,\underline{v})}\widehat{e_1}^i\widehat{e_2}^j\widehat{e_3}^k\widehat{e_4}^\xi \widehat{e_5}^l
\widehat{e_6}^m
\widehat{e_4}.
\end{align*}
Using the commutation relations in Lemma \ref{c3l3}, we have that
\begin{align*}
\Pi_{s=1}^6\widehat{e_s}^{i_s}\widehat{e_4}
=&\sum_{(\xi,\underline{v})\in I}q^\bullet a_{(\xi,\underline{v})}\widehat{e_1}^i\widehat{e_2}^j\widehat{e_3}^k\widehat{e_4}^{\xi+1} \widehat{e_5}^l
\widehat{e_6} ^m
\\&+\sum_{(\xi,\underline{v})\in I} q^\bullet d_1[m]a_{(\xi,\underline{v})}\widehat{e_1}^i\widehat{e_2}^j\widehat{e_3}^k
\widehat{e_4}^{\xi}\widehat{e_5}^{l+3}
\widehat{e_6}^{m-1}.
\end{align*}
All the terms in the above expression belong to the span of $\mathfrak{S}$  except $\widehat{e_1}^i\widehat{e_2}^j\widehat{e_3}^k\widehat{e_4}^2\widehat{e_5}^l
 \widehat{e_6}^m.$  
From  \eqref{c3eo}, we have that
\begin{align}
\widehat{e_4}^2=\beta_0\widehat{e_2}\widehat{e_4}\widehat{e_6}+b\beta_0\widehat{e_2}
\widehat{e_5}^3+
b\beta_0\widehat{e_3}^3\widehat{e_6}
+b'\beta_0\widehat{e_3}^2\widehat{e_5}^2+c'\beta_0\widehat{e_3}\widehat{e_4}\widehat{e_5}
-\beta\beta_0,\label{c3e3}
\end{align}
where $\beta_0=-1/d'.$ Substituting \eqref{c3e3} into
$\widehat{e_1}^i\widehat{e_2}^j\widehat{e_3}^k\widehat{e_4}^2\widehat{e_5}^l
\widehat{e_6}^m,$ 
one can easily verify that $$\widehat{e_1}^i\widehat{e_2}^j\widehat{e_3}^k\widehat{e_4}^2\widehat{e_5}^l
\widehat{e_6}^m\in  \mathrm{Span}(\mathfrak{S}).$$ 
Therefore, $\widehat{e_1}^{i_1}\widehat{e_2}^{i_2}\widehat{e_3}^{i_3}\widehat{e_4}^{i_4+1}
\widehat{e_5}^{i_5}\widehat{e_6}^{i_6}$
 can be written as a finite linear combination of the elements of $\mathfrak{S}$ over $\k$ for all $i_1,\cdots,i_6\in \mathbb{N}.$ By the principle of mathematical induction, $\mathfrak{S}$ is a spanning set of $A_\beta$ over $\k.$\\

Next we show that $\mathfrak{S}$ is a linearly independent set. 
Suppose that
$$\sum_{(\xi,\underline{v})\in I}a_{(\xi,\underline{v})}\widehat{e_1}^i\widehat{e_2}^j\widehat{e_3}^k\widehat{e_4}^\xi \widehat{e_5}^l
\widehat{e_6}
^m=0.$$ Since $A_{\beta}=A/{\langle\Omega_2-\beta\rangle}$, it implies that
$$\sum_{(\xi,\underline{v})\in I}a_{(\xi,\underline{v})} E_1^iE_2^jE_3^kE_4^\xi E_5^l
E_6^m=(\Omega_2-\beta)\nu,$$ with 
$\nu\in A.$ Write $\nu=\displaystyle\sum_{(i,\cdots,n)\in J}b_{(i,\cdots,n)}E_1^iE_2^jE_3^kE_4^lE_5^mE_6^n,$
where $J$ is a finite subset of $\mathbb{N}^6$ and  $b_{(i,\cdots,n)}$ are all scalars. 
It follows that
\begin{align}
\label{xx}
\sum_{(\xi,\underline{v})\in I}a_{(\xi,\underline{v})} E_1^iE_2^jE_3^kE_4^\xi E_5^l
E_6^m&=\displaystyle\sum_{(i,\cdots,n)\in J}b_{(i,\cdots,n)}E_1^iE_2^jE_3^k(\Omega_2-\beta)E_4^lE_5^mE_6^n.
\end{align}
Before we continue the proof, the following needs to be noted.
\begin{itemize}
\item[$\clubsuit$ ] 
Let $(i',j',k',l',m',n'),$  $(i,j,k,l,m,n)\in$ $ \mathbb{N}^6.$ We say that $(i,j,k,l,m,n)<_4 (i',j',k',l',m',n')$ if $[l<l']$ or $[l=l'$ and $i<i']$ or $[l=l', \ i=i'$ and $j<j']$ or $[l=l', \ i=i', \ j=j'$ and $k<k']$ or $[l=l', \ i=i', \ j=j', \ k=k'$ and $m<m']$ or $[l=l', \ i=i', \ j=j', \ k=k', \ m=m'$ and $n\leq n'].$ Note, the purpose of the square bracket $[ \ ]$ is to differentiate the options.
\end{itemize}
From  Section \ref{c2s1}, we have that $\Omega_2=E_2E_4E_6+bE_2E_5^3+bE_3^3E_6
+b'E_3^2E_5^2+c'E_3E_4E_5+d'E_4^2$ in $A=U_q^+(G_2).$ 
Now,
\begin{align*}
\sum_{(\xi,\underline{v})\in I}a_{(\xi,\underline{v})} E_1^iE_2^jE_3^kE_4^\xi E_5^l
E_6^m&=\displaystyle\sum_{(i,\cdots,n)\in J}b_{(i,\cdots,n)}E_1^iE_2^jE_3^k(\Omega_2-\beta)E_4^lE_5^mE_6^n
\\&=\sum_{(i,\cdots,n)\in J}d'b_{(i,\cdots,m)}E_1^iE_2^jE_3^kE_4^{l+2}E_5^mE_6^n + \text{LT}_{<_4},
\end{align*}
where $\text{LT}_{<_4}$ contains lower order terms with respect to 
$<_4$ (as in $\clubsuit$).  Moreover, $\text{LT}_{<_4}$ vanishes when $b_{(i,\cdots,n)}=0$ for all $(i,\cdots,n)\in J$ (one can easily confirm this by fully expanding the right hand side of \eqref{xx}).

Now, suppose that there exists $(i,j,k,l,m,n)\in J$ such that $b_{(i,j,k,l,m,n)}\neq 0.$ \newline
 Let $(i',j',k',l',m',n')$ be the greatest element of $J$ with respect to 
 $<_4$ (defined in $\clubsuit$ above)  such that $b_{(i',j',k',l',m',n')}\neq 0.$  Note, the family $(E_1^iE_2^jE_3^kE_4^lE_5^mE_6^n)_{(i,\cdots,n)\in \mathbb{N}^6}$ is a basis of $A.$ Since $\text{LT}_{<_4}$ contains lower order terms,  identifying the coefficients of $E_1^{i'}E_2^{j'}E_3^{k'}E_4^{l'+2}E_5^{m'}E_6^{n'}$ in the above equality, we have that
$d'b_{(i',\cdots,n')}=0.$ Since $b_{(i',j',k',l',m',n')}\neq 0,$ it follows that $d'={q^{12}}/(q^6-1)=0,$ a contradiction (see Appendix \ref{appenc} for the expression of $d'$). As a result, $b_{(i,j,k,l,m,n)}=0$ for all 
 $(i,j,k,l,m,n)\in J.$ 
 Therefore,
 $\sum_{(\xi,\underline{v})\in I}a_{(\xi,\underline{v})} E_1^iE_2^jE_3^kE_4^\xi E_5^l
E_6^m=0.$ Since $(E_1^iE_2^jE_3^kE_4^lE_5^mE_6^n)_{(i,\cdots,n)\in\mathbb{N}^6}$ is a basis of  $A,$ it follows that $a_{(\xi,\underline{v})}=0$ for all $(\xi,\underline{v})\in I.$   In conclusion, $\mathfrak{S}$ is  a linearly independent set and hence forms a basis of $A_\beta$ as desired. 
\end{proof}

\begin{proposition}
\label{c3p5}
Let $(\alpha,\beta)\in \k^2\setminus \{(0,0)\}.$ The set  $\mathcal{B}=\{e_1^ie_2^je_3^{\epsilon_1}e_4^{\epsilon_2}e_5^ke_6^l\mid i,j,k,l\in \mathbb{N} \ \text{and} \ \epsilon_1,\epsilon_2\in \{0,1\} \}$ is a $\k-$basis of $A_{\alpha,\beta}.$ 
\end{proposition}
\begin{proof}
Since the set  $\mathfrak{S}=\{\widehat{e_1}^{i_1}\widehat{e_2}^{i_2}\widehat{e_3}^{i_3}\widehat{e_4}^{\xi}
\widehat{e_5}^{i_5}\widehat{e_6}^{i_6}\mid i_1,i_2,i_3,i_5,i_6\in\mathbb{N} \ \text{and} \ \xi=0,1\}
$ is a $\k-$basis of $A_\beta$ (Proposition \ref{c3p4}), and $A_{\alpha,\beta}$ is identified with ${A_\beta}/{\langle \widehat{\Omega}_1-\alpha\rangle},$ it follows that  $$(e_1^{i_1}e_2^{i_2}e_3^{i_3}e_4^\xi e_5^{i_5}e_6^{i_6})_{i_1,i_2,i_3,i_5,i_6\in \mathbb{N},  \ \xi\in \{0,1\}}$$ is a spanning set of $A_{\alpha,\beta}$ over $\k.$  We want to show that $\mathcal{B}$ spans $A_{\alpha,\beta}$ by showing that $e_1^{i_1}e_2^{i_2}e_3^{i_3}e_4^\xi e_5^{i_5}e_6^{i_6}$ can be written as a finite linear combination of the elements of $\mathcal{B}$ over $\k$ for all $i_1,i_2,i_3,i_5,i_6\in\mathbb{N}$ and $\xi=0,1.$ By Proposition \ref{c3p4}, it is sufficient to do this by an induction on $i_3.$ The result is obvious when  $i_3=0$ or $1.$  For $ i_3\geq 1,$ suppose that
\begin{align*}
e_1^{i_1}e_2^{i_2}e_3^{i_3}e_4^\xi e_5^{i_5}e_6^{i_6}
&=\sum_{(\epsilon_1,\epsilon_2,\underline{\upsilon})\in I}a_{(\epsilon_1,\epsilon_2,\underline{\upsilon})}
e_1^ie_2^je_3^{\epsilon_1}e_4^{\epsilon_2}e_5^ke_6^l,
\end{align*}
where $\underline{\upsilon}=(i,j,k,l)\in \mathbb{N}^4,$ \ $I$ is a finite subset of $\{0,1\}^2\times \mathbb{N}^4$, and $(a_{(\epsilon_1,\epsilon_2,\underline{\upsilon})})_{(\epsilon_1,\epsilon_2,\underline{\upsilon})\in I}$ is a family of scalars. Using the commutation relations in $A_{\alpha,\beta}$  (Lemma \ref{c3l3}), we have that:
$$e_1^{i_1}e_2^{i_2}e_3^{i_3+1}e_4^\xi e_5^{i_5}e_6^{i_6}
=q^{\bullet}e_3e_1^{i_1}e_2^{i_2}e_3^{i_3}e_4^\xi e_5^{i_5}e_6^{i_6}-q^{\bullet}d_2[i_1]e_1^{i_1-1}e_2^{1+i_2}e_3^{i_3}
e_4^\xi e_5^{i_5}e_6^{i_6}.$$
From the inductive hypothesis, $e_1^{i_1-1}e_2^{1+i_2}e_3^{i_3}
e_4^\xi e_5^{i_5}e_6^{i_6}\in$ Span($\mathcal{B}$) for all $i_1>0$ (note: $d_2[0]=0$). As a result, we proceed to show that
$e_3e_1^{i_1}e_2^{i_2}e_3^{i_3}e_4^\xi e_5^{i_5}e_6^{i_6}$ is also in the span of $\mathcal{B}.$ It follows from the inductive hypothesis that

\begin{align*}
e_3e_1^{i_1}e_2^{i_2}e_3^{i_3}e_4^\xi e_5^{i_5}e_6^{i_6}=&
\sum_{(\epsilon_1,\epsilon_2,\underline{\upsilon})\in I}a_{(\epsilon_1,\epsilon_2,\underline{\upsilon})}
e_3e_1^ie_2^je_3^{\epsilon_1}e_4^{\epsilon_2}e_5^ie_6^l\\
=&\sum_{(\epsilon_1,\epsilon_2,\underline{\upsilon})\in I}a_{(\epsilon_1,\epsilon_2,\underline{\upsilon})}
e_1^ie_2^je_3^{\epsilon_1+1}e_4^{\epsilon_2}e_5^ke_6^l\\
&+\sum_{(\epsilon_1,\epsilon_2,\underline{\upsilon})\in I}d_2[i]a_{(\epsilon_1,\epsilon_2,\underline{\upsilon})}
e_1^{i-1}e_2^{1+j}e_3^{\epsilon_1}e_4^{\epsilon_2}e_5^ke_6^l.
\end{align*}
Clearly, the monomial $e_1^{i-1}e_2^{1+j}e_3^{\epsilon_1}e_4^{\epsilon_2}e_5^ke_6^l$  belongs to the span of $\mathcal{B}$ for all $(\epsilon_1,\epsilon_2,\underline{\upsilon})\in I$  (with $i>0$). Again, the monomial
$e_1^ie_2^je_3^{\epsilon_1+1}e_4^{\epsilon_2}e_5^ke_6^l$ belongs to the span of 
$\mathcal{B}$ for all $(\epsilon_1,\epsilon_2,\underline{\upsilon})\in I;$ 
with $(\epsilon_1,\epsilon_2)=(0,0), (0,1).$ For $(\epsilon_1, \epsilon_2)
=(1,0), (1,1);$ we must show that $e_1^ie_2^je_3^2e_5^ke_6^l$ and 
$e_1^ie_2^je_3^{2}e_4e_5^ke_6^l$  belong to the span of $\mathcal{B}.$ 
From Lemma \ref{c3l2}, one can write $e_1^ie_2^je_3^2e_5^ke_6^l$ and $
e_1^ie_2^je_3^2e_4e_5^ke_6^l$ as finite linear combinations of the elements of $\mathcal{B}$ over $\k.$ Hence,  $e_1^ie_2^je_3^2e_5^ke_6^l$ and $
e_1^ie_2^je_3^2e_4e_5^ke_6^l$ belong to the span of $\mathcal{B}$ for all $(\epsilon_1,\epsilon_2,\underline{\upsilon})\in I;$ with  $(\epsilon_1,\epsilon_2)
=(1,0), (1,1).$ We have therefore established that
$e_3e_1^{i_1}e_2^{i_2}e_3^{i_3}e_4^\xi e_5^{i_5}e_6^{i_6}\in$ Span($\mathcal{B}$).
Consequently, each $e_1^{i_1}e_2^{i_2}e_3^{i_3+1}e_4^\xi e_5^{i_5}e_6^{i_6}$ belongs to the span of $\mathcal{B}.$ By the principle of mathematical induction, $\mathcal{B}$ is a spanning set of $A_{\alpha,\beta}$ over $\k$ as expected.\\

Finally, we show that $\mathcal{B}$ is a linearly independent set. 
Suppose that
\begin{align*}
\sum_{(\epsilon_1,\epsilon_2,\underline{\upsilon})\in I}a_{(\epsilon_1,\epsilon_2,\underline{\upsilon})}
e_1^ie_2^je_3^{\epsilon_1}e_4^{\epsilon_2}e_5^ke_6^l=0.
\end{align*}
In $A_\beta,$ we have that 
\begin{align*}
\sum_{(\epsilon_1,\epsilon_2,\underline{\upsilon})\in I}a_{(\epsilon_1,\epsilon_2,\underline{\upsilon})}\widehat{e_1}^i\widehat{e_2}^j\widehat{e_3}^{\epsilon_1}\widehat{e_4}^{\epsilon_2}
\widehat{e_5}^k\widehat{e_6}^l
&= (\widehat{\Omega}_1-\alpha)\nu,
\end{align*}
with $\nu\in A_{\beta}.$ One can write
$\nu$ in terms of the basis $\mathfrak{S}$ of $A_{\beta}$ (Proposition \ref{c3p4}) as:
$$\nu=\displaystyle\sum_{\underline{w}\in J_1}b_{\underline{w}}\widehat{e_1}^i\widehat{e_2}^j\widehat{e_3}^k
\widehat{e_4}\widehat{e_5}^m
 \widehat{e_6}^n+\displaystyle\sum_{\underline{w}\in J_2}c_{\underline{w}}\widehat{e_1}^i\widehat{e_2}^j\widehat{e_3}^k
\widehat{e_5}^l
 \widehat{e_6}^m,$$
 where  $b_{\underline{w}}$ and $c_{\underline{w}}$ are all scalars and $w:=(i,j,k,l,m)$.
 Hence,
 \begin{align*}
 \sum_{(\epsilon_1,\epsilon_2,\underline{\upsilon})\in I}a_{(\epsilon_1,\epsilon_2,\underline{\upsilon})}
\widehat{e_1}^i\widehat{e_2}^j\widehat{e_3}^{\epsilon_1}\widehat{e_4}^{\epsilon_2}
\widehat{e_5}^k\widehat{e_6}^l=&\sum_{\underline{w}\in J_1}b_{\underline{w}}\widehat{e_1}^i\widehat{e_2}^j\widehat{e_3}^k(\widehat{\Omega}_1-\alpha)
\widehat{e_4}\widehat{e_5}^l
 \widehat{e_6}^m\\&+\displaystyle\sum_{\underline{w}\in J_2}c_{\underline{w}}\widehat{e_1}^i\widehat{e_2}^j\widehat{e_3}^k
(\widehat{\Omega}_1-\alpha)\widehat{e_5}^l
 \widehat{e_6}^m.
 \end{align*}
Note, $\widehat{\Omega}_1=\widehat{e_1}\widehat{e_3}\widehat{e_5}+a\widehat{e_1}\widehat{e_4}+
 a\widehat{e_2}\widehat{e_5}+a'\widehat{e_3}^2.$
Using  \eqref{c3e3} and the relation $e_3^ke_1=q^{-k}e_1e_3^k+d_2[k]e_2e_3^{k-1}$ (see Lemma \ref{c3l3}), one can verify that  the above equality can be written in terms of the basis of $A_{\beta}$ (Propositions \ref{c3p4}) as:
\begin{align}
\label{aa}
\sum_{(\epsilon_1,\epsilon_2,\underline{\upsilon})\in I}a_{(\epsilon_1,\epsilon_2,\underline{\upsilon})}
\widehat{e_1}^i\widehat{e_2}^j\widehat{e_3}^{\epsilon_1}\widehat{e_4}^{\epsilon_2}
\widehat{e_5}^k\widehat{e_6}^l
 =&\sum_{\underline{w}\in J_1} b_{\underline{w}} a'\widehat{e_1}^i\widehat{e_2}^j\widehat{e_3}^{k+2}\widehat{e_4}
 \widehat{e_5}^l\widehat{e_6}^m\nonumber\\
 &+\sum_{\underline{w}\in J_1}q^\bullet b_{\underline{w}} ab\beta_0\widehat{e_1}^{i+1}\widehat{e_2}^j\widehat{e_3}^{k+3}
 \widehat{e_5}^l\widehat{e_6}^{m+1}\nonumber\\
 &+\sum_{\underline{w}\in J_2}q^{\bullet} c_{\underline{w}}a\widehat{e_1}^{i+1}\widehat{e_2}^{j}\widehat{e_3}^k\widehat{e_4}
\widehat{e_5}^{l}\widehat{e_6}^m\nonumber\\
&+\sum_{\underline{w}\in J_2}c_{\underline{w}} a'\widehat{e_1}^i\widehat{e_2}^{j}\widehat{e_3}^{k+2} \widehat{e_5}^l\widehat{e_6}^m+\Upsilon,
  \end{align}
where $\Upsilon$ is defined as follows:
  \begin{align*}
 \Upsilon =&\sum_{\underline{w}\in J_1}q^{\bullet} b_{\underline{w}} \widehat{e_1}^{i+1}\widehat{e_2}^j\widehat{e_3}^{k+1}\widehat{e_4}
 \widehat{e_5}^{l+1}\widehat{e_6}^m + 
 \sum_{\underline{w}\in J_1}q^{\bullet} b_{\underline{w}}d_2[k] \widehat{e_1}^{i}\widehat{e_2}^{1+j}\widehat{e_3}^{k}\widehat{e_4}
 \widehat{e_5}^{l+1}\widehat{e_6}^m\\
 &-\sum_{\underline{w}\in J_1}q^{\bullet} b_{\underline{w}}a\beta\beta_0 \widehat{e_1}^{i+1}\widehat{e_2}^{j}\widehat{e_3}^{k}
 \widehat{e_5}^{l}\widehat{e_6}^m+
 \sum_{\underline{w}\in J_1}q^{\bullet} b_{\underline{w}}a\beta_0 \widehat{e_1}^{i+1}\widehat{e_2}^{1+j}\widehat{e_3}^{k}\widehat{e_4}
 \widehat{e_5}^{l}\widehat{e_6}^{m+1}\\
&+ \sum_{\underline{w}\in J_1}q^{\bullet} b_{\underline{w}}ab\beta_0 \widehat{e_1}^{i+1}\widehat{e_2}^{j+1}\widehat{e_3}^{k}
 \widehat{e_5}^{l+3}\widehat{e_6}^m+
 \sum_{\underline{w}\in J_1}q^{\bullet} b_{\underline{w}}ab'\beta_0 \widehat{e_1}^{i+1}\widehat{e_2}^{j}\widehat{e_3}^{k+2}
 \widehat{e_5}^{l+2}\widehat{e_6}^m\\
 &+\sum_{\underline{w}\in J_1}q^{\bullet} b_{\underline{w}}ac'\beta_0 \widehat{e_1}^{i+1}\widehat{e_2}^{j}\widehat{e_3}^{k+1}
 \widehat{e_4}\widehat{e_5}^{l+1}\widehat{e_6}^m
 -\sum_{\underline{w}\in J_1}q^{\bullet} b_{\underline{w}}a\beta\beta_0 d_2[k] \widehat{e_1}^{i}\widehat{e_2}^{1+j}\widehat{e_3}^{k-1}
 \widehat{e_5}^{l}\widehat{e_6}^m\\&+
 \sum_{\underline{w}\in J_1}q^{\bullet} b_{\underline{w}}a\beta_0d_2[k] \widehat{e_1}^{i}\widehat{e_2}^{j+2}\widehat{e_3}^{k-1}
\widehat{e_4} \widehat{e_5}^{l}\widehat{e_6}^{m+1}+
\sum_{\underline{w}\in J_1}q^{\bullet} b_{\underline{w}}ab\beta_0d_2[k] \widehat{e_1}^{i}\widehat{e_2}^{j+2}\widehat{e_3}^{k-1}
 \widehat{e_5}^{l+3}\widehat{e_6}^m\\
& +\sum_{\underline{w}\in J_1}q^{\bullet} b_{\underline{w}}ab\beta_0d_2[k] \widehat{e_1}^{i+1}\widehat{e_2}^{j}\widehat{e_3}^{k+2}
 \widehat{e_5}^{l}\widehat{e_6}^{m+1}
+\sum_{\underline{w}\in J_1}q^{\bullet} b_{\underline{w}}ab'\beta_0d_2[k] \widehat{e_1}^{i}\widehat{e_2}^{j+1}\widehat{e_3}^{k+1}
 \widehat{e_5}^{l+2}\widehat{e_6}^m\\
&+\sum_{\underline{w}\in J_1}q^{\bullet} b_{\underline{w}}ac'\beta_0d_2[k] \widehat{e_1}^{i}\widehat{e_2}^{j+1}\widehat{e_3}^{k}
 \widehat{e_4}\widehat{e_5}^{l+1}\widehat{e_6}^m
 +\sum_{\underline{w}\in J_1}q^{\bullet} b_{\underline{w}}a \widehat{e_1}^{i}\widehat{e_2}^{j+1}\widehat{e_3}^{k}
 \widehat{e_4}\widehat{e_5}^{l+1}\widehat{e_6}^m\\
 &-\sum_{\underline{w}\in J_1}b_{\underline{w}}\beta \widehat{e_1}^{i}\widehat{e_2}^{j}\widehat{e_3}^{k}\widehat{e_4}
\widehat{e_5}^{l}\widehat{e_6}^m+\sum_{\underline{w}\in J_2}q^{\bullet} c_{\underline{w}}\widehat{e_1}^{i+1}\widehat{e_2}^{j}\widehat{e_3}^{k+1}
\widehat{e_5}^{l+1}\widehat{e_6}^m\\
&
+\sum_{\underline{w}\in J_2}q^{\bullet} c_{\underline{w}}d_2[k]\widehat{e_1}^{i}\widehat{e_2}^{1+j}\widehat{e_3}^{k}
\widehat{e_5}^{l+1}\widehat{e_6}^m+\sum_{\underline{w}\in J_2}q^{\bullet}c_{\underline{w}}ad_2[k]\widehat{e_1}^{i}\widehat{e_2}^{j+1}\widehat{e_3}^{k-1}\widehat{e_4}
\widehat{e_5}^{l}\widehat{e_6}^m\\
&
+\sum_{\underline{w}\in J_2}q^{\bullet}c_{\underline{w}}a\widehat{e_1}^{i}\widehat{e_2}^{j+1}\widehat{e_3}^{k}
\widehat{e_5}^{l+1}\widehat{e_6}^m-\sum_{\underline{w}\in J_2}c_{\underline{w}}\alpha\widehat{e_1}^{i}\widehat{e_2}^{j}\widehat{e_3}^{k}
\widehat{e_5}^{l}\widehat{e_6}^m.
  \end{align*}
Before we continue the proof, the following point needs to be noted.

\begin{itemize}
\item[$\spadesuit$] Let $(\vartheta_1,\vartheta_2,\vartheta_3,\vartheta_5,\vartheta_6)$, $(\varsigma_1,\varsigma_2,\varsigma_3,\varsigma_5,\varsigma_6)\in\mathbb{N}^5.$ We say that $(\varsigma_1,\varsigma_2,\varsigma_3,\varsigma_5,\varsigma_6)<_3 (\vartheta_1,\vartheta_2,\vartheta_3,\vartheta_5,\vartheta_6)$ if $[\vartheta_3>\varsigma_3]$ or 
$[\vartheta_3=\varsigma_3$ and $\vartheta_1>\varsigma_1]$ or $[\vartheta_3=\varsigma_3, \ \vartheta_1=\varsigma_1$ and $\vartheta_2>\varsigma_2]$ 
 or $[\vartheta_3=\varsigma_3,\  \vartheta_1=\varsigma_1, \ \vartheta_2=\varsigma_2$ and $\vartheta_5>\varsigma_5]$ or $[\vartheta_3=\varsigma_3,\ \vartheta_1=\varsigma_1, \ \vartheta_2=\varsigma_2,\  \vartheta_5=\varsigma_5$ and $\vartheta_6\geq \varsigma_6].$
\end{itemize} 
Observe that $\Upsilon$ contains lower order terms with respect to  
$<_3$ (defined in $\spadesuit$ above) in each monomial type (note, there are two different types of monomials in the basis of $A_{\beta}$: one with $\widehat{e_4}$ and the other without $\widehat{e_4}$). Now, suppose that there exists $(i,j,k,l,m)\in J_1$ and $(i,j,k,l,m)\in J_2$ such that $b_{(i,j,k,l,m)}\neq 0$ and $c_{(i,j,k,l,m)}\neq 0.$ Let $(v_1,v_2,v_3,v_5,v_6)$ and $(w_1,w_2,w_3,w_5,w_6)$ be the greatest elements of  $J_1$ and $J_2$ respectively with respect to $<_{3}$  such that $b_{(v_1,v_2,v_3, v_5,v_6)}$ and $ c_{(w_1,w_2,w_3,w_5,w_6)}$  are non-zero. Since  $\mathfrak{S}$ is a linear basis for $A_\beta,$ and  $\Upsilon$ contains lower order terms with respect to $<_3$, we have the following:
 if $w_3-v_3<2,$ then identifying the coefficients  of $\widehat{e_1}^{v_1}\widehat{e_2}^{v_2}\widehat{e_3}^{v_3+2}\widehat{e_4}
 \widehat{e_5}^{v_5}\widehat{e_6}^{v_6}$ in \eqref{aa},  we have that $a'b_{(v_1,v_2,v_3, v_5,v_6)}=0$.  But 
 $b_{(v_1,v_2,v_3, v_5,v_6)}\neq 0,$ hence $ a'=q^6/(q^2-1)=0,$ a contradiction (see Appendix \ref{appenc} for the expression of $a'$). 
Finally, if $w_3-v_3\geq 2,$ then identifying the coefficient  of $\widehat{e_1}^{w_1}\widehat{e_2}^{w_2}\widehat{e_3}^{w_3+2}\widehat{e_5}^{w_5}\widehat{e
_6}^{w_6},$ we have that $a'c_{(w_1,w_2,w_3,w_5,w_6)}=0$.  But  $c_{(w_1,w_2,w_3,w_5,w_6)}\neq 0,$
  hence $a'=0,$ a contradiction! This implies that either  all $b_{(i,j,k,l,m)}$ or  all $ c_{(i,j,k,l,m)}$ are zero. 
  Without loss of generality, suppose that there exists $(i,j,k,l,m)\in J_2$ such that  $c_{(i,j,k,l,m)}$ is not zero. Then, $ b_{(i,j,k,l,m)}$ are all zero. Let $(w_1,w_2,w_3,w_5,w_6)$ be the greatest element of $J_2$ such that $ c_{(w_1,w_2,w_3,w_5,w_6)}\neq 0.$  Identifying the coefficients of 
  $\widehat{e_1}^{w_1}\widehat{e_2}^{w_2}
 \widehat{e_3}^{w_3+2} \widehat{e_5}^{w_5}\widehat{e_6}^{w_6}$ in the above equality, we have that $a'c_{(w_1,w_2,w_3,w_5,w_6)}=0.$ Since $c_{(w_1,w_2,w_3,w_5,w_6)}\neq 0,$ it follows that 
 $a'=0,$ a contradiction! We can therefore conclude that  $b_{(i,j,k,l,m)}$ and $ c_{(i,j,k,l,m)}$ are all zero.  Consequently,
$
\sum_{(\epsilon_1,\epsilon_2,\underline{\upsilon})\in I}a_{(\epsilon_1,\epsilon_2,\underline{\upsilon})}
\widehat{e_1}^i\widehat{e_2}^j\widehat{e_3}^{\epsilon_1}\widehat{e_4}^{\epsilon_2}
\widehat{e_5}^k\widehat{e_6}^l
= 0.
$
Since $(\widehat{e_1}^{i_1}\widehat{e_2}^{i_2}\widehat{e_3}^{i_3}\widehat{e_4}^{\xi}
\widehat{e_5}^{i_5}\widehat{e_6}^{i_6})_{(\xi,i_1,\cdots,i_6)\in \{0,1\}\times \mathbb{N}^5}$ is a basis of $A_{\beta}$, it implies that
$a_{(\epsilon_1,\epsilon_2,\underline{v})}=0$ for all  $(\epsilon_1,\epsilon_2,\underline{v})\in I.$ Therefore, $\mathcal{B}$ is a linearly independent set.
\end{proof}

We note for future use the following immediate consequence of Proposition \ref{c3p5}.

\begin{corollary}
\label{c4c2}
Let $\underline{v}=(i,j,k,l)\in \mathbb{N}^2\times \mathbb{Z}^2, \ I$ represent a finite subset of $\{0,1\}\times \mathbb{N}^2\times \mathbb{Z}^2$  and 
$(a_{(\epsilon_1,\epsilon_2,\underline{v})})_{(\epsilon_1,\epsilon_2, \underline{v})\in I}$ be a family of scalars.  If 
$$\sum_{(\epsilon_1,\epsilon_2, \underline{v})\in I}a_{(\epsilon_1,\epsilon_2, \underline{v})}e_{1}^{i}e_{2}^{j}e_3^{\epsilon_1}
e_4^{\epsilon_2}t_5^{k}t_6^{l}=0,$$ then $a_{(\epsilon_1,\epsilon_2, \underline{v})}=0$ for all $(\epsilon_1,\epsilon_2, \underline{v})\in I.$ 
\end{corollary}

\begin{remark}
Given the basis of $A_{\alpha,\beta},$ we have computed the group of units of $A_{\alpha,\beta},$ however, we do not include the details in this manuscript due to the voluminous computations involved. We only summarise our findings below. 
Set 
$$h_1:=e_3e_5+ae_4 \ \ \text{and} \ \ h_2:=(q^{-3}-q^{-9})e_2e_4-(q^4-2q^2+1)/(q^4+q^2+1)e_3^3.$$
\begin{theorem}
Let $(\alpha,\beta)\in \k^2\setminus \{(0,0)\}$ and 
$\mathcal{U}(A_{\alpha,\beta})$ denote the group of units of $A_{\alpha,\beta}.$ We have that:
$$\mathcal{U}(A_{\alpha,\beta})=\begin{cases}
    \{\lambda h_1^i\mid \lambda\in\k^*, \ i\in\mathbb{Z}\}& \text{if $\alpha=0$}\\
  \{\lambda h_2^i\mid \lambda\in\k^*, \ i\in\mathbb{Z}\}& \text{if $\beta=0$}\\
  \k^*& \text{$otherwise$.}
  \end{cases}
$$ 
\end{theorem}
\end{remark}


\subsection{$A_{\alpha,\beta}$ as a $q$-deformation of  a quadratic extension of $A_2(\k)$}
\label{rec}
Recall that $\gkdim A_{\alpha,\beta} =4$ and so we should compare $A_{\alpha,\beta}$ to the second Weyl algebra. In this section, we prove that, for a suitable choice of $\alpha$ and $\beta$, the simple algebra $A_{\alpha,\beta}$ is a $q$-deformation of (a quadratic extension of) $A_2(\k).$ 

Recall that 
$A_2(\k)$ is generated by $x_1,x_2,y_1$ and $y_2$ subject to the relations: 
\begin{align*}
 y_1y_2&=y_2y_1 & x_2y_1&=y_1x_2 &
x_1x_2&=x_2x_1 &  x_1y_1&-y_1x_1=1\\
y_1y_2&=y_2y_1 & x_1y_2&=y_2x_1&
x_2y_1&=y_1x_2 &   x_
2y_2&-y_2x_2=1.
\end{align*}

Given the relations of $A_{\alpha,\beta}$ at the onset of this section, we have that $A_{1,\frac{1}{9(q^6-1)}}$ satisfies the following relations:
 \begin{align*}
e_2e_1&=q^{-3}e_1e_2& e_3e_1&=q^{-1}e_1e_3-(q+q^{-1}+q^{-3})e_2\\
e_3e_2&=q^{-3}e_2e_3&e_4e_1&=e_1e_4+(1-q^2)e_3^2\\
e_4e_2&=q^{-3}e_2e_4-\frac{q^4-2q^2+1}{q^4+q^2+1}e_3^3&e_4e_3&=q^{-3}e_3e_4\\
e_5e_1&=qe_1e_5-(1+q^2)e_3&e_5e_2&=e_2e_5+(1-q^2)e_3^2\\
e_5e_3&=q^{-1}e_3e_5-(q+q^{-1}+q^{-3})e_4&e_5e_4&=q^{-3}e_4e_5\\
e_6e_1&=q^3e_1e_6-q^3e_5&e_6e_2&=q^3e_2e_6+(q^4+q^2-1)e_4+(q^2-q^4)e_3e_5\\
e_6e_3&=e_3e_6+(1-q^2)e_5^2&e_6e_4&=q^{-3}e_4e_6-\frac{q^4-2q^2+1}{q^4+q^2+1}e_5^3\\
e_6e_5&=q^{-3}e_5e_6,
\end{align*} and
\begin{align*}
(q^{-2}-1)e_1e_3e_5+(q^2+1+q^{-2}) e_1e_4+(q^2+1+q^{-2})e_2e_5-{q^4}e_3^2&=q^{-2}-1,\\
\\
(q^6-1)e_2e_4e_6+\frac{2q^{-1}-q^{-3}-q}{q^4+q^2+1} e_2e_5^3+\frac{2q^{-1}-q^{-3}-q}{q^4+q^2+1}e_3^3e_6&\\+\frac{(q^6-1)(q^{13}-q^{11})}{(q^4+q^2+1)^2}e_3^2e_5^2-\frac{q^9(q^6-1)}{q^4+q^2+1}e_3e_4e_5+q^{12}e_4^2&=\frac{1}{9}.
\end{align*}
Note, we have made the necessary substitutions for $a,\ a', \ b, \ b', \ c'$ and  $d'$ from Appendix \ref{appenc}.

Set $F:=\k[z^{\pm 1}].$ One can define a $F[(z^4+z^2+1)^{-1}]-$algebra $A_{z}$ generated by 
$e_1, \cdots, e_6$ subject to the following relations:
   \begin{align*}
e_2e_1&=z^{-3}e_1e_2& e_3e_1&=z^{-1}e_1e_3-(z+z^{-1}+z^{-3})e_2\\
e_3e_2&=z^{-3}e_2e_3&e_4e_1&=e_1e_4+(1-z^2)e_3^2\\
e_4e_2&=z^{-3}e_2e_4-\frac{z^4-2z^2+1}{z^4+z^2+1}e_3^3&e_4e_3&=z^{-3}e_3e_4\\
e_5e_1&=z e_1e_5-(1+z^2)e_3&e_5e_2&=e_2e_5+(1-z^2)e_3^2\\
e_5e_3&=z^{-1}e_3e_5-(z+z^{-1}+z^{-3})e_4&e_5e_4&=z^{-3}e_4e_5\\
e_6e_1&=z^3e_1e_6-z^3e_5&e_6e_2&=z^3e_2e_6+(z^4+z^2-1)e_4+(z^2-z^4)e_3e_5\\
e_6e_3&=e_3e_6+(1-z^2)e_5^2&e_6e_4&=z^{-3}e_4e_6-\frac{z^4-2z^2+1}{z^4+z^2+1}e_5^3\\
e_6e_5&=z^{-3}e_5e_6,
\end{align*} 
\begin{align*}
(z^{-2}-1)e_1e_3e_5+(z^2+1+z^{-2}) e_1e_4+(z^2+1+z^{-2})e_2e_5-{z^4}e_3^2&=z^{-2}-1, \ \ \text{and}
\end{align*}
\begin{align*}
(z^6-1)e_2e_4e_6+\frac{2z^{-1}-z^{-3}-z}{z^4+z^2+1} e_2e_5^3+\frac{2z^{-1}-z^{-3}-z}{z^4+z^2+1}e_3^3e_6&\\+\frac{(z^6-1)(z^{13}-z^{11})}{(z^4+z^2+1)^2}e_3^2e_5^2-\frac{z^9(z^6-1)}{z^4+z^2+1}e_3e_4e_5+z^{12}e_4^2&=\frac{1}{9}.
\end{align*}
Set $A_1:= A_z / \langle z-1 \rangle$ and observe that $A_1$ satisfies the following relations: 
\begin{align*}
e_2e_1&=e_1e_2& e_3e_1&=e_1e_3-3
e_2&
e_3e_2&=e_2e_3\\ 
e_4
e_1&=e_1e_4&
e_4e_2&=e_2e_4&e_4e_3&=
e_3
e_4\\
e_5e_1&=e_1e_5-
2e_3&e_5e_2&=e_2e_5
&e_5e_3&=e_3e_5
-3e_4\\
e_5e_4&=e_4
e_5&
e_6e_1&=e_1e_6-e_5&
e_6e_2&=e_2e_6
+e_4\\
e_6e_3&=e_3e_6&e_6e_4&=
e_4e_6
&
e_6e_5&=e_5e_6\\
e_4^2-&1/9=0&&&   e_3^2-&3e_1e_4-3e_2e_5=0.
\end{align*}
\begin{lemma}
$e_4\in Z(A_1)$ and it is also invertible.
\end{lemma}
\begin{proof}
 Since $e_4e_i=e_ie_4$ for all $1\leq i\leq 6,$ we have that $e_4\in Z(A_1).$ Again, from $e_4^2-1/9=0,$ we have that $e_4(9e_4)=(9e_4)e_4=1.$ Hence $e_4$ is invertible with $e_4^{-1}=9e_4.$ 
\end{proof}

Given that $e_4^{-1}=9e_4$ and $e_4\in Z(A_1),$ it follows from the relation $e_3^2-3e_1e_4-3e_2e_5=0$  that 
$e_1=3e_3^2e_4-9e_2e_4e_5.$ Therefore, $A_1$ can be generated by only 
$e_2, \cdots, e_6.$
All these generators commute except that $$e_6e_2=e_2e_6+e_4 \ \ \text{and} \ \ e_5e_3=e_3e_5
-3e_4.$$ 
Since $e_4$ is invertible, one can also verify that  $9e_2e_4, 3e_3e_4, e_4, e_5$ and $e_6$ generate $A_1.$

Let  $R$ be an algebra generated by $f_2, f_3, f_4, f_5, f_6$ subject to the following defining relations:

\begin{align*}
f_3f_2&=f_2f_3 & f_4f_2&=f_2f_4& f_4f_3&=f_3f_4\\
f_5f_2&=f_2f_5 & f_5f_4&=f_4f_5 & f_6f_3&=f_3f_6\\
f_6f_4&=f_4f_6& f_6f_5&=f_5f_6& f_4^2&=1/9\\
f_6f_2&=f_2f_6+1&f_5f_3&=f_3f_5-1. 
\end{align*}
\begin{proposition}
$R\cong A_1.$
\end{proposition}
\begin{proof}
One can easily check that we define a homomorphism $\phi: R\longrightarrow A_1$ via 
\begin{align*}
\phi(f_2)&=9e_2e_4& \phi(f_3)&=3e_3e_4& \phi(f_4)&=e_4& \phi(f_5)&=e_5&
\phi(f_6)&=e_6.
\end{align*}
Recall, $e_4^2=1/9.$ To check that $\phi$ is indeed a homomorphism, we just need to check its compatibility with the defining relations of $R.$ We check this on the relation 
$f_6f_2-f_2f_6=1$ and $f_3f_5-f_5f_3=1,$ and leave the remaining ones for the reader to verify. We do that as follows: $\phi(f_6)\phi(f_2)-\phi(f_2)\phi(f_6)=9e_6e_2e_4-9e_2e_4e_6=9e_4(e_6e_2-e_2e_6)=9e_4^2=9(1/9)=1$ as needed. Also, $\phi(f_3)\phi(f_5)-\phi(f_5)\phi(f_3)=3e_3e_4e_5-3e_5e_3e_4=3e_4(e_3e_5-e_5e_3)=3e_4(3e_4)=9e_4^2=1.$

Conversely, one can check that we define a homomorphism $\varphi: A_1\longrightarrow R$ via 
\begin{align*}
\varphi(e_1)&=3f_3^2f_4-f_2f_5& \varphi(e_2)&=f_2f_4& \varphi(e_3)&=3f_3f_4\\
\varphi(e_4)&=f_4& \varphi(e_5)&=f_5 & \varphi(e_6)&= f_6.
\end{align*}
We check this on the relation $e_3^2-3e_1e_4-3e_2e_5=0,$ and leave the remaining ones for the reader to verify. We do that as follows:
$\varphi(e_3)^2-3\varphi(e_1)\varphi(e_4)-3\varphi(e_2)\varphi(e_5)=
(3f_3f_4)^2-3(3f_3^2f_4-f_2f_5)f_4-3f_2f_4f_5=9f_3^2f_4^2-9f_3^2f_4^2+
3f_2f_4f_5-3f_2f_4f_5=0$ as expected.

To conclude we just observe that $\phi$ and $\varphi$ are inverse of each other.
\end{proof}

The corollary  below can easily be deduced from the above proposition. 
\begin{corollary}
Set $\mathbb{F}:=\k[f_4]/\langle f_4^2-1/9\rangle,$ we have that $R\cong A_2(\mathbb{F}),$ where  $A_2(\mathbb{F})$ is the second Weyl algebra over the ring  $\mathbb{F}.$
\end{corollary}

\begin{remark}
Observe that the subalgebra $B$ of $R$ generated by $f_2, f_3,f_5, f_6$ is isomorphic to $A_2(\k)$ and $R\cong B[f_4]\cong A_2(\k)[f_4].$ Therefore, $R$ is a \textit{quadratic extension} of $A_2(\k).$ Note, $A_{1,\frac{1}{9(q^6-1)}}$ is a $q$-deformation of $A_1\cong R\cong A_2(\mathbb{F})\cong A_2(\k)[f_4].$ 
\end{remark} 


\section{Derivations of the simple quotients of $U_q^+(G_2)$}
\label{chapd}
In this section, we compute the derivations of the algebra $A_{\alpha,\beta}$ using  DDA that allows to embed $A_{\alpha,\beta}$ into a suitable quantum torus. Derivations of quantum tori are known, thanks to the work of Osborn and Passman \cite{op}. In our cases, such derivations are always the sum of an inner derivation and a scalar derivation (of the quantum torus). Since $A_{\alpha,\beta}$ can be embedded into a quantum torus, we first extend every derivation of $A_{\alpha,\beta}$ to a derivation of such quantum torus, and then pull back their description as a derivation of the quantum torus to a description of their action on the generators of $A_{\alpha,\beta}$ by ``reverting'' DDA process. We conclude that every derivation of $A_{\alpha,\beta}$ is inner when $\alpha$ and $\beta$ are both non-zero. However, when either $\alpha$ or $\beta$ is zero, we conclude that every derivation of $A_{\alpha,\beta}$ is the sum of an inner and a scalar derivation. In fact, the first Hochschild cohomology group of $A_{\alpha,\beta}$ is of dimension $0$ when $\alpha$ and $\beta$ are both non-zero and $1$ when either $\alpha$ or $\beta$ is zero. 

\subsection{Preliminaries and strategy}
\label{c4ss1}
 Let $2\leq j\leq 7$ and  $(\alpha,\beta)\in \k^2\setminus \{(0,0)\}.$ Set 
$$A_{\alpha,\beta}^{(j)}:=\frac{A^{(j)}}{\langle \Omega_1-\alpha,\Omega_2-\beta\rangle},$$
where $A^{(j)}$ is defined in Section \ref{c2s1} and, $\Omega_1$ and $\Omega_2$ are the generators of the center of $A^{(j)}$, see Remark \ref{remark-centre-Ri}. Note in particular that $A_{\alpha,\beta}^{(7)}=A_{\alpha,\beta}.$
For each $2\leq j\leq 7,$ denote the canonical images of the generators $E_{i,j}$ of $A^{(j)}$ in $A_{\alpha,\beta}^{(j)}$ by $e_{i,j}$ for all 
$1\leq i\leq 6.$  

As usual we denote by $t_i$ the canonical image of $T_i$ in $A_{\alpha,\beta}^{(2)}$ for each $1\leq i\leq 6.$
For each $3\leq j\leq 6,$  define $S_j:=\left\lbrace \lambda t_j^{i_j}t_{j+1}^{i_{j+1}}\cdots t_6^{i_6}\mid i_j,\cdots,i_6\in\mathbb{N} \ \text{and} \ \lambda\in \k^* \right\rbrace.$ One can observe that $S_j$ is a multiplicative system of non-zero divisors (or regular elements) of $A_{\alpha,\beta}^{(j)}.$ Furthermore; $t_j, \cdots, t_6$ are all normal elements of $A_{\alpha,\beta}^{(j)}$ and so $S_j$ is an Ore set in $A_{\alpha,\beta}^{(j)}$. One can localize $A_{\alpha,\beta}^{(j)}$ at $S_j$ as follows: 
$$\cR_j:=A_{\alpha,\beta}^{(j)}S_j^{-1}.$$ Let $3\leq j\leq 6$, and set $\Sigma_j:=\{t_j^k\mid k\in \mathbb{N}\}.$ By \cite[Lemme 5.3.2]{ca}, $\Sigma_j$ is an Ore set in both $A_{\alpha,\beta}^{(j)}$ and $A_{\alpha,\beta}^{(j+1)},$ and $$A_{\alpha,\beta}^{(j)}\Sigma_j^{-1}=A_{\alpha,\beta}^{(j+1)}\Sigma_j^{-1}.$$ As a consequence, similar to \eqref{emb0}, we have that
\begin{align}
\label{der}
\cR_j=\cR_{j+1}\Sigma_j^{-1},
\end{align}
for all $2\leq j\leq 6.$ By convention, $\cR_7:=A_{\alpha,\beta}.$ 
We also construct the following tower of algebras in a manner similar to \eqref{emb}:
\begin{align}
\label{c4ee1}
\cR_7=A_{\alpha,\beta}\subset \cR_6=\cR_7\Sigma_6^{-1}\subset \cR_5=\cR_6\Sigma_5^{-1}\subset \cR_4=\cR_5\Sigma_4^{-1} \subset \cR_3.
\end{align} 
Note, $\cR_3=A_{\alpha,\beta}^{(3)}S_3^{-1}=\cR_4\Sigma_3^{-1}$  is the quantum torus  $\mathscr{A}_{\alpha,\beta}=\k_{q^N}[t_3^{\pm1},t_4^{\pm1},t_5^{\pm1},t_6^{\pm 1}]$ studied in Section  \ref{c3p3}.

Our strategy to compute the derivations of $\cR_7$ is to extend these derivations to derivations of the quantum torus $\cR_3$. Then we can use the description of the derivations of a quantum torus obtained by Osborn and Passman in \cite{op}. Once this is done, we will have a ``nice'' description but involving elements of $\cR_3$ and we will then use the fact that these derivations fix (globally) all $\cR_i$ to obtain a description only involving elements of $\cR_7$. This is a step by step process requiring knowing linear bases for $\cR_i$. We find such bases in the next section.  

Before doing so, we note from \cite[Lemme 5.3.2]{ca} that  DDA theory predicts the following relations between the elements $e_{i,j}$:
\begin{align*}
e_{1,6}&=e_1+re_5e_6^{-1}\\
e_{2,6}&=e_2+te_3e_5e_6^{-1}+ue_4e_6^{-1}+ne_5^3e_6^{-2}\\
e_{3,6}&=e_3+se_5^2e_6^{-1}\\
e_{4,6}&=e_4+be_5^3e_6^{-1}\\
e_{1,5}&=e_{1,6}+he_{3,6}e_{5,6}^{-1}+ge_{4,6}e_{5,6}^{-2}\\
e_{2,5}&=e_{2,6}+fe_{3,6}^2e_{5,6}^{-1}+pe_{3,6}e_{4,6}e_{5,6}^{-2}+ee_{4,6}^2e_{5,6}^{-3
}\\
e_{3,5}&=e_{3,6}+ae_{4,6}e_{5,6}^{-1}\\
e_{1,4}&=e_{1,5}+se_{3,5}^2e_{4,5}^{-1}\\
e_{2,4}&=e_{2,5}+be_{3,5}^3e_{4,5}^{-1}\\
e_{1,3}&=e_{1,4}+ae_{2,4}e_{3,4}^{-1}\\
t_1&:=e_{1,2}=e_{1,3}\\
t_2&:=e_{2,2}=e_{2,3}=e_{2,4}\\
t_3&:=e_{3,2}=e_{3,3}=e_{3,4}=e_{3,5}\\
t_4&:=e_{4,2}=e_{4,3}=e_{4,4}=e_{4,5}=e_{4,6}\\
t_5&:=e_{5,2}=e_{5,3}=e_{5,4}=e_{5,5}=e_{5,6}=e_5\\
t_6&:=e_{6,2}=e_{6,3}=e_{6,4}=e_{6,5}=e_{6,6}=e_6,
\end{align*}
where, as usual, the necessary parameters can be found in Appendix \ref{appenc}.

We also note that we have complete control over the centers of the algebras $\cR_i$.

\begin{lemma}
\label{ev25}
Let  $Z(\cR_i)$ denote the center of $\cR_i,$  then $Z(\cR_i)=\k$ for each $3\leq i\leq 7.$
\end{lemma}
\begin{proof}
One can easily verify that $Z(\cR_3)=\k.$ Note,  $\cR_7=A_{\alpha,\beta}.$  Since $\cR_i$ is a localization of $\cR_{i+1}$ (see \eqref{der}), we have that
$\k\subseteq Z(\cR_7)\subseteq Z(\cR_6)\subseteq Z(\cR_5)\subseteq Z(\cR_4)\subseteq Z(\cR_3)=\k.$
Therefore, $Z(\cR_7)=Z(\cR_6)=Z(\cR_5)=Z(\cR_4)=Z(\cR_3)=\k.$
\end{proof}

\subsection{Linear bases for $\cR_3$, $\cR_4$ and $\cR_5$}
Let $(\alpha,\beta)\in \k^2\setminus \{(0,0)\}.$ We aim to  find a basis of  $\cR_j$ for each $j=3,4,5.$  Since $\cR_3=\mathscr{A}_{\alpha,\beta},$  the set $\{t_3^{i}t_4^{j}t_5^{k}t_6^{l}\mid i,j,k,l \in \mathbb{Z}\}$ is a $\k-$basis of $R_3.$ 

{For simplicity,} we set
\begin{align*}
f_1:&=e_{1,4} & F_1:&=E_{1,4}\\
z_1:&=e_{1,5} & \zbar_1:&=E_{1,5} \\
z_2:&=e_{2,5} & \zbar_2:&=E_{2,5}.
\end{align*}

\textbf{Basis for $\cR_4$.} Observe that $$A_{\alpha,\beta}^{(4)}=\frac{A^{(4)}}{\langle \Omega_{1}-\alpha, \Omega_{2}-\beta\rangle},$$
where $\Omega_{1}=F_1T_3T_5+aT_2T_5$ and $\Omega_{2}=T_2T_4T_6$ in $A^{(4)}$. 
Recall from Section \ref{c3s2} that finding a basis for the algebra $A_\beta$ served as a good ground for finding a basis for $A_{\alpha,\beta}.$ In a similar manner, to find a basis for $\cR_4,$ we will first and foremost find a basis for the algebra 
$$A_\beta^{(4)}S_4^{-1}=\frac{A^{(4)}S_4^{-1}}{\langle \Omega_{2}-\beta \rangle}=\frac{A^{(4)}S_4^{-1}}{\langle T_2T_4T_6-\beta \rangle},$$ 
where $\beta\in \k^*$.  We will denote  the canonical images of $E_{i,4}$ (resp. $T_i$) in ${A}_{\beta}^{(4)}$ by $\widehat{e_{i,4}}$ (resp.  $\widehat{t_i}$)  for all $1\leq i\leq 6.$ Observe that $\widehat{t_2}=\beta \widehat{t_6}^{-1}\widehat{t_4}^{-1}$ in  $A_\beta^{(4)}S_4^{-1}.$ Note, when $\beta=0,$ then one can easily deduce that 
$A_\beta^{(4)}S_4^{-1}={A^{(4)}S_4^{-1}}/{\langle T_2 \rangle}$, hence, 
$\widehat{t_2}=0.$
\begin{proposition}
\label{c4p2}
The set $\mathfrak{S}_4=\left\lbrace \widehat{f_1}^{i_1}\widehat{t_3}^{i_3}
\widehat{t_4}^{i_4}\widehat{t_5}^{i_5}
\widehat{t_6}^{i_6}\mid (i_1,i_3,i_4,i_5,i_6)\in \mathbb{N}^2\times \mathbb{Z}^3 \right\rbrace $ is a $\k-$basis of $A_\beta^{(4)}S_4^{-1},$ where $\beta\in \k.$
\end{proposition}

\begin{proof}
We begin by showing that $\mathfrak{S}_4$ is a spanning set for $A_\beta^{(4)}S_4^{-1}.$ It is sufficient to do this by showing that $\widehat{f_1}^{k_1}\widehat{t_2}^{k_2}\widehat{t_3}^{k_3}
\widehat{t_4}^{k_4}\widehat{t_5}^{k_5}
\widehat{t_6}^{k_6}$ can be written as a finite linear combination of the elements of $\mathfrak{S}_4$ for all $(k_1,\cdots,k_6)\in \mathbb{N}^3\times \mathbb{Z}^3.$ This can easily be done through an induction on $k_2$ using the fact that $\widehat{t_2}=\beta \widehat{t_6}^{-1}\widehat{t_4}^{-1}$ (note that, if $\beta=0$, then $\widehat{t_2}=0$).

We now prove that $\mathfrak{S}_4$ is a linearly independent set.  Suppose that $$\sum_{\underline{i}\in I} a_{\underline{i}}\widehat{f_1}^{i_1}\widehat{t_3}^{i_3}
\widehat{t_4}^{i_4}\widehat{t_5}^{i_5}
\widehat{t_6}^{i_6}=0.$$ This implies that
$$\sum_{\underline{i}\in I} a_{\underline{i}}F_{1}^{i_1}T_3^{i_3}T_4^{i_4}T_5^{i_5}T_6^{i_6}=(\Omega_{2}-\beta)
\nu,$$
 for some $\nu \in A^{(4)}S_4^{-1}.$  Write $\nu=\displaystyle\sum_{\underline{j}\in J}b_{\underline{j}}F_{1}^{i_1}T_2^{i_2}T_3^{i_3}T_4^{i_4}T_5^{i_5}T_6^{i_6},$ where $\underline{j}=(i_1,i_2,i_3,i_4,i_5,i_6) \in J\subset \mathbb{N}^3\times\mathbb{Z}^3$ and $b_{\underline{j}}$ is a family of scalars. Given that $\Omega_2=T_2T_4T_6,$ it follows from the above equality that 
  \begin{align*}
 \sum_{\underline{i}\in I} a_{\underline{i}}F_{1}^{i_1}T_3^{i_3}T_4^{i_4}T_5^{i_5}T_6^{i_6}=&
 \sum_{\underline{j}\in J}q^\bullet b_{\underline{j}} F_{1}^{i_1}T_2^{i_2+1}T_3^{i_3}T_4^{i_4+1}T_5^{i_5}T_6^{i_6+1}-\sum_{\underline{j}\in J}\beta b_{\underline{j}} F_{1}^{i_1}T_2^{i_2}T_3^{i_3}T_4^{i_4}T_5^{i_5}T_6^{i_6}.
 \end{align*}
 
 We denote by $<_2$ the total order on $\mathbb{Z}^6$ defined by $(i_1,i_2,i_3,i_4,i_5,i_6)<_2(w_1,w_2,w_3,w_4,w_5,w_6)$ if $[w_2>i_2]$ or $[w_2=i_2, \ w_1>i_1]$ or  $[w_2=i_2, \ w_1=i_1, \ w_3>i_3]$ or $\cdots$ or [$w_l=i_l, \ w_6\geq t_6, \ l=2,1,3,4,5]$.
 
 Suppose that there exists $(i_1,\cdots,i_6)\in J$ such that $b_{(i_1,\cdots,i_6)}\neq 0.$ Let $(w_1,\cdots,w_6)\in J$ be the greatest element of $J$ with respect to $<_2$  such that $b_{(w_1,\cdots,w_6)}\neq 0.$ Note,
  $\left( {F}_{1}^{i_1}{T}_2^{i_2}{T}_3^{i_3}T_4^{i_4}T_5^{i_5}T_6^{i_6}\right) _{(i_1,\cdots,i_6)\in J}$ is a basis of  $A^{(4)}S_4^{-1}.$ Identifying the coefficients of ${F}_{1}^{w_1}{T}_2^{w_2+1}{T}_3^{w_3}{T}_4^{w_4+1}{T}_5^{w_5}{T_6^{w_6+1}},$ we have that $b_{(w_1,\cdots,w_6)}=0.$   This is a contradiction to our assumption, hence $b_{(i_1,\cdots,i_6)}= 0$ for all $(i_1,\cdots, i_6)\in J.$ This implies that
$$\sum_{\underline{i}\in I} a_{\underline{i}}F_{1}^{i_1}T_3^{i_3}T_4^{i_4}T_5^{i_5}T_6^{i_6}=0.$$  
Consequently, $a_{\underline{i}}=0$ for all $\underline{i}\in I.$
  Therefore,   $\mathfrak{S}_4$ is a linearly independent set. 
\end{proof}

In $\cR_4=A_{\alpha,\beta}^{(4)}S_4^{-1},$ we have the following two relations:
$f_1t_3t_5+at_2t_5=\alpha$ and $t_2t_4t_6=\beta.$ This implies that
$f_{1}t_3=\alpha t_5^{-1}-at_2$ and $t_2=\beta t_6^{-1}t_4^{-1}.$ Putting these two relations together, we have that
\begin{align}
f_{1}t_3=
\alpha t_5^{-1}-a\beta t_6^{-1}t_4^{-1}\label{de1}.
\end{align} 
Note, we will usually identify $\cR_4$ with ${A_{\beta}^{(4)}S_4^{-1}}/{\langle \widehat{\Omega}_{1}-\alpha\rangle}.$
\begin{proposition} 
\label{c4p1} 
The set $\mathcal{B}_4=\left\lbrace f_{1}^{i_1}t_4^{i_4}t_5^{i_5}
t_6^{i_6}, t_3^{i_3}t_4^{i_4}t_5^{i_5}
t_6^{i_6}\mid i_1,i_3\in \mathbb{N} \ \text{and} \ i_4, i_5, i_6\in \mathbb{Z}\right\rbrace$ is a $\k-$basis of $\cR_4.$
\end{proposition}
\begin{proof}
Since
$\left( \widehat{f_{1}}^{k_1}\widehat{t_3}^{k_3}
\widehat{t_4}^{k_4}\widehat{t_5}^{k_5}\widehat{t_6}^{k_6}\right)_{(k_1,k_3,\cdots,k_6)\in \mathbb{N}^2\times \mathbb{Z}^3}$ is a basis of ${A_{\beta}^{(4)}S_4^{-1}}$ (Proposition \ref{c4p2}), the set
$\left( f_{1}^{k_1}t_3^{k_3}
t_4^{k_4}t_5^{k_5}t_6^{k_6}\right)_{(k_1,k_3,\cdots,k_6)\in \mathbb{N}^2\times \mathbb{Z}^3}$ spans $\cR_4.$ We show that $\mathcal{B}_4$ is a spanning set of $\cR_4$ by showing that $f_{1}^{k_1}t_3^{k_3}
t_4^{k_4}t_5^{k_5}t_6^{k_6}$ can be written as a finite linear combination of the elements of $\mathcal{B}_4$ for all $(k_1,k_3,\cdots,k_6)\in \mathbb{N}^2\times \mathbb{Z}^3.$ 
  By Proposition \ref{c4p2}, it is sufficient to do this by induction on  $k_1.$ The result is clear when $k_1=0.$ Assume that the statement is true for $k_1\geq 0.$ That is,
$$f_{1}^{k_1}t_3^{k_3}t_4^{k_4}t_5^{k_5}t_6^{k_6}
=\sum_{\underline{i}\in I_1}{a}_{\underline{i}}f_{1}^{i_1}t_4^{i_4}t_5^{i_5}t_6^{i_6}+\sum_{\underline{j}\in I_2}b_{\underline{j}}
t_3^{i_3}t_4^{i_4}t_5^{i_5}t_6^{i_6},$$ 
where $\underline{i}=(i_{1},i_{4},i_{5},i_{6})\in I_1\subset  \mathbb{N}\times\mathbb{Z}^3$ and  $\underline{j}=(i_{3},i_{4},i_{5},i_{6})\in I_2\subset  \mathbb{N}\times\mathbb{Z}^3.$
Note, $a_{\underline{i}}$ and $b_{\underline{j}}$ are all scalars.
It follows that
\begin{align*}
f_{1}^{k_1+1}t_3^{k_3}t_4^{k_4}t_5^{k_5}t_6^{k_6}&=
f_{1}\left( f_{1}^{k_1}t_3^{k_3}t_4^{k_4}t_5^{k_5}t_6^{k_6}\right) 
=\sum_{\underline{i}\in I_1}{a}_{\underline{i}}f_{1}^{i_1+1}t_4^{i_4}t_5^{i_5}t_6^{i_6}+\sum_{\underline{j}\in I_2}b_{\underline{j}}
f_{1}t_3^{i_3}t_4^{i_4}t_5^{i_5}t_6^{i_6}.
\end{align*}
Clearly, the monomial $f_{1}^{i_1+1}t_4^{i_4}t_5^{i_5}t_6^{i_6}\in$ Span($\mathcal{B}_4$). We have to also show 
that $f_{1}t_3^{i_3}t_4^{i_4}t_5^{i_5}t_6^{i_6}\in$ Span($\mathcal{B}_4$) for all $i_3\in \mathbb{N}$ and 
$i_4,i_5,i_6\in\mathbb{Z}.$ This can easily be achieved by an induction on $i_3,$ and the  use of the relation $f_1t_3=\alpha t_5^{-1}-a\beta t_6^{-1}t_4^{-1}.$ 
Therefore, by the principle of mathematical induction, $\mathcal{B}_4$ is a spanning set of $\cR_4$ over $\k.$\\

We prove that $\mathcal{B}_4$ is a linearly independent set. Suppose that
$$\sum_{\underline{i}\in I_1}a_{\underline{i}}f_{1}^{i_1}t_4^{i_4}t_5^{i_5}t_6^{i_6}
+\sum_{\underline{j}\in I_2}b_{\underline{j}}t_3^{i_3}t_4^{i_4}t_5^{i_5}t_6^{i_6}=0.$$ 
 It follows that there exists $\nu\in A_\beta^{(4)}S_4^{-1}$ such that
$$\displaystyle\sum_{\underline{i}\in I_1}a_{\underline{i}}\widehat{f_{1}}^{i_1}\widehat{t_4}^{i_4}
\widehat{t_5}^{i_5}\widehat{t_6}^{i_6}
+\displaystyle\sum_{\underline{j}\in I_2}b_{\underline{j}}\widehat{t_3}^{i_3}\widehat{t_4}^{i_4}
\widehat{t_5}^{i_5}\widehat{t_6}^{i_6}
=\left( \widehat{\Omega}_{1}-\alpha\right) \nu.$$
Write $\nu=\displaystyle\sum_{\underline{l}\in J}c_{\underline{l}}\widehat{f_{1}}^{i_1}\widehat{t_3}^{i_3}\widehat{t_4}^{i_4}
\widehat{t_5}^{i_5}\widehat{t_6}^{i_6},$
where $\underline{l}=(i_1,i_3,i_4,i_5,i_6)\in J\subset\mathbb{N}^2\times\mathbb{Z}^3$ and $c_{\underline{l}}\in \k.$ Note,
 $\widehat{t_2}=\beta\widehat{t_6}^{-1}\widehat{t_4}^{-1}.$  We have that $\widehat{\Omega}_1=
 \widehat{f_{1}}\widehat{t_3}\widehat{t_5}+a\widehat{t_2}\widehat{t_5}=\widehat{f_{1}}\widehat{t_3}\widehat{t_5}+a\beta\widehat{t_6}^{-1}\widehat{t_4}^{-1}\widehat{t_5}.$  Therefore,
\begin{align*}
\displaystyle\sum_{\underline{i}\in I_1}a_{\underline{i}}\widehat{f_{1}}^{i_1}\widehat{t_4}^{i_4}
\widehat{t_5}^{i_5}\widehat{t_6}^{i_6}
+\displaystyle\sum_{\underline{j}\in I_2}b_{\underline{j}}\widehat{t_3}^{i_3}\widehat{t_4}^{i_4}
\widehat{t_5}^{i_5}\widehat{t_6}^{i_6}=&\displaystyle\sum_{\underline{l}\in J}q^\bullet c_{\underline{l}}
\widehat{f_{1}}^{i_1+1}\widehat{t_3}^{i_3+1}
\widehat{t_4}^{i_4}
\widehat{t_5}^{i_5+1}\widehat{t_6}^{i_6}\\
&+\sum_{\underline{l}\in J}q^\bullet\beta ac_{\underline{l}}\widehat{f_{1}}^{i_1}
\widehat{t_3}^{i_3}\widehat{t_4}^{i_4-1}
\widehat{t_5}^{i_5+1}\widehat{t_6}^{i_6-1}\\
&-\displaystyle\sum_{\underline{l}\in J}\alpha c_{\underline{l}}\widehat{f_{1}}^{i_1}\widehat{t_3}^{i_3}
\widehat{t_4}^{i_4}
\widehat{t_5}^{i_5}\widehat{t_6}^{i_6}.
\end{align*}
 Suppose that there exists $(i_1,i_3,i_4,i_5,i_6)\in J$ such that $c_{(i_1,i_3,i_4,i_5,i_6)}\neq 0.$\newline Let $(w_1,w_3,w_4,w_5,w_6)\in J$ be the greatest element (in the lexicographic order on 
$\mathbb{N}^2\times \mathbb{Z}^3$) of  $J$ such that $c_{(w_1,w_3,w_4,w_5,w_6)}\neq 0.$ Since
   $\left( \widehat{f_{1}}^{k_1}\widehat{t_3}^{k_3}
\widehat{t_4}^{k_4}\widehat{t_5}^{k_5}\widehat{t_6}^{k_6}\right) _{(k_1,k_3,\cdots,k_6)\in \mathbb{N}^2\times \mathbb{Z}^3}$ is a basis of  $A^{(4)}S_4^{-1},$ it implies that the coefficients of  $\widehat{f_{1}}^{w_1+1}\widehat{t_3}^{w_3+1}\widehat{t_4}^{w_4}\widehat{t_5}^{w_5+1}
  \widehat{t_6}^{w_6}$ in the above equality can be identified as: $q^\bullet c_{(w_1,w_3,w_4,w_5,w_6)}=0$. Hence, $c_{(w_1,w_3,w_4,w_5,w_6)}=0,$ a contradiction! Therefore, $c_{(i_1,i_3,i_4,i_5,i_6)}= 0$ for all $(i_1,i_3,i_4,i_5,i_6)\in J.$  This further implies that
  $$\displaystyle\sum_{\underline{i}\in I_1}a_{ \underline{i}}\widehat{f_{1}}^{i_1}\widehat{t_4}^{i_4}
\widehat{t_5}^{i_5}\widehat{t_6}^{i_6}
+\displaystyle\sum_{\underline{j}\in I_2}b_{\underline{j}}\widehat{t_3}^{i_3}\widehat{t_4}^{i_4}
\widehat{t_5}^{i_5}\widehat{t_6}^{i_6}
=0.$$
It follows from the previous proposition that  $a_{\underline{i}}$ and $ b_{\underline{j}}$ are all zero.   In conclusion, $\mathcal{B}_4$ is a linearly independent set. 
  \end{proof}

\textbf{Basis for $\cR_5.$}  We will identify $\cR_5$ with ${A_{\alpha}^{(5)}S_5^{-1}}/{\langle \widehat{\Omega}_{2}-\beta\rangle},$  where  $A_\alpha^{(5)}S_5^{-1}$  $=\dfrac{{A^{(5)}S_5^{-1}}}{{\langle \Omega_{1}-\alpha \rangle}}.$ Note, the canonical images of $E_{i,j}$ (resp. $T_i$) in $A_{\alpha}^{(5)}$ will be denoted by $\widehat{e_{i,j}}$ (resp. $\widehat{t_i}$).
 We now find a basis for $A_\alpha^{(5)}S_5^{-1}.$ Recall that  $\Omega_{1}=\zbar_1T_3T_5+a \zbar_2T_5$ and $\Omega_2=\zbar_2T_4T_6+bT_3^3T_6$ in $A^{(5)}$ (remember, $\zbar_1:=E_{1,5}$ and $\zbar_2:=E_{2,5}$). Since $z_2t_4t_6+bt_3^3t_6=\beta$  and $\widehat{z_1}\widehat{t_3}\widehat{t_5}+a\widehat{z_2}\widehat{t_5}=\alpha$ in $\cR_5$ and  $A_\alpha^{(5)}S_5^{-1}$ respectively, we have the relation $\widehat{z_{2}}=
 \dfrac{1}{a}\left( \alpha\widehat{t_5}^{-1}-\widehat{z_{1}}\widehat{t_3}\right) $ in $A_\alpha^{(5)}S_5^{-1}$ and, in $\cR_5,$ we have the following two relations:  
 \begin{align}
 z_{2}&=\dfrac{1}{a}\left( \alpha t_5^{-1}-z_{1}t_3\right). \label{de2}
 \\
 t_3^3&=\dfrac{1}{b}\left( \beta t_6^{-1}-z_{2}t_4\right) =\dfrac{\beta}{b} t_6^{-1}-\dfrac{q^3\alpha}{ab}t_4t_5^{-1}+\dfrac{1}{ab} z_{1}t_3t_4.\label{de3}
 \end{align}
 \begin{proposition}
 \label{c4p3}
The set $\mathfrak{S}_5=\left\lbrace \widehat{z_{1}}^{i_1}\widehat{t_3}^{i_3}\widehat{t_4}^{i_4}\widehat{t_5
}^{i_5}\widehat{t_6}^{i_6}\mid {(i_1,i_3,\cdots,i_6)\in \mathbb{N}^3\times \mathbb{Z}^2}\right\rbrace $ is a $\k-$basis of $A_\alpha^{(5)}S_5^{-1},$ where $\alpha\in \k.$
\end{proposition}
\begin{proof}
The proof is similar to that of Proposition \ref{c4p2} and so is left to the reader. Details can be found in \cite{io-thesis}. 
\end{proof}

\begin{proposition} 
\label{c4p20}
The set $\mathcal{B}_5=\left\lbrace 
z_{1}^{i_1}t_3^{\xi}t_4^{i_4}t_5^{i_5}t_6^{i_6}
\mid (\xi, i_1,i_4,i_5,i_6)\in \{0,1,2\}\times \mathbb{N}^2\times\mathbb{Z}^2\right\rbrace $ is a $\k-$basis of $\cR_5$.
\end{proposition}
\begin{proof}
The proof is similar to that of Proposition \ref{c4p1} and so is left to the reader. Details can be found in \cite{io-thesis}. 
\end{proof}

We note for future reference the following immediate corollary.

\begin{corollary}
\label{c4c1}
Let $I$ be a finite subset of $\{0,1,2\}\times \mathbb{N}\times\mathbb{Z}^3$ and $(a_{(\xi,\underline{i})})_{\underline{i}\in I}$ be a family of scalars.  If $$\sum_{(\xi,\underline{i})\in I}a_{(\xi.\underline{i})} z_{1}^{i_1}t_3^\xi t_4^{i_4}
t_5^{i_5}t_6^{i_6}=0, 
$$ then $a_{(\xi,\underline{i})}=0$ for all $(\xi, \underline{i})\in I.$
\end{corollary}

\begin{remark}
We were not successful in finding a basis for $\cR_6.$ However, this has no effect on our main results in this section. Since $\cR_7=A_{\alpha,\beta},$ we already have a basis for $\cR_7$ (Proposition \ref{c3p5}).
\end{remark}


\subsection{Derivations of $A_{\alpha,\beta}$}
We are now going to study the derivations of $A_{\alpha,\beta}.$ We will only treat the case when  both $\alpha$ and $\beta$ are non-zero, and mention results when either $\alpha$ or $\beta$ is zero without details. 

 Throughout this subsection, we assume that $\alpha$ and  $\beta$ are non-zero. Let $\der(A_{\alpha,\beta})$ denote the $\k-$derivations of $A_{\alpha,\beta}$ and $D\in \der(A_{\alpha,\beta}).$ Via localization, $D$ extends uniquely to a derivation of each of the series of algebras in \eqref{c4ee1}. Therefore, $D$ extends to  a derivation of the  quantum torus $\cR_3=\k_{q^N}[t_3^{\pm1},t_4^{\pm1},t_5^{\pm1},t_6^{\pm 1}].$ It follows from \cite[Corollary 2.3]{op} that $D$ can uniquely be written as: $$D=\text{ad}_x+\delta,$$ where $x\in \cR_3,$ and $\delta$ is a scalar derivation of $\cR_3$ defined as $\delta(t_i)=\lambda_it_i$ for each  $i=3,4,5,6.$ Note, $\lambda_i\in Z(\cR_3)=\k.$ Also, $\text{ad}_x$ is an inner derivation of $\cR_3$ defined as $\text{ad}_x(L)=xL-Lx$ for all $L\in \cR_3.$ 

We aim to describe $D$ as a derivation of $A_{\alpha,\beta}=\cR_7.$ We do this in several steps. 

Before starting the process we note the following relations that will be used in this section. They all follow from \cite[Lemme 5.3.2]{ca}.

\begin{remark}
\label{rrr}
Recall the notations:
\begin{align*}
f_1:&=e_{1,4} & F_1:&=E_{1,4}\\
z_1:&=e_{1,5} & \zbar_1:&=E_{1,5} \\
z_2:&=e_{2,5} & \zbar_2:&=E_{2,5}.
\end{align*}

Then 
\begin{align*}
f_1&=t_1-at_2t_3^{-1} & e_{3,6}&=t_3-at_4t_5^{-1}\\
z_1&=f_1-st_3^2t_4^{-1} &e_1&=e_{1,6}-rt_5t_6^{-1}\\
z_2&=t_2-bt_3^3t_4^{-1}& e_3&=e_{3,6}-st_5^2t_6^{-1}\\
e_{1,6}&=z_1-he_{3,6}t_5^{-1}-gt_4t_5^{-2}& e_4&=t_4-bt_5^3t_6^{-1}.
\end{align*}
\end{remark}

We first describe $D$ as a derivation of $\cR_4.$ 
\begin{lemma}
\begin{itemize}\label{c4l1}
\item[1.] $x\in \cR_4.$
\item[2.] $\lambda_5=\lambda_4+\lambda_6,$ $\delta(f_{1})=-(\lambda_3+\lambda_5)f_{1}$  
and $\delta(t_2)=-\lambda_5t_2.$ 
\item[3.] Set $\lambda_1:=-(\lambda_3+\lambda_5)$ and $\lambda_2:=-\lambda_5.$ Then, $D(e_{\kappa,4})=\text{ad}_x(e_{\kappa,4})+\lambda_\kappa e_{\kappa,4}$ for  all $\kappa\in \{1,\cdots, 6\}.$
\end{itemize}
\end{lemma}
\begin{proof}
1. Set $\mathcal{Q}_q:=\k_{q^\mu}[t_4^{\pm 1},t_5^{\pm 1},t_6^{\pm 1}],$ where $\mu$ is the  skew-symmetric sub-matrix of $N$ (see Section \ref{c3p3}) obtained by deleting the first row and first column of $N$. Observe that $\mathcal{Q}_q$ is a subalgebra of both $\cR_3$ and $\cR_4$ with central element $$z:=t_4t_5^{-1}t_6.$$ Furthermore, since $\cR_3$ is a quantum torus, we can present it as a free left $\mathcal{Q}_q-$module with basis $(t_3^s)_{s\in\mathbb{Z}}.$ With this presentation, $x\in \cR_3$ can be written as 
$$x=\sum_{s\in \mathbb{Z}}y_st_3^s,$$ where $y_s\in \mathcal{Q}_q.$ 
Set $$x_{+}:=\sum_{s\geq 0}y_st_3^s \ \ \text{and} \ \ x_{-}:=\sum_{s<0}y_st_3^s.$$ Clearly, $x=x_{+}+x_{-}.$ Obviously, $x_+\in \cR_4,$ hence we aim to also show that $x_-$ belongs to $\cR_4$ by following a pattern similar to \cite[Proposition 7.1.2]{ak}. 
As $D$ is a derivation of $\cR_4$, we have that 
$D(z^j)\in \cR_4$ for all $j\in \mathbb{N}_{\geq 1}.$ Now $D(z^j)=\text{ad}_{x}(z^j)+\delta(z^j)=\text{ad}_{x_+}(z^j)+\text{ad}_{x_-}(z^j)+\delta(z^j).$
Observe that $\text{ad}_{x_+}(z^j)\in \cR_4$; since $x_+, z^j \in \cR_4.$ 
 Also, $\delta(z)=\delta(t_4t_5^{-1}t_6)=(\lambda_4-\lambda_5+\lambda_6)t_4t_5^{-1}t_6=(\lambda_4-\lambda_5+\lambda_6)z,$ where 
 $\lambda_4,\lambda_5,\lambda_6\in \k.$
It follows that $\delta(z^j)=j(\lambda_4-\lambda_5+\lambda_6)z^j\in \cR_4.$ 
We can therefore conclude that each $\text{ad}_{x_-}(z^j)$ belongs to $\cR_4$ since $D(z^j), \text{ad}_{x_+}(z^j), \delta(z^j)\in \cR_4.$
We have:
$$\text{ad}_{x_-}(z^j)=x_-z^j-z^jx_-=\sum_{s=-1}^{-n}y_st_3^sz^j-\sum_{s=-1}^{-n}y_sz^jt_3^s.$$
One can verify that $zt_3=q^{-2}t_3z.$
Therefore, $$\text{ad}_{x_-}(z^j)=\sum_{s=-1}^{-n}(1-q^{-2js})y_st_3^sz^j, \ \text{hence,} \ \ \text{ad}_{x_-}(z^j)z^{-j}=\sum_{s=-1}^{-n}(1-q^{-2js})y_st_3^s.$$
Set $\nu_j:=\text{ad}_{x_-}(z^j)z^{-j}\in \cR_4.$  It follows that 
$$\nu_j=\sum_{s=-1}^{-n}(1-q^{-2js})y_st_3^s,$$ for each $j\in \{1,\cdots,n\}.$ One can therefore write the above equality as a matrix equation as follows:
$$
\begin{bmatrix}
(1-q^2)&(1-q^4)&(1-q^6)&\cdots &(1-q^{2n})\\
(1-q^4)&(1-q^8)&(1-q^{12})&\cdots &(1-q^{4n})\\
(1-q^6)&(1-q^{12})&(1-q^{18})&\cdots &(1-q^{6n})\\
\vdots&\vdots&\vdots&\ddots&\vdots\\
(1-q^{2n})&(1-q^{4n})&(1-q^{6n})&\cdots &(1-q^{2n^2})\\
\end{bmatrix}
\begin{bmatrix}
y_{-1}t_3^{-1}\\
y_{-2}t_3^{-2}\\
y_{-3}t_3^{-3}\\
\vdots\\
y_{-n}t_3^{-n}\\
\end{bmatrix}=
\begin{bmatrix}
\nu_1\\
\nu_2\\
\nu_3\\
\vdots\\
\nu_n\\
\end{bmatrix}.
$$
We already know that each $\nu_j$ belongs to $\cR_4.$ 
We want to show that $y_st_3^s$ also belongs to $\cR_4$ for each $s\in\{-1,\cdots, -n\}.$ 
To establish this, it is sufficient to show that the coefficient matrix of the above  matrix equation is invertible.
Let $U$ represent this matrix. Thus,
$$
U=\begin{bmatrix}
(1-q^2)&(1-q^4)&(1-q^6)&\cdots &(1-q^{2n})\\
(1-q^4)&(1-q^8)&(1-q^{12})&\cdots &(1-q^{4n})\\
(1-q^6)&(1-q^{12})&(1-q^{18})&\cdots &(1-q^{6n})\\
\vdots&\vdots&\vdots&\ddots&\vdots\\
(1-q^{2n})&(1-q^{4n})&(1-q^{6n})&\cdots &(1-q^{2n^2})\\
\end{bmatrix}.
$$
Apply row operations: $ -r_{n-1}+r_n\rightarrow r_n,\cdots, -r_2+r_3\rightarrow r_3, -r_1+r_2\rightarrow r_2$ to $U$ to obtain: 
$$
U'=\begin{bmatrix}
l_1&l_2&l_3&\cdots &l_n\\
q^2l_1&q^4l_2&q^6l_3&\cdots &q^{2n}l_n\\
q^4l_1&q^8l_2&q^{12}l_3&\cdots &q^{4n}l_n\\
\vdots&\vdots&\vdots&\ddots&\vdots\\
q^{2(n-1)}l_1&q^{4(n-1)}l_2&q^{6(n-1)}l_3&\cdots &q^{2n(n-1)}l_n\\
\end{bmatrix},
$$
where $l_i:=1-q^{2i};$ $i\in \{1,2,\cdots,n\}.$
Clearly, $U'$ is similar to a Vandermonde matrix (since the terms in each column form a geometric sequence) which is well known to be invertible when all parameters are pairwise distinct (this is the case here as $q$ is not a root of unity). This further implies that $U$ is invertible. 
So each $y_{s}t_3^{s}$ is a linear combination of the $\nu_j\in \cR_4.$ As a result, $y_{s}t_3^{s}\in \cR_4$ for all $s\in \{-1,\cdots, -n\}.$
Consequently, $x_{-}=\sum_{s=-1}^{-n}y_st_3^s\in \cR_4$ as desired.

2. Recall that $\delta (t_\kappa)=\lambda_\kappa t_\kappa$ for all $\kappa\in \{3,4,5,6\}$ and $\lambda_\kappa\in \k.$
From Remark \ref{rrr}, we have that
$f_{1}=t_1-at_2t_3^{-1}.$ Recall from Section \ref{c3p3} that 
$t_1=\alpha t_5^{-1}t_3^{-1}$ and $t_2^{-1}=\beta t_6^{-1}t_4^{-1}$ in $\cR_3=\mathscr{A}_{\alpha,\beta}.$ As a result, $f_1=\alpha t_5^{-1}t_3^{-1}-a\beta t_6^{-1}t_4^{-1}t_3^{-1}.$ Hence,
\begin{align}
\delta(f_{1})
=&-(\lambda_5+\lambda_3)\alpha t_5^{-1}t_3^{-1}+(\lambda_6+\lambda_4+\lambda_3)a\beta t_6^{-1}t_4^{-1}t_3^{-1}. \label{c4e2}
\end{align}
From Proposition \ref{c4p1}, the set $\mathcal{B}_4=\left\lbrace f_{1}^{i_1}t_4^{i_4}t_5^{i_5}
t_6^{i_6}, t_3^{i_3}t_4^{i_4}t_5^{i_5}
t_6^{i_6}\mid i_1,i_3\in \mathbb{N} \ \text{and} \ i_4, i_5, i_6\in \mathbb{Z}\right\rbrace$ is a $\k-$basis of $\cR_4.$ Since $t_4, t_5$ and $t_6$ $q$-commute with $f_1$ and $t_3,$ one can 
also write $\delta(f_{1})\in \cR_4$ in terms of $\mathcal{B}_4$ as follows:

\begin{align}
\label{de4}
\delta(f_{1})=&\displaystyle\sum_{r> 0}a_rf_{1}^r+\displaystyle\sum_{s\geq 0}b_st_3^s,
\end{align}
where $a_r$ and $b_s$ belong to $\mathcal{Q}_q=\k_{q^\mu}[t_4^{\pm 1},t_5^{\pm 1},t_6^{\pm 1}].$
\begin{align}
f_{1}^r&=(\alpha t_5^{-1}t_3^{-1}-a\beta t_6^{-1}t_4^{-1}t_3^{-1})^r=\sum_{i=0}^r {r\choose i}_{q^\bullet}(\alpha t_5^{-1}t_3^{-1})^i(-a\beta t_6^{-1}t_4^{-1}t_3^{-1})^{r-i} \nonumber\\
&=\sum_{i=0}^r {r\choose i}_{q^\bullet}\alpha^i(-a\beta)^{r-i}q^{\frac{1}{2}i(i-1)+\frac{3}{2}(r-i)(r-i-1)+3i(i-r)}t_5^{-i}(t_6^{-1}t_4^{-1})^{r-i}t_3^{-r}\nonumber\\
&=c_rt_3^{-r},\label{c4e23}
\end{align}
where 
\begin{align}
c_r=\displaystyle\sum_{i=0}^r {r\choose i}_{q^\bullet}q^{\frac{1}{2}i(i-1)+\frac{3}{2}(r-i)(r-i-1)+3i(i-r)}\alpha^i(-a\beta)^{r-i} t_5^{-i}(t_6^{-1}t_4^{-1})^{r-i} \in \mathcal{Q}_q\setminus \{0\}.\label{c4e24}
\end{align}
Substitute \eqref{c4e23} into \eqref{de4} to obtain;
\begin{align}
\delta(f_{1})=&\displaystyle\sum_{r> 0}a_rc_rt_3^{-r}+\displaystyle\sum_{s\geq 0}b_st_3^s.\label{c4e3}
\end{align}
One can rewrite \eqref{c4e2} as
\begin{equation}
\delta (f_{1})=dt_3^{-1}, \label{c4e4}
\end{equation} 
where $d=-(\lambda_5+\lambda_3)\alpha t_5^{-1}+(\lambda_6+\lambda_4+\lambda_3)a\beta t_6^{-1}t_4^{-1}\in\mathcal{Q}_q.$
Comparing \eqref{c4e3} to \eqref{c4e4} shows that 
$b_s=0$ for all $s\geq 0,$  and $a_rc_r=0$ for all $r\neq 1$. Therefore $\delta(f_{1})=a_1c_1t_3^{-1}.$ Moreover, from \eqref{c4e24}, $c_1=-a\beta t_6^{-1}t_4^{-1}+\alpha t_5^{-1}.$ Hence, 
\begin{align}
\delta(f_{1})=a_1c_1t_3^{-1}&=a_1(-a\beta t_6^{-1}t_4^{-1}+\alpha t_5^{-1})t_3^{-1}
=a_1\alpha t_5^{-1}t_3^{-1}-a_1a\beta t_6^{-1}t_4^{-1}t_3^{-1}.\label{c4e25}
\end{align}
Comparing \eqref{c4e25} to \eqref{c4e2} reveals that
$a_1=-(\lambda_5+\lambda_3)=-(\lambda_6+\lambda_4+\lambda_3).$ Consequently, $\lambda_5=\lambda_6+\lambda_4.$ Hence,
$\delta(f_{1})=-(\lambda_5+\lambda_3)\alpha t_5^{-1}
t_3^{-1}+(\lambda_5+\lambda_3)a\beta t_6^{-1}
t_4^{-1}t_3^{-1}=-(\lambda_5+\lambda_3)f_{1}.$
Finally, since $t_2=\beta t_6^{-1}t_4^{-1}$ in $\cR_4,$ it follows that $\delta(t_2)=
-(\lambda_6+\lambda_4)\beta t_6^{-1}t_4^{-1}=-(\lambda_6+\lambda_4)
t_2=-\lambda_5t_2.$

3. Set $\lambda_1:=-(\lambda_3+\lambda_5)$ and $\lambda_2:=-\lambda_5.$ it follows from points (1) and (2) that $D(e_{\kappa,4})=\text{ad}_x(e_{\kappa,4})+\delta(e_{\kappa,4})=\text{ad}_x(e_{\kappa,4})
+\lambda_\kappa e_{\kappa,4}$ for all $\kappa\in \{1,\cdots, 6\}.$ 
In conclusion, $D=\text{ad}_x+\delta,$ with $x\in \cR_4.$ 
\end{proof}


We proceed to describe $D$ as a derivation of $\cR_5.$
\begin{lemma}
\label{ev5}
\begin{itemize}
\item[1.] $x\in \cR_5.$
\item[2.] $\lambda_4=3\lambda_3+\lambda_5, \ \lambda_6=-3\lambda_3, \ \delta(z_{1})=-(\lambda_3+\lambda_5)z_{1}$ and $ \delta(z_{2})=-\lambda_5z_{2}.$ 
\item[3.] Set $\lambda_1:=-(\lambda_3+\lambda_5), \ \lambda_2:=-\lambda_5$ and $\lambda_6:=-3\lambda_3.$ Then, $D(e_{\kappa,5})=\text{ad}_x(e_{\kappa,5})+\lambda_\kappa e_{\kappa,5}$ for  all $\kappa\in \{1,\cdots, 6\}.$
\end{itemize}
\end{lemma}
 
\begin{proof}
In this proof, we denote $\underline{\upsilon}:=(i,j,k,l)\in \mathbb{N}\times \mathbb{Z}^3.$

1.   We already know that 
 $x\in \cR_4=\cR_5[t_4^{-1}].$ Given the basis $\mathcal{B}_5$ of $\cR_5$ (Proposition \ref{c4p20}),  $x$ can be written as 
 $x=\displaystyle\sum_{(\xi, \underline{\upsilon})\in I}a_{(\xi,\underline{\upsilon})}z_{1}^{i}t_3^{\xi}
t_4^{j}t_5^{k}t_6^{l},$ 
where $I$ is a finite subset of $\{0,1,2\}\times \mathbb{N}\times \mathbb{Z}^3$ and
 $a_{(\xi, \underline{\upsilon})}$ are scalars.   Write 
 $x=x_-+x_+,$ where 
$$x_+=\sum_{\substack{(\xi, \underline{\upsilon})\in I\\ j\geq 0}} a_{(\xi, {\underline{\upsilon})}}z_{1}^{i}t_3^{\xi}
 t_4^{j}t_5^{k}t_6^{l}\ \  \text{and} \ \ x_-=\displaystyle \sum_{\substack{(\xi, \underline{\upsilon})\in I \\ j<0}}a_{(\xi,\underline{\upsilon})}z_{1}^{i}t_3^{\xi}
t_4^{j}t_5^{k}t_6^{l}.$$    
 Suppose that there exists a minimum  $j_0<0$ such that  $a_{(\xi, i,j_0,k,l)}\neq 0$ for some
$(\xi,i,j_0,k,l)\in I$ and $a_{(\xi,i,j,k,l)}=0$ for all $(\xi,i,j_0,k,l)\in I$ with $j<j_0.$  Given this assumption, write
\begin{align*}
 x_-=\displaystyle\sum_{ \substack{(\xi,\underline{\upsilon})\in I\\ j_0 \leq j\leq -1}}a_{(\xi,\underline{\upsilon})}z_{1}^{i}t_3^{\xi}
 t_4^{j}t_5^{k}t_6^{l}.
\end{align*} 
Now, $D(t_6)=\text{ad}_{x_+}(t_6)+\text{ad}_{x_-}(t_6)+
 \delta(t_6)\in \cR_5.$ This implies that $\text{ad}_{x_-}(t_6)\in \cR_5,$ since $\text{ad}_{x_+}(t_6)+\delta(t_6)=\text{ad}_{x_+}(t_6)+
 \lambda_6t_6\in \cR_5.$ We aim to show that $x_-=0.$ Since $t_6$ is normal in $\cR_5,$ one can easily verify that
 \begin{align*}
 \text{ad}_{x_-}(t_6)&=
\sum_{\substack{(\xi,\underline{\upsilon})\in I\\ j_0 \leq j\leq -1}}\left( q^{3(i-j-k)}-1\right) a_{(\xi, \underline{\upsilon})}
z_{1}^{i}t_3^{\xi}
t_4^{j}t_5^{k}t_6^{l+1}.
\end{align*} 
Set $\underline{w}:=(i,j,k,l)\in \mathbb{N}^2\times \mathbb{Z}^2.$ 
One can equally write $\text{ad}_{x_-}(t_6)\in \cR_5$ in terms of the basis $\mathcal{B}_5$  of $\cR_5$ (Proposition \ref{c4p20}) as:
 \begin{align*}
\text{ad}_{x_-}(t_6)&= \sum_{(\xi,\underline{w})\in J} b_{(\xi, {\underline{w}})}z_{1}^{i}
t_3^{\xi}t_4^{j}t_5^{k}t_6^{l},
 \end{align*}
 where  $ J$ is a finite subset of $\{0,1,2\}\times \mathbb{N}^2\times \mathbb{Z}^2$ and $b_{(\xi, \underline{w})}$ are all scalars.  It follows that
 $$
 \sum_{\substack{(\xi,\underline{\upsilon})\in I\\ j_0 \leq j\leq -1}}\left( q^{3(i-j-k)}-1\right) a_{(\xi, \underline{\upsilon})}
z_{1}^{i}t_3^{\xi}
t_4^{j}t_5^{k}t_6^{l+1}=  \sum_{(\xi,\underline{w})\in J} b_{(\xi, {\underline{w}})}z_{1}^{i}
t_3^{\xi}t_4^{j}t_5^{k}t_6^{l}.
 $$
 As $\mathcal{B}_5$ is a basis for $\cR_5,$ we deduce from Corollary \ref{c4c1} that 
 $\left(z_{1}^{i}
t_3^{\xi}t_4^{j}t_5^{k}t_6^{l}\right)_{(i\in \mathbb{N}; j,k,l\in \mathbb{Z}; \xi\in \{0,1,2\})}$ is a basis for $\cR_5[t_4^{-1}].$
  Now, at $j=j_0,$ denote $\underline{\upsilon}=(i,j,k,l)$ by $\underline{\upsilon}_0:=
 (i,j_0,k,l).$ Since $\underline{v}_0\in \mathbb{N}\times \mathbb{Z}^3$ (with $j_0<0$) and $\underline{w}=(i,j,k,l)\in \mathbb{N}^2\times \mathbb{Z}^2$ (with $j\geq 0$), it follows from the above equality that, at $\underline{\upsilon}_0,$  we must have
 $$\left( q^{3(i-j_0-k)}-1\right) a_{(\xi, \underline{\upsilon}_0)}=0.$$
From our initial assumption, the coefficients $a_{(\xi, \underline{\upsilon}_0)}$ are all not zero, therefore $q^{3(i-j_0-k)}-1=0. $ This implies   that
 \begin{align}
 k=i-j_0,\label{2e}
 \end{align}
 for some $(\xi,\underline{\upsilon}_0)\in I.$ 
 
In a similar manner, $D(t_3)=\text{ad}_{x_+}(t_3)+\text{ad}_{x_-}(t_3)+
 \delta(t_3)\in \cR_5.$ This implies that $\text{ad}_{x_-}(t_3)\in \cR_5,$ since $\text{ad}_{x_+}(t_3)+\delta(t_3)=\text{ad}_{x_+}(t_3)+
 \lambda_3t_3\in \cR_5.$  We have that
 \begin{align*}
 \text{ad}_{x_-}(t_3)=
\sum_{\substack{(\xi,\underline{\upsilon})\in I\\ j_0\leq j\leq -1}}a_{(\xi,\underline{\upsilon})}
z_{1}^{i}t_3^{\xi}
t_4^{j}t_5^{k}t_6^{l}t_3-
 \sum_{\substack{(\xi,\underline{\upsilon})\in I\\ j_0 \leq j\leq -1}} a_{(\xi, \underline{\upsilon})}
t_3z_{1}^{i}t_3^{\xi}
t_4^{j}t_5^{k}t_6^{l}.
\end{align*} 
One can deduce from Lemma \ref{c3l3}(3a) that $$t_3z_1^{i}=q^{-i}z_1^it_3+d_2[i]z_1^{i-1}z_2,$$ where
$d_2[i]=q^{1-i}d_2[1]\left(\dfrac{1-q^{-2i}}{1-q^{-2}} \right),$ $d_2[1]=-(q+q^{-1}+q^{-3})$ and $d_2[0]=0.$ 
Therefore, the above expression for $ \text{ad}_{x_-}(t_3)$ can be expressed as:
\begin{align*}
\text{ad}_{x_-}(t_3)=&\sum_{\substack{(0,\underline{\upsilon})\in I\\ j_0 \leq j\leq -1}}f[i,j,k] a_{(0,\underline{\upsilon})}z_{1}^{i}t_3t_4^{j}t_5^{k}
t_6^{l}+ \sum_{\substack{(1,\underline{\upsilon})\in I\\ j_0 \leq j\leq -1}}f[i,j,k] a_{(1,\underline{\upsilon})}z_{1}^{i}t_3^{2}t_4^{j}t_5^{k}
t_6^{l} +
\\
&+\sum_{\substack{(2,\underline{\upsilon})\in I\\ j_0 \leq j\leq -1}}f[i,j,k] a_{(2,\underline{\upsilon})}z_{1}^{i}t_3^{3}t_4^{j}t_5^{k}
t_6^{l}-\sum_{\substack{(\xi,\underline{\upsilon})\in I\\ j_0 \leq j\leq -1}}a_{(\xi,\underline{\upsilon})}d_2[i]z_{1}^{i-1}z_{2}t_3^{\xi}t_4^{j}t_5^{k}
t_6^{l},
\end{align*}
where $f[i,j,k]:= q^{-(k+3j)}-q^{-i}.$
Recall from \eqref{de2} and \eqref{de3} that
$$z_{2}=\dfrac{1}{a}\left( \alpha t_5^{-1}-z_{1}t_3\right) \ \ \text{and} \ \ t_3^3=\dfrac{\beta}{b} t_6^{-1}-\dfrac{q^3\alpha}{ab}t_4t_5^{-1}+\dfrac{1}{ab} z_{1}t_3t_4,$$ where $a$ and $b$ are non-zero 
scalars (Appendix \ref{appenc}).
Using these two expressions, one can write $\text{ad}_{x_-}(t_3)$ in terms of the basis of $\cR_5$ as: 

\begin{align}
\text{ad}_{x_-}(t_3)=&\mathcal{K}+\sum_{\substack{(0,\underline{\upsilon}_0)\in I}}
g[i,j_0,k]a_{(0,{\underline{\upsilon}_0})}z_{1}^{i}t_3
t_4^{j_0}t_5^{k}t_6^{l}+\sum_{\substack{(1,\underline{\upsilon}_0)\in I}}
g[i,j_0,k]a_{(1,\underline{\upsilon}_0)}z_{1}^{i}t_3^{2}
t_4^{j_0}t_5^{k}t_6^{l}\nonumber
\\
&+ \sum_{\substack{(2,\underline{\upsilon}_0)\in I}}\frac{q^{\bullet}\beta}{b}  a_{(2,{\underline{\upsilon}_0})}g[i,j_0,k]
z_{1}^{i}
t_4^{j_0}t_5^{k}t_6^{l-1} - \sum_{\substack{(\xi,\underline{\upsilon}_0)\in I}}\frac{q^{\bullet}\alpha}{a} d_2[i]
a_{(\xi,\underline{\upsilon}_0)}z_{1}^{i-1}t_3^{\xi}
t_4^{j_0}t_5^{k-1}t_6^{l}\nonumber\\
=&\sum 1/b\left(  {q^{\bullet}\beta} g[i,j_0,k]a_{(2,i,j_0,k,l+1)}+(q^{\bullet}\alpha b{d_2[i+1]}/{a})a_{(0, i+1,j_0,k+1,l)}\right)   z_{1}^it_4^{j_0}t_5^kt_6^l\nonumber\\
&+\sum \left( g[i,j_0,k]a_{(0,i,j_0,k,l)}+(q^{\bullet}\alpha {d_2[i+1]}/{a}) a_{(1,i+1,j_0,k+1,l)}\right)  z_{1}^it_3t_4^{j_0}t_5^kt_6^l\nonumber\\
&+\sum \left( g[i,j_0,k]a_{(1,i,j_0,k,l)}+(q^{\bullet}\alpha {d_2[i+1]}/{a})a_{(2,i+1,j_0,k+1,l)}\right)  z_{1}^it_3^2t_4^{j_0}t_5^kt_6^l+\mathcal{K}
,\label{3e}
\end{align}
where $g[i,j_0,k]:=q^{-(k+3j_0)}-q^{-i}+{d_2[i]}/{a}$ and   $$\mathcal{K}\in \text{Span}\left( \mathcal{B}_5\setminus \{z_{1}^{i}t_3^{\xi}
 t_4^{j_0}t_5^{k}t_6^{l}\mid (\xi, i,j_0,k,l)\in \{0,1,2\}\times \mathbb{N}\times \mathbb{Z}^3 \}\right).$$
One can also write $\text{ad}_{x_-}(t_3)\in \cR_5$ in terms of the basis $\mathcal{B}_5$ of $\cR_5$ (Proposition \ref{c4p20}) as:
 \begin{align}
\text{ad}_{x_-}(t_3)=& \sum_{(\xi,\underline{w})\in J}b_{(\xi, \underline{w})} z_{1}^{i}
t_3^{\xi}t_4^{j}t_5^{k}t_6^{l},\label{1e}
 \end{align}
 where $J$ is a finite subset of $\{0,1,2\}\times \mathbb{N}^2\times \mathbb{Z}^2$, and $b_{(\xi, {\underline{w})}}\in \k.$ Recall: $\underline{w}=(i,j,k,l)\in \mathbb{N}^2\times \mathbb{Z}^2.$ Now, \eqref{3e} and \eqref{1e} imply that
\begin{align*}
 \sum_{(\xi,\underline{w})\in J}b_{(\xi, \underline{w})} &z_{1}^{i}
t_3^{\xi}t_4^{j}t_5^{k}t_6^{l}\\
=&\sum 1/b\left(  {q^{\bullet}\beta} g[i,j_0,k]a_{(2,i,j_0,k,l+1)}+(q^{\bullet}\alpha b{d_2[i+1]}/{a})a_{(0, i+1,j_0,k+1,l)}\right)   z_{1}^it_4^{j_0}t_5^kt_6^l\nonumber\\
&+\sum \left( g[i,j_0,k]a_{(0,i,j_0,k,l)}+(q^{\bullet}\alpha {d_2[i+1]}/{a}) a_{(1,i+1,j_0,k+1,l)}\right)  z_{1}^it_3t_4^{j_0}t_5^kt_6^l\nonumber\\
&+\sum \left( g[i,j_0,k]a_{(1,i,j_0,k,l)}+(q^{\bullet}\alpha {d_2[i+1]}/{a})a_{(2,i+1,j_0,k+1,l)}\right)  z_{1}^it_3^2t_4^{j_0}t_5^kt_6^l+\mathcal{K}.
\end{align*} 
We have already established that 
 $\left(z_{1}^{i}
t_3^{\xi}t_4^{j}t_5^{k}t_6^{l}\right)_{(i\in \mathbb{N}; j,k,l\in \mathbb{Z}; \xi\in \{0,1,2\})}$ is a basis for $\cR_5[t_4^{-1}].$ Given that  $\underline{v}_0=(i,j_0,k,l)\in \mathbb{N}\times \mathbb{Z}^3$ (with $j_0<0$) and $\underline{w}=(i,j,k,l)\in \mathbb{N}^2\times \mathbb{Z}^2$ (with $j\geq 0$), it follows that
 \begin{align}
 q^{\bullet}\beta g[i,j_0,k]a_{(2,i,j_0,k,l+1)}+(q^{\bullet}\alpha b {d_2[i+1]}/{a})a_{(0,i+1,j_0,k+1,l)}&=0.\label{2ee}\\
 g[i,j_0,k]a_{(0,i,j_0,k,l)}+(q^{\bullet}\alpha {d_2[i+1]}/{a})a_{(1,i+1,j_0,k+1,l)}&=0.\label{3ee}\\
 g[i,j_0,k]a_{(1,i,j_0,k,l)}+(q^{\bullet}\alpha {d_2[i+1]}/{a})a_{(2,i+1,j_0,k+1,l)}&=0.\label{4e}
 \end{align}
Suppose that there exists $(\xi, i,j_0,k,l)\in I$ such that $g[i,j_0,k]=0.$  Then, $$g[i,j_0,k]=q^{-(k+3j_0)}-q^{-i}+d_2[i]/a=0.$$ 
Note, $d_2[i]=d_2[1]q^{1-i}\left( \dfrac{1-q^{-2i}}{1-q^{-2}}\right),$  where $d_2[1]=-(q+q^{-1}+q^{-3})$ and $d_2[0]=0.$ Again, recall from Appendix \ref{appenc} that $a=(q^2+1+q^{-2})/(q^{-2}-1)=\dfrac{qd_2[1]}{1-q^{-2}}.$ Given these expressions for $d_2[i]$ and $a$, we have that
$$g[i,j_0,k]=q^{-(k+3j_0)}-q^{-i}+d_2[i]/a=q^{-3j_0-k}-q^{-3i}=0.$$ Since $q$ is not a root of unity, we get
\begin{align}
k=3(i-j_0).\label{4ee}
\end{align}
Comparing \eqref{4ee} to \eqref{2e} shows that
 $i-j_0=0$ which implies that $ i=j_0<0,$ a contradiction (note, $i\geq 0$). 
Therefore, $g[i,j_0,k]\neq 0$ for all $(\xi, i,j,k,l)\in I.$

Now, observe that if there exists $\xi\in \{0,1,2\}$ such that $a_{(\xi, i,j_0,k,l)}=0$ for all $(i,j_0,k,l)\in \mathbb{N}\times \mathbb{Z}^3$, then one can easily deduce from equations \eqref{2ee}, \eqref{3ee} and \eqref{4e} that $a_{(\xi, i,j_0,k,l)}=0$ for all $(\xi, i,j_0,k,l)\in I.$  This will contradict our initial assumption. Therefore,
there exists some $(i,j_0,k,l)\in \mathbb{N}\times \mathbb{Z}^3$ such that $a_{(\xi, i,j_0,k,l)}\neq 0$ for each $\xi\in \{0,1,2\}.$ Without loss of generality, let
$(u,j_0,v,w)$ be the greatest element in the lexicographic order on $\mathbb{N}\times \mathbb{Z}^3$ such that $a_{(0,u,j_0,v,w)}\neq 0$ and $a_{(0,i,j_0,k,l)}=0$ for all $i>u.$

From \eqref{3ee}, at $(i,j_0,k,l)=(u,j_0,v,w)$, we have: 
\begin{align*}
g[u,j_0,v]a_{(0,u,j_0,v,w)}+(q^{\bullet}\alpha {d_2[u+1]}/{a})a_{(1,u+1,j_0,v+1,w)}&=0.
\end{align*}
From \eqref{4e}, at $(i,j_0,k,l)=(u+1,j_0,v+1,w)$, we have: 
\begin{align*}
g[u+1,j_0,v+1]a_{(1, u+1,j_0,v+1,w)}+(q^{\bullet}\alpha {d_2[u+2]}/{a})a_{(2, u+2,j_0,v+2,w)}&=0.
\end{align*}
Finally, from \eqref{2ee}, at $(i,j_0,k,l)=(u+2,j_0,v+2,w-1),$ we have: 
\begin{align*}
q^{\bullet}\beta g[u+2,j_0,v+2]a_{(2,u+2,j_0,v+2,w)}+(q^{\bullet}\alpha b {d_2[u+3]}/{a})a_{(0,u+3,j_0,v+3,w-1)}&=0.
\end{align*}
Note: $ a, b,$ $\alpha,$ $\beta,$ $q^\bullet\neq 0;$  $g[i,j_0,k]\neq 0$ for all $(\xi, i,j_0,k,l)\in I;$ and $d_2[i]\neq 0$ for $i>0.$ Since $u+3>u,$ it follows from the above list of equations (starting from the last one)  that
$$a_{(0,u+3,j_0,v+3,w-1)}=0\Rightarrow 
a_{(2,u+2,j_0,v+2,w)}=0\Rightarrow a_{(1,u+1,j_0,v+1,w)}=0\Rightarrow 
a_{(0,u,j_0,v,w)}=0,$$ a contradiction!
Hence, $a_{(0,i,j_0,k,l)}=0$ for all $(i,j_0,k,l)\in \mathbb{N}\times \mathbb{Z}^3.$
From \eqref{2ee}, \eqref{3ee} and \eqref{4e}, one can easily conclude that $a_{(\xi, i,j_0,k,l)}=0$ for all  $(\xi, i,j_0,k,l)\in I.$   
This contradicts our initial assumption, hence $x_-=0.$ Consequently, $x=x_+\in \cR_5$ as desired.

2. From Remark \ref{rrr}, we have
$z_{2}=t_{2}-bt_3^3t_4^{-1}$.  Since $\delta(t_\kappa)=\lambda_\kappa t_\kappa, \ \kappa\in \{2,\cdots,6\}, $ with $\lambda_2:=-\lambda_5$ (see Lemma \ref{c4l1}), it follows that
\begin{align*}
\delta(z_{2})=&-\lambda_5t_{2}-b(3\lambda_3-\lambda_4)
t_3^3t_4^{-1}
=-\lambda_5z_{2}-b(3\lambda_3-\lambda_4+\lambda_5)
t_3^3t_4^{-1}.
\end{align*}
Furthermore, 
\begin{align*}
D(z_{2})&=\text{ad}_x(z_{2})+\delta(z_{2})
=\text{ad}_x(z_{2})
-\lambda_5z_{2}-b(3\lambda_3-\lambda_4+\lambda_5)
t_3^3t_4^{-1}\in \cR_5.
\end{align*}
Hence $b(3\lambda_3-\lambda_4+\lambda_5)
t_3^3t_4^{-1}\in \cR_5,$ since $\text{ad}_x(z_{2})
-\lambda_5z_{2}\in \cR_5.$
This implies that 
$b(3\lambda_3-\lambda_4+\lambda_5)
t_3^3\in \cR_5{t}_4$ (note, from Appendix \ref{appenc}, $b\neq 0$).  Set 
$w:=3\lambda_3-\lambda_4+\lambda_5.$ Suppose that $w\neq 0.$ From \eqref{de3}, we have:
$$t_3^3=\dfrac{\beta}{b} t_6^{-1}-\dfrac{q^3\alpha}{ab}t_4t_5^{-1}+\dfrac{1}{ab} z_{1}t_3t_4.$$
It follows that
 $$wbt_3^3=
{w\beta}t_6^{-1}-\frac{q^3w\alpha}{a}
t_4t_5^{-1}+\frac{w}{a}z_{1}t_3t_4\in \cR_5t_4.$$ 
Since $t_3^3, \ t_4t_5^{-1}$ and $z_1t_3t_4$ are all elements of $\cR_5t_4,$ it implies that
$t_6^{-1}\in \cR_5t_4.$ Hence, $1\in \cR_5t_4t_6.$  Using the basis  $\mathcal{B}_5$ of $\cR_5$ (Proposition \ref{c4p20}), this leads to a contradiction. Therefore, $w=0$. That is, $3\lambda_3-\lambda_4+\lambda_5=0,$ and so $\lambda_4=
3\lambda_3+\lambda_5.$ This further implies that $\delta(z_{2})=-\lambda_5 z_{2}$ as desired.

Again, from Lemma \ref{c4l1}, we have that $\delta(f_1)=-(\lambda_3+\lambda_5)f_1.$ 
Recall from Remark \ref{rrr} that $z_{1}=f_{1}-st_3^2t_4^{-1}.$
It follows that
\begin{align*}
\delta(z_{1})=&-(\lambda_3+\lambda_5)f_{1}-s(2\lambda_3-\lambda_4)t_3^2t_4^{-1}
=-(\lambda_3+\lambda_5)z_1-s(3\lambda_3-\lambda_4+\lambda_5)t_3^2t_4^{-1}\\
=&-(\lambda_3+\lambda_5)z_1-s(3\lambda_3-(3\lambda_3+\lambda_5)+\lambda_5)t_3^2t_4^{-1}
=-(\lambda_3+\lambda_5)z_{1}.
\end{align*}
Finally, we know that $\delta(t_6)=\lambda_6t_6.$ This implies that $ \delta(t_6^{-1})=-\lambda_6t_6^{-1}.$ From \eqref{de3}, we have that
$$t_3^3=\dfrac{\beta}{b} t_6^{-1}-\dfrac{q^3\alpha}{ab}t_4t_5^{-1}+\dfrac{1}{ab} z_{1}t_3t_4,$$ where 
$a$ and $b$ are non-zero scalars (Appendix \ref{appenc}). This implies that $$t_6^{-1}=\dfrac{b}{\beta}t_3^3+\dfrac{q^3\alpha }{a\beta}
t_4t_5^{-1}-\dfrac{1}{a\beta}z_1t_3t_4.$$
Given that $\delta(z_1)=-(\lambda_3+\lambda_5)z_1, \ \delta(t_3)=\lambda_3t_3, \ \delta(t_4)=(3\lambda_3+\lambda_5)t_4$ and $\delta(t_5)=\lambda_5t_5,$ 
applying  $\delta$ to the above relation gives
$$-\lambda_6t_6^{-1}=3\lambda_3\left( \dfrac{b}{\beta}t_3^3+\dfrac{q^3\alpha }{a\beta}
t_4t_5^{-1}-\dfrac{1}{a\beta}z_1t_3t_4\right).$$
It follows that $\lambda_6=-3\lambda_3$ as desired.

3. Set $\lambda_1:=-(\lambda_3+\lambda_5)$ and $\lambda_2:=-\lambda_5.$ It follows from points (1) and (2) that $D(e_{\kappa,5})=\text{ad}_x(e_{\kappa,5})+\delta (e_{\kappa,5})=\text{ad}_x(e_{\kappa,5})+\lambda_\kappa e_{\kappa,5}$ for  all $\kappa\in \{1,\cdots,6\}.$ In conclusion, $D=\text{ad}_x+\delta$ with $x\in \cR_5.$ 
\end{proof}

We are now ready to describe $D$ as a derivation of $A_{\alpha,\beta}.$
\begin{lemma}
\label{ev19}
\begin{itemize}
\item[1.] $x\in A_{\alpha,\beta}.$
\item[2.] $\delta(e_\kappa)=0$ for all
$\kappa\in \{1,\cdots, 6\}.$
\item[3.] $D=\text{ad}_x.$  
\end{itemize}
\end{lemma}
\begin{proof}
In this proof, we denote $\underline{\upsilon}:=(i,j,k,l)\in \mathbb{N}^2\times\mathbb{Z}^2.$  Also, recall from  DDA of $A_{\alpha,\beta}$ at the beginning of this section that $t_5=e_5$ and $t_6=e_6.$

1. Given the basis $\mathcal{B}$ of $A_{\alpha,\beta}$ (Proposition \ref{c3p5}), one can write
 $x\in \cR_5=A_{\alpha,\beta}[t_5^{-1},t_6^{-1}]$ as: $$x=\displaystyle\sum_{(\epsilon_1, \epsilon_2,\underline{\upsilon})\in I}a_{(\epsilon_1,\epsilon_2, \underline{\upsilon})}e_{1}^{i}e_{2}^{j}e_3^{\epsilon_1}
e_4^{\epsilon_2}t_5^{k}t_6^{l},$$
where $I$ is a finite subset of $\{0,1\}^2\times \mathbb{N}^2\times\mathbb{Z}^2,$ and $a_{(\epsilon_1,\epsilon_2,\underline{\upsilon})}$ are scalars. Write 
 $x=x_-+x_+,$ where 
 $$x_+=\sum_{\substack{(\epsilon_1, \epsilon_2,\underline{\upsilon})\in I\\ k, \ l\geq 0 }}a_{(\epsilon_1,\epsilon_2, \underline{\upsilon})}e_{1}^{i}e_{2}^{j}e_3^{\epsilon_1}
 e_4^{\epsilon_2}t_5^{k}t_6^{l},$$ and $$
x_-=\displaystyle\sum_{\substack{(\epsilon_1, \epsilon_2,\underline{\upsilon})\in I\\ k<0 \ \text{or} \ l<0}}a_{(\epsilon_1,\epsilon_2, \underline{\upsilon})}e_{1}^{i}e_{2}^{j}e_3^{\epsilon_1}
 e_4^{\epsilon_2}t_5^{k}t_6^{l}.$$

Suppose that there exists a minimum negative integer $k_0$ or $l_0$ such that 
$a_{(\epsilon_1,\epsilon_2, i,j,k_0,l)}\neq 0$ or $a_{(\epsilon_1,\epsilon_2, i,j,k,l_0)}\neq 0$  for some  $(\epsilon_1,\epsilon_2, i,j,k_0,l), (\epsilon_1,\epsilon_2, i,j,k,l_0)\in I,$  
and  $a_{(\epsilon_1,\epsilon_2, i,j,k,l)}=0$ whenever $k<k_0$ or $l<l_0.$  Write
 $$x_-=\displaystyle\sum_{\substack{(\epsilon_1, \epsilon_2,\underline{\upsilon})\in I\\
 k_0\leq k \leq -1 \ \text{or} \ l_0\leq l\leq -1}}a_{(\epsilon_1,\epsilon_2, \underline{\upsilon})}e_{1}^{i}e_{2}^{j}e_3^{\epsilon_1}
 e_4^{\epsilon_2}t_5^{k}t_6^{l}.$$ 
  Now, $D(e_3)=\text{ad}_{x_+}(e_3)+\text{ad}_{x_-}(e_3)+
 \delta(e_3)\in A_{\alpha,\beta}.$  From Remark \ref{rrr}, we have that
 $e_3=e_{3,6}-st_5^2t_6^{-1}$ and $e_{3,6}=t_3-at_4t_5^{-1}.$ Putting these two together gives
\begin{align*}
e_3
=t_3-at_4t_5^{-1}-st_5^2t_6^{-1}.
\end{align*} 
Again, from Remark \ref{rrr}, we also have that $t_4=e_4+bt_5^3t_6^{-1}.$ Note, $\delta(t_\kappa)
=\lambda_\kappa t_\kappa, \ \kappa\in \{3,4,5,6\}.$ Now,
 \begin{align}
 \delta(e_3)&=\lambda_3t_3-a(\lambda_4-\lambda_5)t_4t_5^{-1}-s(2\lambda_5-\lambda_6)t_5^2t_6^{-1}\nonumber\\
 &=\lambda_3(e_{3,6}+at_4t_5^{-1})+
 a(\lambda_5-\lambda_4)t_4t_5^{-1}+s(\lambda_6-2\lambda_5)
 t_5^2t_6^{-1}\nonumber\\
 &=\lambda_3e_{3,6}+a(\lambda_3-\lambda_4+\lambda_5)t_4
 t_5^{-1}+s(\lambda_6-2\lambda_5)t_5^2t_6^{-1}\nonumber\\
 &=\lambda_3(e_3+st_5^2t_6^{-1})+
 a(\lambda_3-\lambda_4+\lambda_5)(e_{4}+bt_5^3t_6^{-1})
t_5^{-1}+s(\lambda_6-2\lambda_5)t_5^2t_6^{-1}\nonumber\\
 &=\lambda_3e_3+\alpha_1e_4t_5^{-1}
 +\alpha_2t_5^2t_6^{-1},\label{fe1}
\end{align}  
 where $\alpha_1=a(\lambda_3-\lambda_4+\lambda_5)$ and $\alpha_2=s(\lambda_3-2\lambda_5+\lambda_6)+q^{-3}ab(\lambda_3-\lambda_4+
 \lambda_5).$
Therefore,
$D(e_3)=\text{ad}_{x_+}(e_3)+\text{ad}_{x_-}(e_3)+
 \lambda_3e_3+\alpha_1e_4t_5^{-1}
 +\alpha_2t_5^2t_6^{-1}\in A_{\alpha,\beta}.$  It follows that
 $D(e_3)t_5t_6=\text{ad}_{x_+}(e_3)t_5t_6+\text{ad}_{x_-}(e_3)t_5t_6+
 \lambda_3e_3t_5t_6+
 \alpha_1e_4t_6
 +q^3\alpha_2t_5^3\in A_{\alpha,\beta}.$ Hence, $\text{ad}_{x_-}(e_3)t_5t_6\in A_{\alpha,\beta},$ since 
 $\text{ad}_{x_+}(e_3)t_5t_6+
 \lambda_3e_3t_5t_6+
 \alpha_1e_4t_6
 +q^3\alpha_2t_5^3\in A_{\alpha,\beta}.$
 
Now,
  \begin{align}
 \text{ad}_{x_-}(e_3)&=\displaystyle\sum_{(\epsilon_1, \epsilon_2,\underline{\upsilon})\in I}a_{(\epsilon_1,\epsilon_2, \underline{\upsilon})}e_{1}^{i}e_{2}^{j}e_3^{\epsilon_1}
 e_4^{\epsilon_2}t_5^{k}t_6^{l}
 e_3-
\displaystyle\sum_{(\epsilon_1, \epsilon_2,\underline{\upsilon})\in I}a_{(\epsilon_1,\epsilon_2, \underline{\upsilon})}e_3e_{1}^{i}e_{2}^{j}e_3^{\epsilon_1}
 e_4^{\epsilon_2}t_5^{k}t_6^{l}.\label{c4e5}
 \end{align}
Using Lemma \ref{c3l3}, we have the following:
\begin{align}
t_5^{k}t_6^{l}e_3&=q^{-k}e_3t_5^{k}t_6^{l}+d_2[k]e_4
t_5^{k-1}t_6^{l}+d_3[l]t_5^{k+2}
t_6^{l-1},\label{c4e6}\\
e_3e_1^{i}e_2^{j}&=q^{-i-3j}e_1^{i}e_2^{j}e_3+d_2[i]e_1^{i-1}
e_2^{j+1},\label{c4e7}
\end{align}
 Substitute \eqref{c4e6} and \eqref{c4e7} into \eqref{c4e5}, simplify and multiply (on the right) by $t_5t_6$ to obtain
\begin{align}
\text{ad}_{x_-}(e_3)t_5t_6=\nonumber\\
\sum_{(\epsilon_1, \epsilon_2,\underline{\upsilon})\in I}&a_{(\epsilon_1,\epsilon_2, \underline{\upsilon})}\left(g[i,j,\epsilon_2,l] e_{1}^{i}e_{2}^{j}e_3^{\epsilon_1+1}
e_4^{\epsilon_2}t_5^{k+1}t_6^{l+1}+q^{-3l}d_2[k]e_{1}^{i}e_{2}^{j}e_3^{\epsilon_1}
 e_4^{\epsilon_2+1}t_5^{k}t_6^{l+1}\right.  \nonumber\\
  +&\left.q^{-3(l-1)} d_3[l] e_{1}^{i}e_{2}^{j}e_3^{\epsilon_1}
 e_4^{\epsilon_2}t_5^{k+3}t_6^{l}-q^{-3l} d_2[i]e_{1}^{i-1}e_{2}^{j+1}e_3^{\epsilon_1}
 e_4^{\epsilon_2}t_5^{k+1}t_6^{l+1}\right), \label{c4e8}
\end{align} 
where $g[i,j,\epsilon_2,l]:=q^{-k-3\epsilon_2-3l}-q^{-i-3j-3l}.$ 
 
 Assume that there exists $l<0$ such that $a_{(\epsilon_1,\epsilon_2,\underline{v})}\neq 0.$ It follows from our initial assumption that
 $a_{(\epsilon_1,\epsilon_2,i,j,k,l_0)}\neq 0.$ Now, at $l=l_0,$ denote $\underline{\upsilon}=(i,j,k,l)$ by $\underline{\upsilon}_0:=(i,j,k,l_0).$   
  From \eqref{c4e8}, we have that
 \begin{align*}
 \text{ad}_{x_-}(e_3)t_5t_6=&\displaystyle\sum_{(\epsilon_1, \epsilon_2,\underline{\upsilon}_0)\in I}
 q^{-3(l_0-1)}
 a_{(\epsilon_1,\epsilon_2, \underline{\upsilon}_0)}d_3[l_0] e_{1}^{i}e_{2}^{j}e_3^{\epsilon_1}
 e_4^{\epsilon_2}t_5^{k+3}t_6^{l_0} +\mathcal{J}_1,
 \end{align*}
where $\mathcal{J}_1 \in \text{Span}\left( \mathcal{B}\setminus 
 \{e_1^{i}e_2^{j}e_3^{\epsilon_1}e_4^{\epsilon_2}
 t_5^{k}t_6^{l_0}\mid  \epsilon_1,\epsilon_2\in \{0,1\}, \ k\in\mathbb{Z} \ \text{and} \  i,j\in \mathbb{N}  \}\right) .$ 
 
 Set $\underline{w}:=(i,j,k,l)\in \mathbb{N}^4.$ One can also write $\text{ad}_{x_-}(e_3)t_5t_6\in A_{\alpha,\beta}$ in terms of the basis $\mathcal{B}$ of $A_{\alpha,\beta}$ (Proposition \ref{c3p5}) as:
\begin{align}
 \text{ad}_{x_-}(e_3)t_5t_6=&\sum_{(\epsilon_1,\epsilon_2, \underline{w})\in J}b_{(\epsilon_1,\epsilon_2,\underline{w})}e_1^{i}
e_2^{j}e_3^{\epsilon_1}e_4^{\epsilon_2}
t_5^{k}t_6^{l},\label{c4e10}
 \end{align}
 where $J$ is a finite subset of  $\{0,1\}^2\times \mathbb{N}^4$, and $b_{(\epsilon_1,\epsilon_2,\underline{w})}\in \k.$  It follows that
 $$\sum_{(\epsilon_1,\epsilon_2, \underline{w})\in J}b_{(\epsilon_1,\epsilon_2,\underline{w})}e_1^{i}
e_2^{j}e_3^{\epsilon_1}e_4^{\epsilon_2}
t_5^{k}t_6^{l}= \displaystyle\sum_{(\epsilon_1, \epsilon_2,\underline{\upsilon}_0)\in I}
 q^{-3(l_0-1)}
 a_{(\epsilon_1,\epsilon_2, \underline{\upsilon}_0)}d_3[l_0] e_{1}^{i}e_{2}^{j}e_3^{\epsilon_1}
 e_4^{\epsilon_2}t_5^{k+3}t_6^{l_0} +\mathcal{J}_1.$$
 Since $\mathcal{B}$ is a basis for $A_{\alpha,\beta},$ we deduce from Corollary \ref{c4c2} that $\left(e_1^{i}
e_2^{j}e_3^{\epsilon_1}e_4^{\epsilon_2}
t_5^{k}t_6^{l}\right)_{((\epsilon_1,\epsilon_2, \underline{v})\in \{0,1\}^2\times \mathbb{N}^2 \times \mathbb{Z}^2)}$ is also a basis for $A_{\alpha,\beta}[t_5^{-1}, t_6^{-1}].$ Since $\underline{v}_0=(i,j,k,l_0)\in \mathbb{N}^2\times \mathbb{Z}^2$ (with $l_0<0$) and $\underline{w}=(i,j,k,l)\in \mathbb{N}^4$ (with $l\geq 0$)  in the above equality,  we must have 
 $$q^{-3(l_0-1)}a_{(\epsilon_1,\epsilon_2,\underline{\upsilon}_0)}d_3[l_0]=0.$$
 Given that $q^{-3(l_0-1)}d_3[l_0]\neq 0,$ it follows that $a_{(\epsilon_1,\epsilon_2,\underline{\upsilon}_0)}=a_{(\epsilon_1,\epsilon_2,i,j,k,l_0)}$ are all zero. This is a contradiction. Therefore, $l\geq 0$ (i.e. there is no negative exponent for $t_6$).
 
Since   $l\geq 0,$ it follows from our initial assumption that there exists $k=k_0<0$ such that $a_{(\epsilon_1,\epsilon_2,i,j,k_0,l)}\neq 0.$  The rest of the proof will show that this assumption cannot also hold. 

Set $\underline{\upsilon}_0:=(i,j,k_0,l)\in \mathbb{N}^2\times \mathbb{Z}\times \mathbb{N}.$ From \eqref{c4e8}, we have that
 \begin{align*}
 \text{ad}_{x_-}(e_3)t_5t_6=& \displaystyle\sum_{(\epsilon_1,\epsilon_2, \underline{\upsilon}_0)\in I}q^{-3l} a_{(\epsilon_1,\epsilon_2, \underline{\upsilon}_0)}d_2[k_0]e_{1}^{i}e_{2}^{j}e_3^{\epsilon_1}
 e_4^{\epsilon_2+1}t_5^{k_0}t_6^{l+1}+V,
 \end{align*}
 where $V\in \mathcal{J}_2 := \text{Span}\left( \mathcal{B}\setminus 
 \{e_1^{i}e_2^{j}e_3^{\epsilon_1}e_4^{\epsilon_2}
 t_5^{k_0}t_6^{l}\mid \epsilon_1,\epsilon_2\in \{0,1\} \ \text{and} \ i,j,l\in \mathbb{N} \}\right).$ It follows that:
 \begin{align}
  \text{ad}_{x_-}&(e_3)t_5t_6=\nonumber\\
  &\sum_{(0,0,\underline{\upsilon}_0)\in I}q^{-3l} a_{(0,0,\underline{\upsilon}_0)}d_2[k_0]e_{1}^{i}e_{2}^{j}
e_4t_5^{k_0}t_6^{l+1} 
+\sum_{(1,0, \underline{\upsilon}_0)\in I} a_{(1,0,\underline{\upsilon}_0)}d_2[k_0]e_{1}^{i}e_{2}^{j}e_3
e_4t_5^{k_0}t_6^{l+1}\nonumber\\
 &  +\sum_{(0,1, \underline{\upsilon}_0)\in I}q^{-3l} a_{(0,1,\underline{\upsilon}_0)}d_2[k_0]e_{1}^{i}e_{2}^{j}
e_4^2t_5^{k_0}t_6^{l+1}\sum_{(1,1, \underline{\upsilon}_0)\in I} a_{(1,1,\underline{v}_0)}d_2[k_0]e_{1}^{i}e_{2}^{j}e_3
 e_4^2t_5^{k_0}t_6^{l+1}+V.\label{c4e11}
 \end{align}
 Write the relations in Lemma \ref{c3l2}(2),(4) as:
\begin{align}
e_4^2=&b_1\beta+b_2e_2e_4e_6
+b_4\alpha e_3e_6+b_6e_1e_3e_4e_6
+L_1,\label{c4e12}\\
e_3e_4^2=&
\beta b_1e_3+k_1e_2e_3
e_4e_6
+k_3\alpha^2 e_6+k_5\alpha e_1e_4
e_6
+k_{14}\beta e_1^2e_6
\nonumber\\&+k_{15}e_1^2
e_2e_4e_6^2
+k_{17}\alpha e_1^2e_3e_6^2+
k_{19} e_1^3e_3e_4e_6^2+
L_2,\label{c4e13}
\end{align}
where $L_1$ and $L_2$ are some elements of the left ideal  $A_{\alpha,\beta}t_5\subseteq \mathcal{J}_2.$ 
Substitute \eqref{c4e12} and \eqref{c4e13} into  \eqref{c4e11}, and simplify to obtain:

 \begin{align}
\text{ad}_{x_-}(e_3)t_5t_6= \sum&[\lambda_{1,1}\beta a_{(0,1,i,j,k_0,l-1)}+\lambda_{1,2}\alpha^2a_{(1,1, i,j,k_0,l-2)}\nonumber\\
&+\lambda_{1,3}\beta a_{(1,1, i-2,j,k_0,l-2)}]e_1^ie_2^jt_5^{k_0}t_6^l\nonumber \\
+&\sum[\lambda_{2,1}\alpha a_{(0,1, i,j,k_0,l-2)}  
+\lambda_{2,2}\beta a_{(1,1, i,j,k_0,l-1)}\nonumber\\&+\lambda_{2,3}\alpha a_{(1,1,i-2,j,k_0,l-3)}]e_1^ie_2^je_3t_5^{k_0}t_6^l\nonumber\\
 +&\sum[\lambda_{3,1}a_{(0,1, i,j-1,k_0,l-2)}  
+\lambda_{3,2}\alpha a_{(1,1, i-1,j,k_0,l-2)}\nonumber\\
                        &+\lambda_{3,3}a_{(1,1, i-2,j-1,k_0,l-3)}+\lambda_{3,4}a_{(0,0,i,j,k_0,l-1)}]e_1^ie_2^je_4t_5^{k_0}t_6^l\nonumber\\
+&\sum[\lambda_{4,1}a_{(0,1,i-1,j,k_0,l-2)}  
+\lambda_{4,2}a_{(1,1,i,j-1,k_0,l-2)}\nonumber\\
&+\lambda_{4,3}a_{(1,1,i-3,j,k_0,l-3)}
+\lambda_{4,4}a_{(1,0,i,j,k_0,l-1)}]e_1^ie_2^je_3e_4t_5^{k_0}t_6^l+V',
\label{c4e14}
 \end{align}
where $V'\in \mathcal{J}_2$.  Also, $\lambda_{s,t}:=\lambda_{s,t}(j,k_0,l)$ are some families of scalars which are non-zero for all $s,t\in\{1,2,3,4\}$ and $j,l\in\mathbb{N}$, except 
 $\lambda_{1,4}$ and $\lambda_{2,4}$ which are assumed to be zero since they do not exist in the above expression.  Note, although each  $\lambda_{s,t}$ depends on $j,k_0,l,$ we have not made this dependency explicit in the above expression since the minimum requirement we need to complete the proof is for all the $\lambda_{s,t}$ existing in the above expression to be non-zero, which we already have.
   
Observe that \eqref{c4e14} and \eqref{c4e10} are equal, hence,
\begin{align*}
\sum_{(\epsilon_1,\epsilon_2, \underline{w})\in J}b_{(\epsilon_1,\epsilon_2,\underline{w})}e_1^{i}
e_2^{j}e_3^{\epsilon_1}e_4^{\epsilon_2}
t_5^{k}t_6^{l} =& \sum[\lambda_{1,1}\beta a_{(0,1,i,j,k_0,l-1)}+\lambda_{1,2}\alpha^2a_{(1,1, i,j,k_0,l-2)}\nonumber\\
&+\lambda_{1,3}\beta a_{(1,1, i-2,j,k_0,l-2)}]e_1^ie_2^jt_5^{k_0}t_6^l\nonumber \\
+&\sum[\lambda_{2,1}\alpha a_{(0,1, i,j,k_0,l-2)}  
+\lambda_{2,2}\beta a_{(1,1, i,j,k_0,l-1)}\nonumber\\&+\lambda_{2,3}\alpha a_{(1,1,i-2,j,k_0,l-3)}]e_1^ie_2^je_3t_5^{k_0}t_6^l\nonumber\\
 +&\sum[\lambda_{3,1}a_{(0,1, i,j-1,k_0,l-2)}  
+\lambda_{3,2}\alpha a_{(1,1, i-1,j,k_0,l-2)}\nonumber\\
                        &+\lambda_{3,3}a_{(1,1, i-2,j-1,k_0,l-3)}+\lambda_{3,4}a_{(0,0,i,j,k_0,l-1)}]e_1^ie_2^je_4t_5^{k_0}t_6^l\nonumber\\
+\sum&[\lambda_{4,1}a_{(0,1,i-1,j,k_0,l-2)}  
+\lambda_{4,2}a_{(1,1,i,j-1,k_0,l-2)}\nonumber\\
+\lambda_{4,3}&a_{(1,1,i-3,j,k_0,l-3)}
+\lambda_{4,4}a_{(1,0,i,j,k_0,l-1)}]e_1^ie_2^je_3e_4t_5^{k_0}t_6^l+V'.
\end{align*}
We have previously established that $\left(e_1^{i}
e_2^{j}e_3^{\epsilon_1}e_4^{\epsilon_2}
t_5^{k}t_6^{l}\right)_{((\epsilon_1,\epsilon_2, \underline{v})\in \{0,1\}^2\times \mathbb{N}^2 \times \mathbb{Z}^2)}$ is a basis for $A_{\alpha,\beta}[t_5^{-1}, t_6^{-1}]$ (note, in this part of the proof $l\geq 0$). 
 Since $\underline{v}_0=(i,j,k_0,l)\in \mathbb{N}^2\times \mathbb{Z}\times \mathbb{N}$ (with $k_0<0$) and $\underline{w}=(i,j,k,l)\in \mathbb{N}^4$ (with $k\geq 0$)  in the above equality, it follows  that
\begin{align}
\lambda_{1,1}&\beta a_{(0,1, i,j,k_0,l-1)}+\lambda_{1,2}\alpha^2 a_{(1,1, i,j,k_0,l-2)}
+\lambda_{1,3}\beta a_{(1,1, i-2,j,k_0,l-2)}=0,\label{c4e15}\\
\lambda_{2,1}&\alpha a_{(0,1, i,j,k_0,l-2)}  
+\lambda_{2,2}\beta a_{(1,1, i,j,k_0,l-1)}+\lambda_{2,3}\alpha a_{(1,1, i-2,j,k_0,l-3)}
=0,\label{c4e16}\\
\lambda_{3,1}&a_{(0,1, i,j-1,k_0,l-2)}  
+\lambda_{3,2}\alpha a_{(1,1, i-1,j,k_0,l-2)}+\lambda_{3,3}a_{(1,1, i-2,j-1,k_0,l-3)}\nonumber\\&+\lambda_{3,4}a_{(0,0, i,j,k_0,l-1)}=0,\label{c4e17}\\
\lambda_{4,1}&a_{(0,1, i-1,j,k_0,l-2)}  
+\lambda_{4,2}a_{(1,1, i,j-1,k_0,l-2)}+\lambda_{4,3}a_{(1,1,i-3,j,k_0,l-3)}\nonumber\\
&+\lambda_{4,4}a_{(1,0, i,j,k_0,l-1)}=0.\label{c4e18}
\end{align}   
From \eqref{c4e15} and \eqref{c4e16}, one can easily deduce that
\begin{align}
&a_{(0,1, i,j,k_0,l)}=-\frac{\alpha^2\lambda_{1,2}}{\beta\lambda_{1,1}}a_{(1,1, i,j,k_0,l-1)}
-\frac{\lambda_{1,3}}{\lambda_{1,1}}a_{(1,1, i-2,j,k_0,l-1)},\label{c4e19}\\
&  
a_{(1,1, i,j,k_0,l)}=-\frac{\alpha\lambda_{2,1}}{\beta\lambda_{2,2}}a_{(0,1, i,j,k_0,l-1)}-\frac{\alpha\lambda_{2,3}}{\beta\lambda_{2,2}}a_{(1,1, i-2,j,k_0,l-2)}.\label{c4e20}
\end{align}
Note,  $a_{(\epsilon_1,\epsilon_2, i,j,k_0,l)}:=0$ whenever $i<0$ or $j<0$ or $l<0$ for all $\epsilon_1,\epsilon_2\in \{0,1\}.$ 

\textbf{Claim.} The coefficients
$a_{(0,1, i,j,k_0,l)}$ and $a_{(1,1, i,j,k_0,l)}$ are all zero for all $l\geq 0$. We now justify the claim by an induction on $l.$ From  \eqref{c4e19} and \eqref{c4e20}, the result is obviously true when $l=0.$ For $l\geq 0,$ assume that $a_{(0,1, i,j,k_0,l)}=a_{(1,1, i,j,k_0,l)}=0.$ Then, it follows from  \eqref{c4e19} and \eqref{c4e20} that
$$a_{(0,1, i,j,k_0,l+1)}=-\frac{\alpha^2\lambda_{1,2}}{\beta\lambda_{1,1}}a_{(1,1, i,j,k_0,l)}
-\frac{\lambda_{1,3}}{\lambda_{1,1}}a_{(1,1, i-2,j,k_0,l)},$$
$$a_{(1,1, i,j,k_0,l+1)}=-\frac{\alpha\lambda_{2,1}}{\beta\lambda_{2,2}}a_{(0,1,i,j,k_0,l)}-\frac{\alpha\lambda_{2,3}}{\beta\lambda_{2,2}}a_{(1,1, i-2,j,k_0,l-1)}.$$
From the inductive hypothesis, $a_{(1,1,i,j,k_0,l)}=a_{(1,1,i-2,j,k_0,l)}=a_{(0,1,i,j,k_0,l)}=a_{(1,1, i-2,j,k_0,l-1)}=0.$ 
Hence, $a_{(1,1, i,j,k_0,l+1)}=a_{(0,1,i,j,k_0,l+1)}=0.$  By the principle of mathematical induction,  $a_{(0,1, i,j,k_0,l)}=a_{(1,1,i,j,k_0,l)}=0$ for all $l\geq 0$ as desired.
Given that the families $a_{(0,1, i,j,k_0,l)}$ and $a_{(1,1, i,j,k_0,l)}$ are all zero, it follows from \eqref{c4e17} and \eqref{c4e18} that $a_{(0,0, i,j,k_0,l)}$ and 
$a_{(1,0, i,j,k_0,l)}$ are also zero for all $(i,j,k_0,l)\in \mathbb{N}^2\times\mathbb{Z}\times \mathbb{N}.$  Since $a_{(\epsilon_1,\epsilon_2, i,j,k_0,l)}$ are all zero, it contradicts
our assumption. Hence, $x_-=0.$  Consequently, $x=x_+\in A_{\alpha,\beta}$ as desired.

2. From Remark \ref{rrr}, we have
$e_4
=t_4-bt_5^3t_6^{-1}.$  Again, from Lemma \ref{ev5}, we have that $\lambda_4=3\lambda_3+\lambda_5$ and $\lambda_6=-3\lambda_3.$ Therefore,
\begin{align*}
\delta(e_4)
&=\lambda_4t_{4}-b(3\lambda_5-\lambda_6)t_5^3
t_6^{-1}\\&
=(3\lambda_3+\lambda_5)e_{4,6}-3b(\lambda_3+\lambda_5)t_5^3
t_6^{-1}\\&
=(3\lambda_3+\lambda_5)(e_4+bt_5^3t_6^{-1})-3b(\lambda_3+\lambda_5)t_5^3
t_6^{-1}\\
&=(3\lambda_3+\lambda_5)e_4-2b\lambda_5t_5^3t_6^{-1}.
\end{align*}
Moreover, 
$D(e_4)=\text{ad}_x(e_4)+\delta(e_4)
=\text{ad}_x(e_4)+(3\lambda_3+\lambda_5)e_4-2b\lambda_5t_5^3t_6^{-1}\in A_{\alpha,\beta}.$
It follows that
$b\lambda_5t_5^3t_6^{-1}\in A_{\alpha,\beta},$ since 
$\text{ad}_x(e_4)+(3\lambda_3+\lambda_5)e_4\in A_{\alpha,\beta}.$ Consequently, $
b\lambda_5t_5^3\in A_{\alpha,\beta}t_6.$ Since $b\neq 0$ (Appendix \ref{appenc}), we must have  $\lambda_5
=0,$ otherwise, there will be a contradiction using the basis of $A_{\alpha,\beta}$ (Proposition \ref{c3p5}).  Therefore, $\delta(e_4)
=3\lambda_3e_4$ and $\delta(t_5)
=0.$ We already know from Lemma \ref{ev5} that $\delta(t_6)
=-3\lambda_3t_6$. 
From \eqref{fe1}, we have  that
$\delta(e_3)=\lambda_3e_3+a(\lambda_3-\lambda_4+\lambda_5)e_4t_5^{-1}
+[s(\lambda_3-2\lambda_5+\lambda_6)+q^{-3}ab(\lambda_3-\lambda_4+\lambda_5)]t_5^2t_6^{-1}.$  Given that $\lambda_4=3\lambda_3, \ \lambda_5=0$ and $\lambda_6=-3\lambda_3,$ we have that $\delta(e_3)=\lambda_3e_3-2a\lambda_3e_4t_5^{-1}$ (note, from Appendix \ref{appenc}, one can confirm that $q^{-3}ab+s=0$).
Now, $D(e_3)=\text{ad}_x(e_3)+\delta(e_3)=\text{ad}_x(e_3)+\lambda_3e_3-2a\lambda_3e_4t_5^{-1}
\in A_{\alpha,\beta}.$ Observe that $\text{ad}_x(e_3)+\lambda_3e_3\in A_{\alpha,\beta}.$ Hence, $2a\lambda_3e_4t_5^{-1}
\in A_{\alpha,\beta}, \ \text{and so} \ 2a\lambda_3e_4
\in A_{\alpha,\beta}t_5.$ Since $a\neq 0,$ it implies that $\lambda_3=0,$ otherwise, there will be a contradiction using the basis of $A_{\alpha,\beta}.$ We now have that
$\delta(e_3)=\delta(e_4)=\delta(e_5)=\delta(e_6)=0.$ We finish the proof by showing that 
$\delta(e_1)=\delta(e_2)=0.$ Recall from \eqref{e2e} that
$$e_2e_4e_6+be_2e_5^3+be_3^3e_6+b'e_3^2e_5^2+c'e_3e_4e_5+d'e_4^2=\beta.$$ Apply $\delta$ to this relation to obtain $\delta(e_2)e_4e_6+b\delta(e_2)e_5^3=0.$ This implies that
$\delta(e_2)(e_4e_6+be_5^3)=0.$ Since $e_4e_6+be_5^3\neq 0,$ it follows that $\delta(e_2)=0.$ Similarly, from \eqref{e1e}, we have that
$$e_1e_3e_5+ae_1e_4+ae_2e_5+a'e_3^2=\alpha.$$ Apply $\delta$ to this relation to obtain
$\delta(e_1)(e_3e_5+ae_4)=0.$ Since $e_3e_5+ae_4\neq 0,$ we have that $\delta(e_1)=0.$
In conclusion, $\delta(e_\kappa)=0$ for all $\kappa\in \{1,\cdots, 6\}.$
 
3. As a result of (1) and (2), we have that $D(e_\kappa)=\text{ad}_x(e_\kappa).$ Therefore, $D=\text{ad}_x$ as desired.
\end{proof}


Using similar techniques, one can describe the derivations of $A_{\alpha,0}$ and $A_{0,\beta}$. 
Details can be found in \cite{io-thesis}. There are fundamental differences in these two cases. Indeed, there exist in both cases derivations which are not inner. More precisely, one can check that the linear map $\theta$ of $A_{\alpha,0}$ defined by
$$\theta(e_1)=-e_1, \ \ \theta(e_2)=-e_2, \ \ \theta(e_3)=0, \ \ \theta(e_4)=e_4, \ \ \theta(e_5)=e_5, \ \  \ \theta(e_6)=2e_6$$
is a  $\k-$derivation of $A_{\alpha,0}$.

Similarly, the linear map $\tilde{\theta}$ of $A_{0,\beta}$ by
$$\tilde{\theta}(e_1)=-2e_1, \ \ \tilde{\theta}(e_2)=-3e_2, \ \ \tilde{\theta}(e_3)=-e_3, \ \ \tilde{\theta}(e_4)=0, \ \ \tilde{\theta}(e_5)=e_5, \ \  \ \tilde{\theta}(e_6)=3e_6$$
is a $\k-$derivation of $A_{0,\beta}$.  
 
We summarize our main results in the theorem below.

 \begin{theorem}
Given that 
$A_{\alpha,\beta}=U_q^+(G_2)/\langle \Omega_1-\alpha,\Omega_2-\beta\rangle,$ with $(\alpha,\beta)\in \k^2\setminus \{(0,0)\},$  we have the following results:
\begin{itemize}
\item[1.]  if $\alpha, \beta\neq 0;$ then every derivation $D$ of $A_{\alpha,\beta}$ can uniquely be written as $D=\text{ad}_x,$ where $x\in A_{\alpha,\beta}.$
\item[2.]if $\alpha\neq 0$ and $\beta=0,$ then every derivation $D$ of $A_{\alpha,0}$ can uniquely be written as $D=\text{ad}_x+\lambda\theta,$ where $\lambda\in \k$ and $x\in A_{\alpha,0}.$
 \item[3.] if $\alpha=0$ and $\beta\neq 0,$ then every derivation $D$ of $A_{0,\beta}$ can uniquely be written as $D=\text{ad}_x+\lambda\tilde{\theta},$ where $\lambda\in \k$ and $x\in A_{0,\beta}.$
\item[4.] $HH^1(A_{\alpha,0})=\k[\theta]$ and $HH^1(A_{0,\beta})=\k[\tilde{\theta}],$  where
 $[\theta]$ and $[\tilde{\theta}]$ respectively denote the classes of $\theta$ and $\tilde{\theta}$  modulo the space of inner derivations. 
\item[5.] if $\alpha,\beta\neq 0;$ then $HH^1(A_{\alpha,\beta})=\{[0]\},$ 
 where $[0]$ denotes the class of $0$   modulo the space of inner derivations.
\end{itemize}
 \end{theorem}

The above theorem shows that $A_{\alpha,\beta}$ when both $\alpha$ and $\beta$ are nonzero shares a number of properties with the second Weyl algebra over $\k$: it is simple, units are reduced to scalars, and all derivations are inner. 

It would be interesting to compute the automorphism group of these algebras and verify if all endomorphisms are automorphisms, i.e. an analogue of the celebrated Dixmier Conjecture \cite{dj2}. 

In general, the present work and \cite{sl} suggest that the primitive quotients of $U_q^+(\mathfrak{g})$ by primitive ideals from the 0-stratum provide algebras that could (should?) be regarded (and studied) as quantum analogue of Weyl algebras.

\appendix

\section{Some general relations of $U_q^+(G_2)$} 
\label{sc}

We have the following selected general relations of $U_q^+(G_2).$

\begin{lemma}
\label{c3l3}
For any $n\in\mathbb{N}$, we have that: 

1(a) $E_jE_i^n=q^{-3n}E_i^{n}E_j$ \quad (b) $E_j^nE_i=q^{-3n}E_iE_j^n$ for all $1\leq i,j\leq 6,$ with 
$j-i=1.$\vspace{0.15in}

2(a) $E_6E_4^n=q^{-3n}E_4^nE_6+d_1[n]E_4^{n-1}
E_5^3$ \quad (b) $E_6^nE_4=q^{-3n}E_4E_6^n+d_1[n]
E_5^3E_6^{n-1}$

\ \ (c) $E_4E_2^n=q^{-3n}E_2^nE_4+d_1[n]E_2^{n-1}E_3^3$ \quad (d) $E_4^nE_2=q^{-3n}E_2E_4^n+
d_1[n]E_3^3E_4^{n-1},$

where $d_1[n]=q^{3(1-n)}d_1[1]\left( \dfrac{1-q^{-6n}}{1-q^{-6}}\right);$ $d_1[1]=-\dfrac{q^4-2q^2+1}{q^4+q^2+1}$ and $d_1[0]:=0.$
 \vspace{0.15in}
 
3(a) $E_3E_1^n=q^{-n}E_1^nE_3+d_2[n]
E_1^{n-1}E_2$ \quad (b) $E_3^nE_1=q^{-n}E_1E_3^n+d_2[n]E_2E_3^{n-1}$

\ \ (c) $E_5E_3^n=q^{-n}E_3^nE_5+d_2[n]
E_3^{n-1}E_4$ \quad (d) $E_5^nE_3=q^{-n}E_3E_5^n+d_2[n]E_4E_5^{n-1},$

where $d_2[n]=q^{1-n}d_2[1]\left( \dfrac{1-q^{-2n}}{1-q^{-2}}\right);$  $d_2[1]=-(q+q^{-1}+q^{-3})$ and $d_2[0]:=0.$ \vspace{0.15in}

4(a) $E_6^nE_3=E_3E_6^n+d_3[n]E_5^2E_6^{n-1}$ \quad (b) $E_5E_2^n=E_2^nE_5+d_3[n]E_2^{n-1}E_3^2,$

where $d_3[n]=d_3[1]\left( \dfrac{1-q^{-6n}}{1-q^{-6}}\right);$ $d_3[1]=1-q^2$ and
$d_3[0]:=0.$
\end{lemma}
\begin{proof} This is an easy proof by induction, left to the reader. 
\end{proof}

\section{Definition of parameters used throughout}
\label{appenc}
In this appendix, we define parameters used in this article. 
Note, for all $n\in \mathbb{N},$ we have already defined the scalars $d_2[n]$ in Lemma \ref{c3l3}, hence, we are not going to repeat them here. Any other scalars not defined here must be defined in/before the context in which it is found.  
\begin{align*}
a&=\dfrac{q^2+1+q^{-2}}{q^{-2}-1}& b&=-\dfrac{q^7-2q^5+q^3}{(q^4+q^2+1)(1-q^{-6})}\\
g&=\dfrac{q+q^{-1}+q^{-3}}{(1-q^{-2})^2}&f&=\dfrac{1-q^2}{1-q^{-2}}\\
h&=\dfrac{q+q^{-1}}{q^{-2}-1}&s&=\dfrac{1-q^2}{1-q^{-6}}\\
t&=\dfrac{q^{-1}-q}{1-q^{-6}}&u&=\frac{q+q^{-1}-q^{-3}}{1-q^{-6}}\\
p&=\frac{q^4+q^2+1}{q^2-1}&r&=\frac{-1}{1-q^{-6}}\\
e&=\dfrac{-(q^7+q^5+q^3)}{q^4-2q^2+1}&q''&=\dfrac{q^7-2q^5+q^3}{q^4+q^2+1}\\
n&=\dfrac{q^{12}}{(q^4+q^2+1)^3}&q'&=-(q^2+1+q^{-2})\\
k_1&=q^{-3}b_2+b_6d_2[1]&a'&=af+hq=\frac{q^6}{q^2-1}\\
k_2&=q^{-3}b_3+b_{12}d_2[1]&b'&=\frac{q^{13}-q^{11}}{(q^4+q^2+1)^2}\\
k_3&=b_4c_1  &d'&=\frac{q^{12}}{q^6-1}\\
k_4&=b_4c_2+q^{-3}b_5c_1+b_7d_2[1]  & c'&=-\frac{q^9}{q^4+q^2+1}\\
k_5&=b_4c_2+q^{-1}b_6c_1 &c_1&=\frac{1}{a'}\\
k_6&=c_3b_4+q^{-1}b_7+q^{-3}b_4b_{13}c_2& c_2&=-ac_1\\
k_7&=q^{-3}c_2b_5+b_8d_2[1]  & c_3&=-c_1\\
k_{8}&=b_1b_{13}c_2  &b_1&=\frac{1}{d'}\\
k_9&=q^{-4}b_6c_2+b_9d_2[2]+q^{-3}b_2c_2b_{13}+q^{-3}b_5c_2& b_2&=b_1bc_2(q+q^{-1}+q^{-3})-b_1\\
k_{10}&=q^{-1}b_6b_8c_2 &b_3&=-b'b_1c_2-bb_1\\
k_{11}&=b_{13}c_1  &
b_4&=-b_1bc_1\\
k_{12}&=q^{-1}b_6b_7c_2 &
b_5&=b_1b(c_3(q+q^{-1}+q^{-3})-q^{-3}c_2)\\
k_{13}&=q^{-4}b_6c_3+q^{-2}b_9+q^{-3}b_6b_{13}c_2+q^{-1}b_6b_{13}c_2&b_6&=-q^{-1}c_2b_1b\\
k_{14}&= q^{-1}b_1b_6c_2 &b_7&=-q^{-1}b_1bc_1c_3\\
k_{15}&=q^{-1}b_2b_6c_2 & b_8&=b_9=-q^{-1}b_1bc_2c_3\\
k_{16}&=q^{-1}b_3b_6c_2+q^{-2}b_{10}c_2+q^{-3}b_8b_{13}c_2 &b_{10}&=-q^{-1}c_3^2b_1b \\
k_{17}&=q^{-1}b_4b_6c_2   &b_{11}&=-b'b_1c_1\\
k_{18}&= q^{-1}b_5b_6c_2 &b_{12}&=-b'b_1c_3\\
k_{19}&=q^{-1}b_6^2c_2   &b_{13}&=-b_1c'\\
k_{20}&=q^{-1}b_6b_9c_2 &b_{14}&=-b'b_1c_2\\
k_{21}&=q^{-1}b_6b_{11}c_2+q^{-2}b_{10}c_1+q^{-3}b_7b_{13}c_2 &b_{15}&=q^{-3}c_3+c_2b_{13}\\
k_{23}&=q^{-1}b_6b_{12}c_2+q^{-2}b_{10}c_3+q^{-3}b_{10}b_{13}c_2&k_{22}&=q^{-1}b_6b_{10}c_2\\
k_{25}&=q^{-1}b_{12}c_1+b_{11}b_{13}c_2 &k_{24}&=q^{-1}b_6b_{14}c_2+q^{-2}b_{10}c_2+q^{-3}b_9b_{13}c_2\\
k_{27}&=q^{-1}b_{12}c_3+b_{12}b_{13}c_2&k_{26}&=q^{-1}b_{12}c_2+b_{13}b_{14}c_2\\
k_{29}&=b_{13}b_{15}+q^{-1}b_{14}&k_{28}&=q^{-3}b_{13}c_2+b_{14}d_2[1]\\
k_{31}&=q^{-3}b_5b_{13}c_2+q^{-3}b_5c_3+q^{-4}b_8+b_{10}d_2[2]&k_{30}&=b_3b_{13}c_2+q^{-1}b_{12}c_2\\
\end{align*}


\begin{minipage}{\textwidth}
\noindent S Launois \\
School of Mathematics, Statistics and Actuarial Science,\\
University of Kent\\
Canterbury, Kent, CT2 7FS,\\ UK\\[0.5ex]
email: S.Launois@kent.ac.uk \\

\noindent I Oppong\\
School of Mathematics, Statistics and Actuarial Science,\\
University of Kent\\
Canterbury, Kent, CT2 7FS,\\ UK\\[0.5ex]
email: isaac.oppong@aims.ac.rw \\

\end{minipage}

\end{document}